%% file: main.tex
\documentclass[11pt,smallheadings]{scrartcl}
\usepackage{amsmath,amssymb,latexsym, amsthm,dsfont,bm}
\usepackage[utf8]{inputenc}
\usepackage{mathtools}
\usepackage{tikz-cd}
\usepackage{enumitem}
\usepackage{float}
\usepackage{titlesec}
\usepackage{fancyhdr}
\usepackage{hyperref}

\usepackage[titles]{tocloft}
\fancypagestyle{specialfooter}{%
	\fancyhf{}

	\fancyfoot[L]{\footnotesize The author is funded by the Deutsche Forschungsgemeinschaft (DFG,
		German Research Foundation) under Germany's Excellence Strategy – EXC 2044 – 390685587,
		Mathematics Münster: Dynamics – Geometry - Structure}
}

\author{Dennis Wulle}
\title{Cohomogeneity one manifolds with quasipositive curvature}

\newtheorem{thm}{Theorem}[section]
\newtheorem{lem}{Lemma}[section]
\newtheorem*{lem*}{Lemma}

\newtheorem{dfn}{Defninition}[section]
\newtheorem*{dfn*}{Definition}
\newtheorem{prop}{Proposition}[section]

\newtheorem{cor}{Corollary}[section]
\newtheorem{theorem}{Theorem}

\newtheorem*{theorem*}{Theorem}
\newtheorem*{rem*}{Remark}

\newcommand{\bb}{\mathbb}
\newcommand{\fr}{\mathfrak}
\newcommand{\Ad}{\text{Ad}}
\newcommand{\proj}{\text{pr}}
\newcommand{\Sp}{\text{Sp}}
\newcommand{\SO}{\text{SO}}
\newcommand{\SU}{\text{SU}}
\renewcommand{\O}{\text{O}}
\newcommand{\U}{\text{U}}
\newcommand{\diag}{\text{diag}}
\newcommand{\rk}{\text{rk}\,}
\newcommand{\Spin}{\text{Spin}}
\newcommand{\Id}{\text{Id}}

\newcommand{\norm}[1]{\left\lVert#1\right\rVert}

\newcommand{\RN}[1]{%
	\textup{\uppercase\expandafter{\romannumeral#1}}%
}
\newcommand{\II}{\RN{2}}
\newcommand{\Hess}{\text{Hess}\,}
\newcommand{\ad}{\text{ad}}

\titleformat{\section}
{\centering\normalfont\normalsize\bfseries}{\thesection.}{.5em}{}
\titleformat{\subsection}
{\centering\normalfont\normalsize}{\thesubsection.}{.5em}{}

\begin{document}
	\maketitle
	
	\begin{abstract}
		In this paper we give a classification of cohomogeneity one manifolds admitting an invariant metric with quasipositive sectional curvature except for the two $7$-dimensional families $P_k$ and $Q_k$, $k \ge 1$, that were described in \cite{gwz}. The main result carries over almost verbatim from the classification results in positive curvature carried out by Verdiani and Grove, Wilking and Ziller (\cite{v1,v2,gwz}). Three main tools used in the positively curved case that we generalized to quasipositively curved cohomogeneity one manifolds are Wilking's Chain Theorem (\cite{w2}), the classification of positively curved fixed point homogeneous manifolds by Grove and Searle (\cite{gs}) and the Rank Lemma.
	\end{abstract}
	
	\input{documents/Introduction}

	\tableofcontents
	\input{documents/cohom1}
	\input{documents/obstructions}
	\input{documents/weylgroup}

	%\input{documents/essential}
	\input{documents/block}
	\input{documents/even}
	\input{documents/odd}
	\newpage
	\input{documents/appendix}

	\newpage
	\bibliographystyle{alpha}
	\bibliography{documents/references}
	
\end{document}

%% file: documents/Introduction.tex
 The study of compact manifolds with lower sectional curvature bounds has always been one of the classical themes in Riemannian geometry. In particular non-negatively and positively curved manifolds are of special interest. In the first class a lot of examples are known. For instance each quotient space of a compact Lie group by a closed subgroup admits a metric with non-negative curvature. Compared to that there are only a few examples of positively curved manifolds. Indeed,  by the best knowledge of the author, all known examples in dimensions greater than $24$ are the compact rank one symmetric spaces and apart from dimensions $7$ and $13$ only finitely many are known in each dimension. In \cite{w0} Wilking suggested, that it might be fruitful to have a class in between non-negatively and positively curved manifolds. There are two candidates for this class: manifolds with quasipositive sectional curvature, i.e. the manifold is non-negatively curved and there exists one point, such that all sectional curvatures at that point are positive, and manifolds with positive sectional curvature on an open and dense set. Of course it is a natural question to ask, if any of these classes coincide. There are only a few obstructions known distinguishing between these classes, such as Synge's Lemma, which states that a compact positively curved manifold is simply connected, if it is orientable and even dimensional, and orientable, if it is odd dimensional. Unfortunately, once the manifold is assumed to be simply connected, there are no obstructions known to distinguish between these classes.

It turned out to be a fruitful approach to study compact manifolds with lower curvature bounds in presence of an isometric action by a compact Lie group. There are several ways to measure the size of the group action. We will focus on the \emph{cohomogeneity} of the action, i.e. the dimension of the orbit space. The study of positively curved manifolds with symmetry has lead to a variety of results. For example Wilking showed, that in large dimensions (compared to the cohomogeneity) manifolds with invariant metrics of positive curvature are homotopy equivalent to a rank one symmetric space (cf. \cite{w2}, Corollary 3). Furthermore the classification of positively curved cohomogeneity one manifolds has been carried out by Verdiani in even dimensions \cite{v1,v2} and by Grove, Wilking and Ziller in odd dimensions up to two families in dimension seven \cite{gwz}. These families have cohomogeneity one actions by $S^3 \times S^3$ and are further described in Table \ref{appendix_table_odd_cohom_one}. There also have been results on positively curved manifolds almost everywhere. In particular Wilking constructed a series of these manifolds in cohomogeneity two:

\begin{theorem*}[Wilking, \cite{w0} Theorem 1, Proposition 6.1.]
	Each of  the projective tangent bundles  \emph{$P_{\bb R} T \bb R \bb P^n$}, \emph{$P_{\bb C} T \bb C \bb P^n$} and \emph{$P_{\bb H} T \bb H \bb P^n$} of $\bb R \bb P^n$, $\bb C \bb P^n$ and $\bb H \bb P^n$ admits a Riemannian metric with positive sectional curvature on an open and dense set. These metrics are invariant by cohomogeneity two actions of \emph{$\O(n)$}, \emph{$\U(n)$} and \emph{$\Sp(n)$}, respectively. Furthermore the natural inclusions are totally geodesic.
\end{theorem*}

$P_{\bb R} T \bb R \bb P^{2n+1}$ cannot have a metric with positive sectional curvature by Synge's Lemma, since it is not orientable and odd dimensional. Note, that even in the simply connected case a series like this with positively curved metrics and increasing dimension cannot exist in any fixed cohomogeneity, unless the manifolds are homotopy equivalent to rank one symmetric spaces,  by the already mentioned result of Wilking. It is therefore only natural to ask, whether a series like this can exist in cohomogeneity one in the presence of a metric with positive sectional curvature almost everywhere or with quasipositive curvature.  We can answer this question negatively and give a classification of cohomogeneity one manifolds with quasipositive curvature except for dimension seven.

Our main result shows that the already mentioned classification result (\cite{v1,v2, gwz}) carries over nearly verbatim, if we relax the condition of positive sectional curvature to quasipositive curvature. 
\begin{theorem}\label{main}
	Let $M$ be a simply connected compact cohomogeneity one manifold. If $M$ admits an invariant metric with quasipositive sectional curvature, then one of the following holds:
	\begin{enumerate}[ topsep=0pt,itemsep=-1ex,partopsep=1ex,parsep=1ex]
		\item The dimension of $M$ is even and $M$ is equivariantly diffeomorphic to a rank one symmetric space with a linear action.
		\item The dimension of $M$ is odd and $M$ is  equivariantly diffeomorphic to one of the following spaces:
		\begin{enumerate}[ topsep=0pt,itemsep=-1ex,partopsep=1ex,parsep=1ex]
			\item A sphere with a linear action.
			\item One of the Eschenburg spaces $E_p^7$, one of the Bazaikin spaces $B_p^{13}$ for $p \ge 0$ or $B^7$.
			\item One of the $7$-manifolds $P_k$ or $Q_k$ for $k \ge 1$.
		\end{enumerate}
	\end{enumerate}
\end{theorem}
In the odd dimensional case the cohomogeneity one actions in the non spherical examples can be found in Table \ref{appendix_table_odd_cohom_one}. Among these examples $E_p^7$, $B_p^{13}$ for $p \ge 1$, $B^7$, $P_1$ and $Q_1$ admit an invariant metric with positive curvature, since $P_1 = \bb S^7$ and $Q_1$ is equivariantly diffeomorphic to a positively curved Aloff-Wallach space. Kerin showed in \cite{k}, that the non-negatively curved Eschenburg and Bazaikin spaces $E_0^7$ and $B_0^{13}$ admit metrics with positive sectional curvature almost everywhere. Those are the only two examples that do not occur in the classification result of \cite{gwz}. Note, that the result in \cite{gwz} contains an additional seven dimensional manifold $R$. Verdiani and Ziller showed in \cite{vz}, that this manifold cannot admit a metric with positive sectional curvature. The same argument also shows, that $R$ cannot admit a metric with quasipositive curvature.\\

 There are three main difficulties in carrying over the tools from \cite{gwz}. The first one is that we cannot prove directly that there is an isotropy group whose corank is at most one. This makes the classification of cohomogeneity one manifolds with quasipositive curvature and trivial principal isotropy group more complicated. The second essential tool is just missing. We cannot rule out a cohomogeneity one manifold solely by the fact that there are zero curvatures at a singular orbit. The third is that the Chain Theorem needs a completely different proof. A version of this theorem for cohomogeneity one actions on quasipositively curved manifolds, is stated below as the Block Theorem.
\\

We continue with a short overview on the tools used to prove Theorem \ref{main}. The basic method is, of course, to compute the possible group diagrams induced by the cohomogeneity one action in presence of a quasipositively curved metric. Since there exist a great variety of these, we need certain \emph{recognition tools} to narrow down the possible cases. For this a first step is the complete classification of  low cohomogeneity group actions on spheres by Straume \cite{st}. In cohomogeneity one the classification contains a family of spheres with non-linear actions and some of these are exotic Kevaire spheres. As pointed out in \cite{gwz}, it was observed by Back and Hsiang \cite{bh} for dimensions greater than $5$ and by Searle \cite{s} in dimension $5$ that these families cannot admit metrics with positive sectional curvature. In \cite{gvwz} it was shown that these families even do not  admit metrics with non negative curvature in dimensions greater than $5$. In Section \ref{secobstructions} we will use a new obstruction to show that in dimension $5$ these manifolds also do not carry an invariant metric with quasipositive curvature. In particular this will help, since we only have to recognize a sphere up to homotopy. Similar conclusions hold for the other rank one symmetric spaces. There are two main tools we use to determine the homotopy type. The first one is the classification of positively curved fixed point homogeneous manifolds by Grove and Searle \cite{gs}. The proof relies on the fact that in such manifolds there is exactly one orbit at maximal distance to the fixed point set. This follows from the Alexandrov geometry of the orbit space, since the distance function from its boundary is strictly concave (cf. \cite{p} Theorem 6.1). In the case of quasipositive curvature this proof does not carry over right away. In Section \ref{secobstructions} we will show that the same classification can still be carried out in quasipositive curvature provided there is a cohomogeneity one action containing a fixed point homogeneous subaction. We have the following result.
\begin{theorem} \label{intro_fp_hom}
	Let $M$ be a compact quasipositively curved cohomogeneity one $G$-manifold, where $G$ is a compact Lie group. Assume, there is a subgroup $G^\prime\subset G$ that acts fixed point homogeneous on $M$. Then $M$ is $G$-equivariantly diffeomorphic to a rank one symmetric space with a linear action.
\end{theorem}
The second tool is Wilking's Connectedness Lemma (see Lemma \ref{obstructions_thm_connect}), which also holds in many situations in quasipositive curvature.  \\
In the presence of positive curvature there is a third tool used in \cite{gwz}, which is Wilking's Chain Theorem (\cite{w2}, Theorem 5.1) that states that a simply connected positively curved manifold with an isometric action by one of the classical Lie groups $\SO(d)$, $\SU(d)$ or $\Sp(d)$ is homotopy equivalent to a rank one symmetric space, provided that the principal isotropy group contains a $k \times k$-block with $k\ge 3$, or $k \ge 2$, if $\Sp(d)$ acts. We were able to generalize this result to quasipositive curvature, provided that there is a cohomogeneity one action on the manifold. 
\begin{theorem}[Block Theorem]\label{intro_block}
	Let $M$ be a simply connected Riemannian manifold with quasipositive sectional curvature, such that $G= L \times G_d$ acts isometrically on $M$ with cohomogeneity one, where $L$ is a compact connected Lie group and \emph{$(G_d,u) \in \lbrace (\text{SO}(d),1),$ $(\text{SU}(d),2),$ $(\text{Sp}(d),4) \rbrace$}. Assume that the principal isotropy group contains up to conjugacy a lower $k\times k$-block, with $k \ge 3$, if $u = 1,2$, and $k\ge 2$, if $u = 4$. Then $M$ is equivariantly diffeomorphic to a rank one symmetric space with a linear action.
\end{theorem}
Another very important variable for a cohomogeneity one $G$-manifold $M$ is the rank of $G$. In positive curvature it is a simple fact, that the corank of the principal isotropy group in $G$ is equal to $1$, if $M$ is even dimensional, and $0$ or $2$, if $M$ is odd dimensional. In Section \ref{secweyl} we will prove, that this basic fact also holds in the presence of an invariant metric with quasipositive curvature. The original proof does not carry over and our proof of this result requires knowledge on the Weyl group, i.e. the stabilizer group of a horizontal geodesic.\\

The structure of this paper is very similar to \cite{gwz}. One reason for this is that, once the main tools from the positively curved case are carried over to the quasipositively curved case, some parts of the classification work in the same or a very similar way as in the case of positive curvature. In this way it will be easier to compare both papers. Furthermore, this will be helpful to point out different approaches compared to the positively curved case. The paper is organised as follows. We start with a short section about cohomogeneity one manifolds without any curvature assumptions. Since \cite{gwz} Section 1 covers the same material and is already quite detailed, we recommend to also read this. 
%Note, that the reader should take a look at the proof of Lemma \ref{cohom1_lem_connected}, since it will be needed in section \ref{secobstructions}, which contains a collection of useful obstructions to quasipositively curved metrics that will be used throughout the paper.
In Section \ref{secobstructions} we proceed with a collection of obstructions to quasipositive curvature used throughout the paper, including the proof of Theorem \ref{intro_fp_hom}.
We will continue with the proof of the Rank Lemma in Section \ref{secweyl} and go on with the proof of Theorem \ref{intro_block} in Section \ref{secblock}. Afterwards we will continue with the classification in Sections \ref{seceven} and \ref{secodd}. In the even dimensional case, the main difference is the case, where the group acting is semisimple but not simple. The simple cases are mostly handled as in \cite{gwz}, section 14. In the odd dimensional case, the main difference will also be the non simple cases as well as the simple cases in low dimensions, which require mostly new proofs. We recommend to compare our proofs to \cite{gwz}. Note that most of the results of Sections \ref{secweyl}, \ref{seceven} and \ref{secodd} are taken over (almost) verbatim from \cite{gwz} with the relaxed curvature condition to make the comparison easier, although new proofs are required for most key statements.\\

This paper is part of the authors PhD project. The results in Section \ref{secweyl} were part of the authors unpublished master thesis.\\

\emph{Acknowledgements:} I want to thank Burkhard Wilking for introducing this problem to me and for the useful discussions and comments while working on this project. Furthermore, I want to thank Kevin Poljsak and Jan Nienhaus for useful comments and discussions. I especially thank Linus Kramer for his comments regarding the results in Section \ref{secweyl}, during the work on my master thesis.

%% file: documents/cohom1.tex
\section{Cohomogeneity one manifolds} \label{seccohom1}

We will discuss a few general facts about cohomogeneity one manifolds and fix some notation. If the reader already has a good working knowledge one can also proceed with the following chapters. We give a short version of Section 1 of \cite{gwz}, since this section is quite detailed. For further details see \cite{aa} and \cite{gwz}. We are interested in compact simply connected cohomogeneity one $G$-manifolds $M$ with quasipositive curvature and $G$ connected. It will be convenient to understand the more general case, where $M$ has finite fundamental group and $G$ is not connected. Note, that a compact quasipositively curved manifold always has finite fundamental group by the Cheeger-Gromoll Splitting Theorem. 

Let $M$ be cohomogeneity one $G$-manifold with finite fundamental group. In this case $M/G$ is diffeomorphic to an interval, whose 
boundary points correspond to the two singular orbits and all points in the interior to principal orbits. By scaling the metric we can assume, that $M/G \cong \lbrack -1,+1\rbrack$. Let $c \colon \bb R \to M$ denote a geodesic perpendicular to all orbits, which is an infinite horizontal lift of $M/G$ by the quotient map $\pi \colon M \to M/G$. By $H$ we denote the \emph{principal} isotropy group $G_{c(0)}$ and by $K^\pm$ the isotropy groups at $p_\pm = c(\pm 1)$. Note that for all $t \neq 1 \mod 2\bb Z$, we have $G_{c(t)} = H$. By the Slice Theorem $M$ is now the union of tubular neighbourhoods of the two non principal orbits $B_\pm = G\cdot p_\pm = G/K^\pm$ glued along their common boundary $G/H$, i.e.
\begin{align*}
	M = G \times_{K^-}\bb D_- \cup G \times_{K^+} \bb D_+
\end{align*} 
Here $\bb D_\pm$ denotes the normal disc to the orbit $B_{\pm}$ at $p_\pm$. Furthermore, $K^\pm/H = \partial \bb D_\pm = \bb S^{l_\pm}$ are the spheres normal to the singular orbits. $M$ is now determined by the \emph{group diagram} $H \subset \lbrace K^-,K^+\rbrace \subset G$. Conversely a group diagram with $K^\pm /H = \bb S^{l\pm}$ defines a cohomogeneity one $G$-manifold. \\
The spheres $K^\pm/H$ are in general highly ineffective and we denote their ineffective kernels by $H_\pm$. It will be convenient to consider almost effective $G$-actions instead of effective ones. \\
A non principal orbit $G/K$ is called \emph{exceptional}, if $K/H = \bb S^0$, and  \emph{singular} otherwise. We denote the collection of principal orbits by $M_0 = M\backslash (B_-\cup B_+)$.\\
The cohomogeneity one \emph{Weyl group} $W(G,M) =W$ is defined to be the stabilizer of the geodesic $c$ modulo its kernel $H$. Let $N(H)$ be the normalizer of $H$ in $G$, then $W$ is a dihedral subgroup of $N(H)/H$, generated by two unique involutions $w_\pm \in (N(H)\cap K^\pm)/H$, since $K^\pm/H$ are both spheres (cf. \cite{aa}). Furthermore $M/G = c/W$. Each of these involutions can be described as the unique element $a \in K^\pm$ modulo $H$, such that $a^2$ but not $a$ lies in H. Any non principal isotropy group is of the form $wK^\pm w$ for some $w \in N(H)$ representing an element in $W$. The isotropy groups alternate along $c$ and hence half of them are isomorphic to $K^-$ and half of them to $K^+$, if $W$ is finite. In the presence of a metric with quasipositive sectional curvature $W$ is indeed finite (\cite{gwz}, Lemma 3.1).\\
The following lemma can easily be obtained, from the fact, that $M$ is a double disc bundle. We will still give a proof of this using Morse theory of path spaces, since this will lead to an obstruction to quasipositive curvature in Section \ref{secobstructions}. 
\begin{lem}\label{cohom1_lem_connected}
	The inclusion $G/K^\pm \cong B_\pm \to M$ is $l_\mp$-connected and the inclusion $G/H \to M$ is $\min\lbrace l_-,l_+\rbrace$-connected.
\end{lem}
\begin{proof}
	We remember, that by definition a continuous map $f \colon X \to Y$ is \emph{$l$-connected}, if the map $f_i \colon \pi_i(X) \to \pi_i(Y)$ between homotopy groups is an isomorphism for $i < l$ and surjective for $i=l$. 
	
	Let $B=B_-$ be a non principal orbit. Consider the energy functional 
	\begin{align*}
		E: c \mapsto \frac{1}{2}\int_0^1 \norm{\dot c}^2dt
	\end{align*}
	on the space $\Omega_B(M)$ of all piecewise smooth curves starting and ending in $B$. We can embed $B$ into $\Omega_B(M)$ as the point curves. It is now possible to determine the topology of this space using Morse theory as in the proof of the Connectedness Lemma (cf. \cite{w1}). 
	We claim that the inclusion $B \to \Omega_B(M)$ is $(l_+-1)$-connected:
	The critical points of the energy functional are the geodesics starting and ending perpendicularly to $B$. We want to show that the index of such a geodesic $c$, which is the dimension of the space on which $\Hess E$ is negative definite, is at least $l_+$. We consider the Hessian of the energy functional along $c$ up to a time $t$ and denote it by $\Hess E_t$. As long as $c$ is minimal the index is $0$ and raises by the dimension of the nullspace of $\Hess E_t$, which is also called nullity, after $c$ encountered a focal point (cf. \cite{m}). To compute the nullity of $\Hess E_t$, we have to consider the space $W$ of all Jacobi fields orthogonal to $c$, that start and end tangent to $B$, such that $J^\prime(0) + S_{\dot c(0)}(J(0))$ is orthogonal to $B$. Here $S_{\dot c(0)}$ is the shape operator of $B$ given by the second fundamental form in direction $\dot c$. It is easy to see that this space has dimension $n-1$. Now let $H$ be the principal isotropy group of the action along $c$. Then $\fr g = \fr h \oplus \fr m$, where $\fr m$ is the orthogonal complement of the Lie algebra $\fr h$ of $H$.  The space $\fr m$ is $(n-1)$-dimensional and induces a family of Jacobi fields in the following way. For $X \in \fr m$, we define $X^*(p) \coloneqq d/dt\vert_{t=0}\exp(tX)\cdot p$. Since these fields are induced by isometries, they are Killing fields and hence Jacobi fields along $c$, fulfilling the desired conditions at $B$. Therefore $W = \lbrace X^* \vert X \in \fr m \rbrace$. Since the geodesic is minimal until it reaches the second non principal orbit we have to compute the nullity at that point. Let $\fr k^\pm$ be the Lie algebra of $K^\pm$. Then all action fields $X^*$ with $X \in \fr m \cap \fr k^+$ vanish at the second non principal orbit. This is a vector space of dimension $l_+$. Since the next focal point lies in $B$, we have that each geodesic has index at least $l_+$ (and exactly $l_+$ if it intersect $B_+$ only once). Therefore $\Omega_B(M)$ can be obtained from $B$ by attaching cells of dimension at least $l_+$. Hence $B \to \Omega_B(M)$ is $(l_+-1)$-connected. \\
	It is well known that $\pi_i(M,B) = \pi_{i-1}(\Omega_B(M),B)$ (cf. \cite{w1}). Therefore $\pi_i(M,B) = 0$ for $i = 1, \ldots, l_+$. By the long exact sequence for homotopy groups $B \to M$ is $l_+$-connected. Hence the first part of the lemma is proven. The second part can be proven in the same way, by noting that a geodesic orthogonal to $G/H$ either intersects $B_+$ first and then $G/H$ again or $B_-$ first and then $G/H$ again.
\end{proof}

	We now proceed with some conclusions from the above lemma (cf. \cite{gwz}). 
%	Note, that not both $l_\pm =0$. If both normal bundles to $G/K^\pm$ are trivial, then $M$ is an $\bb S^1$-bundle. If one of $B_\pm$, say $B_+$ has non trivial normal bundle, the twofold cover $G/H \to G/K^+$ gives rise to a twofold cover $M^\prime \to M$ with group diagram $H \subset \lbrace K^-, w^+K^-w^+\rbrace \subset G$. Now we are either in the first situation or can repeat the argument indefinitely, contradicting, that $\pi_1(M)$ is finite. If both $l_\pm>0$, then $G/H$ is connected and hence $G$ and $G_0$ have the same orbits and the same Weyl group. If $l_- =0$ and $l_+>0$, then $G/K^-$ is connected. Since $G/H$ is then a sphere bundle over $B_-$, the principal orbit $G/H$ has at most two components, which implies (\cite{gwz}, (1.5)):
	\begin{lem}[\cite{gwz}, (1.5)]\label{cohom1_lem_weylindex}
		$M$ has at most one exceptional orbit and the Weyl group of $G_0$ has at most index $2$ in the Weyl group of $G$.
	\end{lem}
%	We proceed by assuming that $M$ is simply connected. Then the above covering argument implies, that there are no exceptional orbits. If both $l_\pm\ge 2$, then all orbits are simply connected and all isotropy groups are connected. If $l_-=1$ and $l_+ \ge 2$, then $K^-$ is connected and $G/H$ is a circle bundle over $G/K^-$, which implies, that $\pi_1(G/H)$ is cyclic and hence also $H/H_0 \cong K/K_0$. This implies
	\begin{lem}[\cite{gwz}, Lemma 1.6]
		Assume $G$ acts by cohomogeneity one with $M$ simply connected and $G$ connected. Then:
		\begin{enumerate}[label = (\alph*), topsep=0pt,itemsep=-1ex,partopsep=1ex,parsep=1ex]
			\item  There are no exceptional orbits.
			\item  If both $l_\pm\ge2$, then $K^\pm$ and $H$ are connected.
			\item If $l_- =1$, and $l_+\ge 2$, then $K^- = H \cdot S^1=H_0\cdot S^1$, $H=H_0\cdot \bb Z_k$ and $K^+ = K_0^+\cdot \bb Z_k$.
		\end{enumerate}
	\end{lem}
	Coverings of cohomogeneity one manifolds can be obtained by adding components to the isotropy groups in a suitable way (see \cite{gwz}, Lemma 1.7) or if both non principal isotropy groups are conjugate to each other by an order two element in $N(H)/H$.\\
	
	We will now take at look at fixed point components of subactions. Let $L \subset K^\pm$ be a subgroup, that is not conjugated to any subgroup of $H$. Then no component of the fixed point set $M^L$ intersects the regular part. Hence all components of $M^L$ are contained in some singular orbit and are therefore homogeneous. The component containing $p_\pm$ is equal to $N(L)_0/(N(L)_0 \cap K^\pm)$. This is in particular useful, if $K^\pm$ contains a central involution of $G$, that is not contained in $H$, since then $G/K^\pm$ is a fixed point component of $\iota$.
	
	Let $L$ be conjugate to a subgroup of $H$. Then the component of $M^L$ intersecting the regular part of $M$ is a cohomogeneity one manifold by the action of $N(L)$. We will consider the fixed point component $M^L_c$ of $L$, which contains the geodesic $c$, together with its stabilizer subgroup $N(L)_c \subset N(L)$. In general the length of $M^L/N(L)_c$ is a multiple of the length of $c/W$. They coincide, if both $N(L) \cap K^\pm$ act non trivially on the normal spheres of $M^L_c \cap K^\pm$ at $p_\pm$. In this case the group diagram of $M^L_c$ is given by $N(L)_c \cap H \subset \lbrace N(L)_c \cap K^-, N(L)_c \cap K^+ \rbrace \subset N(L)_c $. The Weyl groups of $(M,G)$ and $(M^L_c, N(L)_c)$ might be different. By \cite{gwz} Lemma 1.8, the Weyl groups coincide, if $L$ is equal to $H$ or to a maximal torus of $H_0$.
%	\begin{lem} [Reduction Lemma, Lemma 1.8 \cite{gwz}]
%		If $L$ is either equal to $H$ or given by a maximal torus of $H_0$, then $N(L)_c/L$ acts by cohomogeneity one on $M_c^L$ and the corresponding Weyl groups coincide. 
%	\end{lem}	
		In the case $L = H$ we call $M_c^H$, the \emph{core} of $M$ and $N(H)_c$ the \emph{core group}. In Section \ref{secweyl}, we will determine, which possible core groups can occur, if the curvature of $M$ is quasipositive.\\

	We conclude this section with the equivalence of group diagrams. Two group diagrams $H\subset \lbrace K^-,K^+\rbrace \subset G$ and $\bar H \subset \lbrace \bar K^-, \bar K^+ \rbrace \subset G$ describe  the same cohomogeneity one manifold up to $G$-equivariant diffeomorphism, if and only if after possibly switching $K^-$ and $K^+$, the following holds: There exist $b \in G$ and $a \in N(H)_0$, such that $K^- = b\bar K^-b^{-1}$, $H = b \bar H b^{-1}$ and $K^+ = ab \bar K^+ b^{-1}a^{-1}$ (cf. \cite{aa}, \cite{gwz}).%

%% file: documents/obstructions.tex
\section{Obstructions to quasipositive curvature}\label{secobstructions}

In this section we will discuss some obstructions for cohomogeneity one manifolds to have an invariant metric of quasipositive curvature.  As mentioned before there is a number of obstructions for positively curved metrics that are still valid in the case of quasipositive curvature but the proofs do not carry over. Hence we have to give new proofs of these results, which will be done in Sections \ref{secweyl} and \ref{secblock}. For the entire section let $M$ be a cohomogeneity one $G$-manifold with group diagram $H \subset \lbrace K^-,K^+\rbrace \subset G$ unless  stated otherwise. We start with a result that expresses in two ways that the representation of the group diagram is in some sense maximal. The first one is \emph{primitivity}, which is defined below. The second one, we call \emph{linear primitivity} and it states, that the isotropy groups along a horizontal geodesic generate the Lie algebra $\fr g$ of $G$.

\begin{dfn}[\cite{aa}, p.17]
	A $G$-manifold $M$ is called \emph{non-primitive} if there is a $G$-equi\-va\-ri\-ant map $M \to G/L$ for some proper subgroup $L \subset G$. Otherwise we say that $M$ is \emph{primitive}.
\end{dfn} 

	In the case of cohomogeneity one manifolds non primitivity means that for some metric on $M$ the non principal isotropy groups $K^\pm$ generate a proper subgroup $L \subset G$. Hence the action is primitive, if $K^-$ and $n K^+ n^{-1}$ generate $G$ for all $n \in N(H)_0$. Cohomogeneity one actions on quasipositively curved manifolds fulfil both properties:
\begin{lem}[Primitivity Lemma]\label{obstructions_lem_primitivity}
	Let $M$ have a metric with quasipositive curvature and let $c \colon \bb R \to M$ be a horizontal geodesic. Then:
	\begin{enumerate}[label = (\alph*), topsep=0pt,itemsep=-1ex,partopsep=1ex,parsep=1ex]
		\item \emph{(Linear Primitivity)} The Lie algebras of the isotropy groups along $c$ generate $\fr g$ as a vectorspace, i.e $\fr g = \sum_{w\in W} w\fr k^-w + w\fr k^+w$, where $\fr k^\pm$ are the Lie algebras of $K^\pm$. 
		\item  \emph{(Lower Weyl Group Bound)} The Weyl group is finite and $\vert W\vert \ge 2 \cdot \dim(G/H)/(l_-+l_+)$
		\item \emph{(Primitivity)} Any of the singular isotropy groups $K^\pm$ together with any conjugate of the other by an element of the core group generate $G$ as a group. In particular this is true for conjugation by elements of $N(H)_0$.
	\end{enumerate}
\end{lem}
The proof of the results (a) and (b) carry over from the positive curvature case (cf. \cite{gwz}, Lemma 2.2), since it purely relies on \cite{w3} Corollary 10 and only needs the geodesic $c$ to pass a point, where all sectional curvatures are positive. Result (c) requires a different and more difficult proof, which will be given in Section \ref{secweyl}. The following lemma is a direct consequence of linear primitivity.
\begin{lem}[Isotropy Lemma, \cite{gwz} Lemma 2.3]\label{obstructons_lem_isotropy}
	Suppose $H$ is non trivial. Then:
	\begin{enumerate}[label = (\alph*), topsep=0pt,itemsep=-1ex,partopsep=1ex,parsep=1ex]
		\item Any irreducible subrepresentation of the isotropy representation $G/H$ is equivalent to a subrepresentation of the isotropy representation of one of $K/H$, where $K$ is a non principal isotropy group.
		\item The isotropy representation of $G/H_0$ is spherical, i.e. $H_0$ acts transitively on the unit sphere of each $k$-dimensional irreducible subrepresentation if $k>1$.
	\end{enumerate}
\end{lem}

By part (a) any subrepresentation of $G/H$ is weakly equivalent, i.e. equivalent up to an automorphism of $H$, to a subrepresentation of $K^-/H$ or $K^+/H$. We say that a representation has to \emph{degenerate} in $K^-/H$ or $K^+/H$. In Table \ref{appendix_table_homsphere} in the appendix homogeneous spheres together with their isotropy representation are collected.\\
Table \ref{appendix_table_spherical} in the appendix contains a list of simple groups with their spherical simple subgroups, that was originally given in \cite{w2}. Another useful consequence is the following lemma.
\begin{lem}[\cite{gwz} Lemma 2.4]\label{obstructions_lem_simple}
	If $G$ is simple, $H$ can have at most one simple normal subgroup of rank at least two.
\end{lem}

One of the most important tools for classifying positively curved cohomogeneity one manifolds is the \emph{Rank Lemma}, which states:
\begin{lem}[Rank Lemma]\label{obstructions_lem_rank}
	One of $K^\pm$ has corank $0$, if $M$ is even dimensional, and at most corank $1$, if $M$ is odd dimensional. In particular $H$ has corank $1$, if $M$ is even dimensional, and corank $0$ or $2$, if $M$ is odd dimensional.
\end{lem}
In the case of positive curvature the proof relies on the well known fact, that a torus action on an even dimensional manifold always has a fixed point. In quasi positive curvature the lemma is still valid. The proof is more complicated and requires knowledge about the Weyl group and the classification of quasipositively curved cohomogeneity one manifolds with trivial principal isotropy group. We will proof this lemma in Section \ref{secweyl}.\\

Another very useful recognition tool will be Wilking's Connectedness Lemma.

\begin{thm}[Connectedness Lemma, \cite{w1}]\label{obstructions_thm_connect}
	Let $M^n$ be a compact positively curved Riemannian manifold and $N^{n-k}$ a compact totally geodesic submanifold. Then the inclusion map $N^{n-k} \to M$ is $(n-2k+1)$-connected. 
\end{thm}

As it is pointed out in \cite{w1} the above theorem also works in non negative curvature, if any geodesic emanating perpendicularly to $N^{n-k}$, intersects the set of points with positive sectional curvature. In particular it is shown, that each such geodesic has index at least $n-2k+1$ with respect to the energy functional. If we now have a cohomogeneity one manifold with quasipositive curvature and one of the non principal orbits, say $B_-$, is totally geodesic the proof of the Connectedness Lemma together with the proof of Lemma \ref{cohom1_lem_connected} implies that $n - 2(l_-+1)+1 \le l_+$, since each geodesic starting perpendicularly to $B_-$ and intersecting $B_+$ exactly once has index exactly $l_+$ with respect to the energy functional. This proves the following 
\begin{lem}\label{obstructions_lem_totgeod}
	Let $M$ be a quasipositively curved cohomogeneity one manifold, such that the non principal orbit $B_-$ is totally geodesic. Then $\dim G/H \le 2\cdot l_- + l_+$.
\end{lem}
Note, that, if one of the Weyl group elements, say $w_-$, is represented by a central element in $G$, the orbit $B_-$ is also totally geodesic, but here linear primitivity already implies the above bound on the dimension of the principal orbit.\\
Another consequence of the Connectedness Lemma together with Poincaré duality is:
\begin{lem}
	If $V$ is totally geodesic of codimension $2$ in $M$, which has odd dimension and positive curvature, then the universal covering $\tilde M$ of $M$ is a homotopy sphere.
\end{lem}
In the case of cohomogeneity one manifolds with quasipositive curvature, we again get the same conclusion, if one of the singular orbits is totally geodesic with codimension $2$.\\
%\begin{lem}
%	If $M$ is odd dimensional and has an invariant metric of quasipositive curvature, such that $B_-$ is totally geodesic and $l_- =1$. Then the universal covering of $M$ is a homotopy sphere.
%\end{lem}
By the previous lemma, it will be important to study cohomogeneity one homotopy spheres. We proof the following equivalent of Theorem 2.7 of \cite{gwz}.
\begin{thm}\label{obsturctions_thm_homotopy}
	Any cohomogeneity one homotopy sphere $\Sigma^n$ with an invariant metric of quasipositive curvature is equivariantly diffeomorphic to the standard sphere $\bb S^n$ with a linear action.
\end{thm}
\begin{proof}
	By the classification of low cohomogeneity homology spheres by Straume \cite{st}, the only cohomogeneity one homotopy spheres that are not standard spheres with linear actions are  $(2n-1)$-dimensional Brieskornvarieties $M^{2n-1}_d$ with an action by $\SO(2) \times \SO(n)$. For $d$ and $n\ge 3$ odd they are spheres and exotic if $2n-1 \equiv 1 \text{ mod } 8$ and $d\ge 3$. The group diagram of $M^{2n-1}_d$ for $n$ and $d$ odd is given by 
	\begin{align*}
		 &H = \bb Z_2 \times \SO(n-2) = \lbrace(\epsilon, \diag(\epsilon, \epsilon, A)) \vert \, \epsilon = \pm 1, \, A\in \SO(n-2)\rbrace\\
		 &K^-= \SO(2)\cdot \SO(n-2) = (e^{-i \theta}, \diag(R(d\theta),A))\\
		 &K^+=\O(n-1) = (\det B, \diag(\det B,B))
	\end{align*}
	Here $R(\theta)$ is the counterclockwise rotation by angle $\theta$. Note that $d=1$ is a standard sphere with a linear action and $l_- =1$ and $l_+ = n-2$. By \cite{gvwz} $M^{2n-1}_d$ does not have an invariant metric of non negative curvature, if $n \ge 4$ and $d \ge 3$. We will show, that $M^{2n-1}_d$ cannot have a metric of quasipositive curvature for $n \ge 3$ and $d\ge 3$. To obtain this, note, that $\SO(2)\times 1 \cap K^- = \bb Z_d$. Therefore $B_-$ is totally geodesic. If $M_d^{2n-1}$ has a metric with quasipositive curvature, by Lemma \ref{obstructions_lem_totgeod} we have $2n-2= \dim G/H \le 2 + n-2 =n$, which is equivalent to $n \le 2$, a contradiction. 
\end{proof}
	The analogue conclusion holds for cohomogeneity one manifolds with the rational cohomology ring of a non spherical rank one symmetric space (cf. \cite{i1,i2,u}).\\

As mentioned in the introduction Wilking's Chain Theorem (\cite{w2}, Theorem 5.1) is a useful tool in the classification of positively curved cohomogenity one manifolds. For quasipositive curvature we have the same result in Theorem \ref{intro_block}. To make this section complete, we restate it here and prove it in Section \ref{secblock}. 
\begin{thm} [Block Theorem]\label{obstructions_thm_block}
	Let $M$ be a simply connected Riemannian manifold with quasipositive sectional curvature, such that $G= L \times G_d$ acts isometrically on $M$ with cohomogeneity one, where $L$ is a compact connected Lie group and \emph{$(G_d,u) \in \lbrace (\text{SO}(d),1),$ $(\text{SU}(d),2),$ $(\text{Sp}(d),4) \rbrace$}. Assume that the principal isotropy group contains up to conjugacy a lower $k\times k$-block $B_k$, with $k \ge 3$, if $u = 1,2$, and $k\ge 2$, if $u = 4$. Then $M$ is equivariantly diffeomorphic to a rank one symmetric space with a linear action.    
	% If the action is essential, then $L \in \lbrace \lbrace e\rbrace, \text{SO}(2), \text{SU}(2), \text{Sp}(2)\rbrace $ and the action of $L \times G_d$ on $M$ is given by one of the elements of the Lists in \cite{gwz}
\end{thm}                                    
	The following obstruction is a generalization of Frankel's Theorem to quasipositive curvature, but it is not necessary for the submanifolds to be totally geodesic. We will denote by $\II$ the second fundamental form of a submanifold of $M$.	
\begin{lem}[Partial Frankel Lemma]\label{obstruction_lem_partial}
	Let $M$ be a Riemannian manifold and $B$ and $C$ disjoint closed submanifolds on $M$. Let $c\colon [a,b] \to M$ be a minimal geodesic from $B$ to $C$. Then there is no parallel vector field $V$ along $c$, tangent to $B$ at $t = a$ and to $C$ at $t= b$, such that $\langle \II (V(a),V(a)), \dot c(a) \rangle \ge 0$, $\langle \II(V(b),V(b), \dot c(b) \rangle\le 0$ and $\int_a^b \langle R(V(t),\dot c(t))\dot c(t), V(t)\rangle dt >0$. 
\end{lem}

\begin{proof}
	We assume $c$ to be parametrized by arc length and $p \coloneqq c(a)$ and $q \coloneqq c(b)$. Let $\epsilon >0$ be small enough, such that $\exp_{c(t)}$ is a diffeomorphism on $B_\epsilon(0_{c(t)})$ for each $t$. Write $S=\exp_p(B_\epsilon(0_p)\cap T_pB)$ and choose $\delta>0$ small enough for $\exp$ to be a diffeomorphism, when restricted to $\nu^{\le \delta}S$. By shrinking $\epsilon$ and $\delta$, if needed, we can write $B$ as a graph in $U \coloneqq \exp(\nu^{\le \delta}S)$, i.e. there exists a vector field $N$ normal to $S$, such that $B\cap U = \lbrace \exp(N(x))\vert \, x \in S\rbrace$. Hence $N(p) = 0$ and $\partial_s\vert_{s=0} N(d(s)) = 0$ for each curve in $S$ with $d(0) = p$. In a neighbourhood of $ p$ we can extend $N$ smoothly to a vector field $N_1$, which fulfills $N_1\vert_S = N$, $N_1(c(t)) =0$ for all $t$ and $\partial_s\vert_{s=0} N_1(d(s)) = 0$ for each curve $d$ in $M$ orthogonal to $c$. We construct a vector field $N_2$ locally at $C$ in a similar way. Now let $f_{1,2}\colon [a,b] \to [0,1]$ be smooth functions, such that $f_1$ is $1$ in a neighbourhood of $a$ and $0$ outside a slightly larger neighbourhood of $a$. Similar for $f_2$ and $b$. For a parallel vector field tangent to $B$ and $C$ we define 
	\begin{align*}
	\alpha_s(t) \coloneqq \exp_{\exp_{c(t)}(sV(t))}(f_1(t)N_1(\exp_{c(t)}(sV(t))) + f_2(t) N_2(\exp_{c(t)}(sV(t))))
	\end{align*}
	By construction $\alpha_0 = c$ and $\partial_s\vert_{s=0} \alpha_s(t) = V(t)$. Furthermore $\alpha_s(a) \in B$ and $\alpha_s(b) \in C$ for each $s$. Hence it is a variation of curves connecting $B$ and $C$. Let $L$ denote the length functional. Since $c$ is a geodesic $\partial_s \vert_0 L(\alpha_s) = 0$. By the second variation formula
	\begin{align*}
	\partial_s^2\vert_0 L(\alpha_s) &= \langle \dot c(t), \nabla_{\partial_s}\vert_0 \partial_s \alpha_s(t) \rangle \vert_a^b + \int_a^b \norm{\nabla_{\partial_t}V(t)}^2 - \langle R(V(t),\dot c(t))\dot c(t), V(t)\rangle dt\\
	&= \langle \dot c(b) , \II(V(b),V(b)) \rangle  - \langle \dot c(a) , \II(V(a),V(a)) \rangle - \int_a^b \langle R(V(t),\dot c(t))\dot c(t), V(t)\rangle dt
	\end{align*}
	If $V$ fulfils the assumptions this is a contradiction to $c$ being minimal.
\end{proof}
This lemma is particularly useful, if the tangent spaces of the non principal orbits contain subspaces, on which the second fundamental form vanishes, which can often be obtained by equivariance arguments. If the sum of the dimension of these subspaces is at least as large as the dimension of the manifold, we get a contradiction, like in Frankel's Theorem.\\ 

We finish this section with a useful reduction of the problem to so called \emph{essential actions}. Let $G = G^\prime \cdot G^{\prime \prime}$. If $G^\prime$ is contained in one of the singular isotropy groups, say $K^-$, then either $G^\prime$ acts transitively on the normal sphere to $B_-$ or $G^{\prime \prime}$ acts with the same orbits as $G$. An action is called \emph{fixed point homogeneous}, if the group acts transitively on the normal sphere to a fixed point component. A subaction is called \emph{orbit equivalent}, if it acts with the same orbits as the original action. Hence, if a singular isotropy group contains a normal factor of $G$, either a normal subgroup acts fixed point homogeneous or orbit equivalent. This motivates the following
\begin{dfn}[\cite{gwz} Definition 2.13]
	An action is called \emph{essential}, if no subaction is fixed point homogeneous and no normal subaction is orbit equivalent.
\end{dfn}
If $M$ is quasipositively curved and a subgroup of $G$ acts fixed point homogeneous, then it is easy to see, that a normal subgroup of $G$ acts fixed point homogeneous on a singular orbit: Let $L\subset G$ be a subgroup acting fixed point homogeneous. First assume that $L$ is not conjugated to a subgroup of $H$. Without loss of generality $L \subset K^-$. Since $K^-$ acts transitively on the normal sphere of $B_-$, so does $L$, since otherwise it would be contained in $H$, and hence $L$ is a normal subgroup of $K^-$. Since the tangent space of $B_-$ is invariant by the action of $L$ it has to act trivially there. Now let $K^\prime$ be the normal subgroup of $K^-$, that acts trivially on $B_-$. $K^\prime$ still acts transitively on the normal sphere. Since $kg.p_- = g.p_-$ for all $g$, we have $g K^\prime g^{-1} \subset K^-$ for all $g$. Furthermore $gkg^{-1}h.p_- =h.p_-$ for each $h,g \in G$, hence $g K^\prime g^{-1} \subset K^\prime$ and is therefore a normal subgroup of $G$. Now let $L \subset H$. $L$ cannot be in the kernel of both normal spheres to the singular orbits, since otherwise it is contained in the center of $G$ by Lemma \ref{lem_weylgroup_ineffkern} and cannot act fixed point homogeneous. Therefore we can assume, that $L$ acts non trivially on the sphere normal to $B_-$. Since the action is fixed point homogeneous and $T_{p_-}B_-$ is an invariant subspace, $L$ acts trivially on $B_-$. Let $K^\prime$ be the kernel of the isotropy action of $B_- = G/K^-$. Then $K^- = K^\prime \cdot K^{\prime\prime}$. By the classification of homogeneous spheres one of the factors acts transitively on $\bb S^{l_-}$. If $K^{\prime \prime}$ acts transitively, then $H$ projects onto $K^\prime$ and hence $L \subset H\cap K^\prime$ acts trivially on the normal sphere. Therefore $K^\prime$ has to act transitively on $\bb S^{l_-}$ and thus acts fixed point homogeneous.\\
%Note, that essential actions are defined slightly different in \cite{gwz}, namely they do not allow any subgroup of $G$ to act fixed point homogeneous. In order ot do the reduction to essential actions in the case of quasi positive curvature, we have to make this slightly stricter definition. We will proof that we can reduce the classification tot essential action in section \ref{secessential}.
By Theorem \ref{intro_fp_hom}, it is enough to consider essential actions. The result follows from the following lemma.
\begin{lem}\label{obstructions_lem_alexandrov}
	Let $X$ be a non negatively curved cohomogeneity one Alexandrov space with boundary $F$. Assume there is an open set $U \subset X$ which has positive curvature. Then the set $C$ of maximal distance to $F$ consists of a single point.   
\end{lem}
\begin{proof}
	 We will use the partial flat strip property by Shioya and Yamaguchi (see \cite{ShiYa00} Proposition 9.10 and \cite{RoWa22} Lemma 1.4). It states the following. Let $X_t \coloneqq \lbrace x \in X \vert d_F(x) \ge t \rbrace$. Furthermore let $c \colon \lbrack 0, 1 \rbrack \to C$ be a geodesic and $\gamma_0$ a shortest geodesic from $c(0)$ to $\partial X_t$ perpendicular to $C$. Then there exists a minimal geodesic $\gamma_1$ from $c(1)$ to $\partial X_t$, such that $\gamma_0$, $c$ and $\gamma_1$ bound a totally geodesic flat strip. Since $X$ has a cohomogeneity one action, a flat strip will always intersect some set of positive curvature if $t$ is chosen close to $0$. Hence $C$ is a point.
\end{proof}
Lemma \ref{obstructions_lem_alexandrov} implies, that there is an orbit of maximal distance to a fixed point component $F$ in $M$, which implies that $M$ is the union of tubular neighbourhoods of $F$ and this orbit. Theorem \ref{intro_fp_hom} is now proven in the same way as the classification of positively curved fixed point homogeneous manifolds by Grove and Searle \cite{gs}.

%% file: documents/weylgroup.tex
\section{Weyl groups}\label{secweyl}

We already know, that the Weyl group of a cohomogeneity one manifold with quasipositive curvature is finite. In this chapter we will establish bounds on the order of $W$. Together with the lower Weyl group bound this will turn out to be a powerful tool in the classification. Furthermore we will proof the Rank Lemma. To do this we first determine, which groups can act on a quasipositively curved cohomogeneity one manifold with trivial principal isotropy group and give bounds on the order of the Weyl groups of these actions.

\begin{lem}[Core-Weyl Lemma]\label{lem_weyl_coreweyl}
	Let $M$ be a compact Riemannian manifold with quasipositive sectional curvature and $G$  a compact Lie group acting isometrically on $M$ with cohomogeneity one and trivial principal isotropy group. Then $G$ has at most two components and the action is primitve. Moreover 
		\begin{align*}
		\vert W \vert \text{ divides } 2\cdot \text{\emph{rk}}\,(G) \cdot\vert G/G_0\vert \le 8
		\end{align*}
	Furthermore $G_0$ is one of the groups $S^1$, $S^3$, $T^2$, $S^1\times S^3$, \emph{$\text{U}(2)$}, $S^3 \times S^3$, \emph{$\SO(3)\times S^3$}, or \emph{$\SO(4)$}, and $M$ is fixed point homogeneous in all cases but \emph{$G_0 = \SO(3) \times S^3$}.
\end{lem}

In \cite{gwz} Lemma 3.2 the same result was proven for positively curved cohomogeneity one manifolds. The proof reduces to Lie groups of rank at most two by the Rank Lemma. In the case of quasipositive sectional curvature we cannot use the Rank Lemma at this point. We will see later, that it also follows from Lemma \ref{lem_weyl_coreweyl}. For the remaining section let $M$ be a quasipositively curved cohomogeneity one $G$-manifold with group diagram $H \subset \lbrace K^-,K^+\rbrace \subset G$. We will now start the proof of the Core-Weyl Lemma, with giving an upper bound on the order of the Weyl group, if $H$ is trivial.

\begin{lem}\label{lem_weyl_weylcommute}
	If $H$ is trivial and $G$ connected, then the Weyl group elements $w_-$ and $w_+$ commute, i.e. $\vert W \vert \le 4$.
\end{lem} 

\begin{proof}
	We first want to proof, that the Weyl group of any cohomogeneity one metric on $M$ is finite. This means, that for each $g \in G$ the corresponding Weyl group $W(g)$ generated by $w_-$ and $gw_+g^{-1}$ is finite. For the following we do not need to assume that $H$ is trivial: It is straightforward to show, that we can find a neighbourhood $U \subset N(H)_0$ of $e$ and cohomogeneity one metrics on $M$, which are close to the original metric and still have quasipositive curvature, such that $p_-$ and $g.p_+$ are joint by a minimal geodesic for each $g \in U$. In particular it is enough to vary the metric on a small tubular neighbourhood of a principal orbit, where the curvature is positive. If we now assume that $H = e$, then $N(H) = G$. Therefore the order of $w_-gw_+g^{-1}$ is finite or each $g \in U$.\\
	We claim that, if for two involutions $a$ and $b$ in a compact connected Lie group $G$ the order of $agbg^{-1}$ is finite for each $g$ in an open connected subset $U$, then the order of $agbg^{-1}$ is constant in $U$:
	By the Peter-Weyl Theorem any Lie group can be embedded into $\U(N)$ for some $N$ large enough. We fix an embedding $G \subset \U(N)$.
	%and denote by $\bb K$, the countable field, generated by the $n$-th complex roots of unity for each $n \in \bb N$. 
	Let $A \in \U(N)$ be a matrix of finite order. Then the coefficients $k_0(A), \ldots, k_N(A)$ of the characteristic polynomial $\chi_A$ are polynomials in complex roots of unity, which generate a countable subset of $\bb C$.
	%since $A$ is conjugate to a diagonal matrix of the same order.
%	 $\chi_A$ depends only on the conjugacy class of $A$. Since unitary matrices can be diagonalised there are $\lambda_1, \ldots,\lambda_N\in \bb C\backslash\lbrace 0\rbrace $, such that $A$ is conjugated to $B\coloneqq\diag(\lambda_1, \ldots,\lambda_N)$. $B$ has the same order as $A$, because they are in the same conjugacy class. Let $m$ be the order of $A$. Then 
%	\begin{align*}
%	I = B^m = (\diag(\lambda_1,\ldots,\lambda_N))^m=\diag(\lambda_1^m,\ldots,\lambda_N^m)
%	\end{align*}
%	Hence $\lambda_1, \ldots,\lambda_N$ are complex roots of $1$ and $\chi_A(t) = (\lambda_1-t)\cdots(\lambda_N-t)$.
%	The coefficients are polynomials in $\lambda_1, \ldots, \lambda_N$ and are therefore elements in $\bb K$. \\
%	We denote the coefficients of $\chi_A$ by $k_0(A), \ldots, k_N(A)$. 
	The by continuity of the coefficients, the map $r\colon G \to \bb C^{N+1},g \mapsto (k_0(agbg^{-1}),\ldots, k_N(agbg^{-1}))$ is continuous. Since each $agbg^{-1}$ has finite order for $g \in U$, the image of $r$ lies in a countable subset of $\bb C^{N+1}$ and hence $r\vert_U$ is constant, since $U$ is connected. This proves, that the characteristic polynomial of $agbg^{-1}$ is constant in $g$ and therefore $agbg^{-1}$ are conjugate to each other in $\U(N)$. Hence the order of $agbg^{-1}$ is constant in $g$.\\
	By the above claim the map $r$ is constant in an open neighbourhood of $e$. Since it is analytical in a real analytic structure of $G$, it must be constant on all of $G$ and hence $agbg^{-1}$ has constant order in $g \in G$. Now we can pick $g \in G$, such that $w_-$ and $gw_+g^{-1}$ are contained in the same maximal torus and therefore the order of $w_-gw_+g^{-1}$ is two. But then also $w_-w_+$ has order two, which implies, that $w_-$ and $w_+$ commute.
\end{proof}
\begin{proof}[Proof of Lemma \ref{lem_weyl_coreweyl}]
	First note, that both $K^\pm$ are diffeomorphic to spheres, since $H$ is trivial. Therefore $K^\pm \in \lbrace \bb Z_2, S^1,S^3\rbrace$ and at most one of $K^\pm$ is $\bb Z_2$. $G$ has at most two components. If $G$ is not connected, then the Weyl group of $G_0$ has index $2$ in $W$ and the bound follows from the connected case. Consequently, we can focus on connected Lie groups $G$.\\
	Firt note that $\emph{\rk} G \le \dim K^- + \dim K^+$, which proves as follows:
	Denote the Lie algebras of $K^\pm$ by $\fr k^\pm$. Let $ w_0 \in W$ be the product of $w_-$ and $w_+$ and $Z(w_0) $ the centraliser of $w_0$. Since $w_0$ is contained in a maximal torus $\rk G = \rk Z(w_0)$ and we can proof the claim for $Z(w_0)$. Let $\mathfrak z(w_0)$ be the Lie algebra of $Z(w_0)$. Using an arbitrary biinvariant metric on $G$ we can pick a projection $P \colon \mathfrak g \to \mathfrak z (w_0) $, which commutes with $\Ad_{w_0}$.  We have $\mathfrak g = \sum_{i=0}^{1}\Ad_{w_0^ {i}}(\fr k^-)+\Ad_{w_0^{i}}(\fr k^+)$ by linear primitivity. Hence 
	\begin{align*}
		\mathfrak z(w_0) &= P(\mathfrak g) = P\left(\sum_{i=0}^{1}\Ad_{w_0^i}(\fr k^-)+\Ad_{w_0^i}(\fr k^+)\right)\\
		&= \sum_{i=0}^{1}\Ad_{w_0^i}(P(\fr k^-))+\Ad_{w_0^i}(P(\fr k^+)) = P(\fr k^-)+P(\fr k^+)
	\end{align*}
	The last equality holds, since $P(\fr k^\pm) \subseteq  \mathfrak z(w_0)$. 
	
	By Lemma \ref{lem_weyl_weylcommute} and linear primitivity, we also get a bound on the dimension of $G$: $\dim G \le 2\cdot (l_-+l_+)$. The proof of \cite{gwz} Lemma 3.2  already handles the cases with $\rk G \le 2$ and applies here, if we assume $\rk G \le 2$. Therefore we will only rule out the cases, where $3 \le\rk G \le 6$. This implies, that at least one of $K^\pm$, say $K^-$, is isomorphic to $S^3$, and we consider the cases $l_+ = 0,1,3$.\\
	$l_+ = 0$: In this case $K^+ = \bb Z_2$, $\dim G \le 6$ and $\rk G =3$. By linear primitivity $G$ is generated by two groups isomorphic to $S^3$ and hence $G$ is semisimple. But semisimple groups of rank three are at least $9$-dimensional.\\
	$l_+ = 1$: $K^+ = S^1$, $\dim G \le 8$ and $3 \le \rk G \le 4$. The center of $G$ is at most one dimensional, since $\fr k^-$ and $\Ad_{w_+}(\fr k^-)$ are contained in the semisimple part of $\fr g$ and hence $\fr k^+$ projects onto the center. By the dimensional bound on $G$ it cannot be semisimple and must contain exactly two $3$-dimensional factors, from which one is isomorphic to $S^3$. Since the center is at most one dimensional and there is no eight dimensional semisimple group of rank 3, $G$ is at most of dimension $7$. Therefore $G$ is covered by $\tilde G = S^3 \times S^3 \times S^1$. A lift of $w_+$ in $\tilde G$ has order at most $4$. But obviously conjugating $K^- = S^3 \subset \tilde G$ with such an element leaves an at least one dimensional subgroup invariant. Hence, we have $\dim \Ad_{w_+}(\fr k^-) \cap \fr k^- \ge 1$. Since $w_-$ must be central, because $K^- = S^3$,  linear primitivity implies $\dim G \le 3 +3 +1 -1 = 6$, a contradiction. \\
	$l_+ =3$: We have $K^+ = S^3$, $\dim G \le 12$ and $3 \le \rk G \le 6$. Since $\fr k^\pm$ and $\Ad_{w_\mp}(\fr k^\pm)$ are contained in the semisimple part of $\fr g$, linear primitivity implies, that $G$ is semisimple. By the bounds on dimension and rank $G$ either consists of at most four simple rank one groups or its universal covering is isomorphic to $S^3 \times \SU(3)$. In the first case, we get again $\dim G \le 6$, since both $w_\pm$ are central. Hence we are left with the case $\tilde G = S^3 \times \SU(3)$, which has center $\bb Z_3 \times \bb Z_2$. Note that it is easy to show, that for nontrivial involutions $a$ and $b$ in $L = \SO(3)$, $\SU(3)$ and $\SU(3)/\bb Z_3$ there exists $g \in L$ such that $agbg^{-1}$ has infinite order. This only leaves the cases, where $G$ contains an $S^3$-factor. But then one of $w_\pm$, say $w_-$ is central in $G$ and linear primitivity implies $\dim G \le 9$, a contradiction. 
\end{proof}

It is now possible to proof the Rank Lemma using the Core-Weyl Lemma.
\begin{proof}[Proof of the Rank Lemma]
	Let $T \subset H$ be a maximal torus and $N(T)$ its normalizer. Then $N(T)_0/T$ acts with cohomogeneity one on a component of $M^\prime \subset M^T$. The corank of $H$ is the rank of $N(T)/T$ and the principal isotropy group $H^\prime = (N(T)_0 \cap H)/T$ is finite. The dimensions of $M$ and of $M^\prime$ have the same parity. Therefore it is enough to show the Rank Lemma for $G$ acting with finite principal isotropy group. Denote the order of $H$ by $n$. Suppose $n=1$. Then $G$ acts with trivial principal isotropy group on $M$. If $\dim M = 1$, then $G$ is finite and has therefore rank $0$. If $\dim M \ge 2$, then $G_0$ is one of the groups of the Core-Weyl Lemma and $K^\pm$ must be isomorphic to $\bb Z_2$, $S^1$ or $S^3$, such that at most one of them is $\bb Z_2$. Therefore $\rk G \le 2$. If the dimension of $M$ is even, then the rank of $G$ is odd and hence  one. Since one of $K^\pm$ is $S^1$ or $S^3$, it has full rank in $G$. If $\dim M$ is odd, then the rank of $G$ is even and hence two. Since one of $K^\pm$ is of rank one, we get the result of the Rank Lemma.\\
	Now suppose $n>1$. Let $\iota \in H$. Then $C(\iota)/\iota$ has the same rank as $G$ and acts with cohomogeneity one on a component of $M^\iota$ with principal isotropy group $\tilde H = (C(\iota)_0 \cap H)/\iota$. Therefore the order of $\tilde H$ is strictly less than the order of $H$ and the result follows by induction.
\end{proof}
  
We will finish the section with repeating results from \cite{gwz} Section 3 in terms of quasipositive curvature. All proofs carry over from the positive curvature case.

\begin{cor}[Group Primitivity]
	Suppose that $M$ admits a quasipositively curved metric. Consider any other cohomogeneity one metric on $M$, then the corresponding groups $K^-$ and $K^+$ generate $G$ as a Lie group. Equivalently $K^-$ and $n K^+n^{-1}$ generate $G$ for any $n \in N(H)_0$.
\end{cor}

\begin{lem}\label{lem_weylgroup_ineffkern}
	Assume $G$ acts effectively. Then the intersection $H_-\cap H_+$ of the ineffective kernels $H_\pm$ of $K^\pm/H$ is trivial. 
\end{lem}

\begin{lem}\label{lem_weylgroup_codim2}
	Suppose $M$ is a simply connected quasipositively curved cohomogeneity one $G$ manifold with singular orbits of codimension two. Then one of the following holds:
	\begin{enumerate}[label = (\alph*), topsep=0pt,itemsep=-1ex,partopsep=1ex,parsep=1ex]
		\item $H=\lbrace 1 \rbrace$ and both $K^\pm$ are isomorphic to \emph{$\SO(2)$}.
		\item $H = H_- = \bb Z_2$, \emph{$K^- = \SO(2)$}, \emph{$K^+ = \O(2)$}.
		\item $H=H_-\cdot H_+ = \bb Z_2 \times \bb Z_2$, and both $K^\pm$ are isomorphic to \emph{$\O(2)$}.
	\end{enumerate}
\end{lem}

\begin{prop}[Upper Weyl Group Bound]\label{prop_weylgroup_upperweyl}
	Assume that $M$ is simply connected and $G$ connected. Then:
	\begin{enumerate}[label = (\alph*), topsep=0pt,itemsep=-1ex,partopsep=1ex,parsep=1ex]
		\item If $H/H_0$ is trivial or cyclic, we have $\vert W \vert \le 8$, if the corank of $H$ is two, and $\vert W\vert \le 4$, if the corank is one.
		\item If $H$ is connected and $l_\pm$ are both odd, then $\vert W \vert \le 4$ in the corank two case and $\vert W\vert \le 2$ in the corank one case.
		\item If none of $(N(H)\cap K^\pm)/H$ is finite, we have $\vert W \vert \le 4$ in the corank two case and $\vert W\vert \le 2$ in the corank one case.
	\end{enumerate}
\end{prop}
%\begin{rem*}
%	Note, that, if $K^-_0$  is covered by $S^3 \times S^1$, $K^+ = T^2$ and $H_0 = S^1$, then $(N(H)\cap K^-)/H$ has dimension one, since  $S^3 \times S^1$ has only one outer automorphism acting by complex conjugation on the center. This automorphism cannot commute with $H_0$, that must project to the center as well as to the maximal torus of the $S^3$ factor. Thus $(N(H)\cap K^-)\cong T^2$ and we can always apply result (c) from the above lemma. 
%\end{rem*}

%% file: documents/block.tex
\section{The Block Theorem}\label{secblock}

In this section we want to prove the Block Theorem (Theorem \ref{intro_block}). Wilking's original result  uses the Connectedness Lemma and requires positive curvature at all points on the manifold. Therefore we will follow another approach using the following key lemma, which is a generalization of  \cite{wz} Lemma 1.3 to any cohomogeneity.
\begin{lem}[Block Lemma]\label{block_lem_blocklemma}
	Let $L$ be a compact Lie group and \emph{$(G_d,u) \in \lbrace (\SO(d),1),$} \emph{$(\SU(d),2),$ $(\Sp(d),4)\rbrace$}. Suppose $G = L \times G_d$ acts almost effectively and isometrically on a Riemannian manifold $M$ with cohomogeneity $r$, such that the principal isotropy group $H$ contains up to conjugacy a lower $k \times k$-block with $k \ge 3$, if $u = 1,2$, and $k \ge 2$, if $u = 4$. Let $k$ be maximal with this property. Suppose $M$ contains a point $p$, such that all sectional curvatures at $p$ are positive. Then $d-k \le r +1$.
\end{lem}

\begin{proof}
	Without loss of generality we can assume, that $p$ is a regular point of the $G$-action, since those points form an open and dense set in $M$. Moreover we can assume, that $H$ contains the lower $k \times k$-block. Since $k$ is maximal, $B_k$ is a normal subgroup of $H_0$ by \cite{w2} Lemma 2.7. For some small $\epsilon > 0$ we set $S_p \coloneqq \nu^{<\epsilon}G.p$. By the Slice Theorem $\exp(\nu^{< \epsilon}G.p) \cong G \times_{H} S_p$, if $\epsilon$ is small enough. All $q \in S_p$ have the same isotropy group $H$, because $H$ is the principal isotropy group. Therefore $G \times_HS_p = G/H \times S_p$. Note that for any geodesic $c$ orthogonal to $G.p$ the orbit $G.c(t)$ is orthogonal to $c$. We now pick a biinvariant metric $Q$ on $\fr g$. Let $\fr q$ be the orthogonal complement of $\fr h$ with respect to $Q$, so $\fr g = \fr h \oplus \fr q$. Since all points in $S_p$ have the same isotropy group, we can pick a smooth family of symmetric $\Ad_H$-invariant Endomorphisms 
	$(P_x \colon \fr q \to \fr q)_{x \in S_p}$,
	such that $g_x(X^*,Y^*) = Q(P_xX,Y)$ for all $X,Y \in \fr q$, where $X^*$ denotes the action field.
	We will now write down the Gauss equations for the principal orbit $G.p$. First pick an orthonormal basis $(e_1, \ldots,e_r)$ of $\nu_pG.p$. We compute the Shape operators $S_k (X) \coloneqq S_{e_k} (X) = -(\nabla_{X^*} e_k)^t = -(\nabla_{e_k}X^*)^t$ of $G.p$ in these directions:
	\begin{align*}
	-g_p(S_kX^*,Y^*) &= g_p(\nabla_{e_k}X^*,Y^*) = g_p(X^*,\nabla_{e_k}Y^*) \\&= \frac{1}{2} e_k(g(X^*,Y^*)) = \frac{1}{2} e_k (Q(P_xX,Y)) = \frac{1}{2} Q(\partial_k P_pX,Y)
	\end{align*}
	To shorten the notation we will write $P \coloneqq P_p$ and $P_k \coloneqq \partial_kP_p$. By the Gauss equations now the following holds, where $R^\prime$ is the intrinsic curvature of the Orbit $G.p$:
	\begin{align*}
	g(R(X,Y)Z,W) =& \,g(R^\prime(X,Y)Z,W) + g(\II (X,Z),\II(Y,W)) - g(\II(Y,Z),\II(X,W))\\
	=& \,g(R^\prime(X,Y)Z,W) + \sum_{k=1}^{r}g(S_kX,Z)g(S_kY,W) - g(S_kY,Z)g(S_kX,W)\\
	=& \,g(R^\prime(X,Y)Z,W) + \frac{1}{4}\sum_{k=1}^{r}\lbrack Q(P_kX,Z)Q(P_kY,  W) \\
	&- Q(P_kY,Z)Q(P_kX,W) \rbrack
	\end{align*}
	So for $Y=Z$ and $X=W$, we get:
	\begin{align}\label{gauss}
	g(R(X,Y)Y,X) = g(R^\prime(X,Y)Y,X) + \frac{1}{4}\sum_{k=1}^{r}Q(P_kX,Y)^2 - Q(P_kY,Y)Q(P_kX,X)
	\end{align}
	The summand $g(R^\prime(X,Y)Y,X)$ was already analysed in \cite{gz}. It vanishes if $X$ and $Y$ are commuting eigenvectors of $P$ (see also \cite{wz}). We will now make use of the assumption, that $G = L \times G_d$ and that $H$ contains a lower $k \times k$-block:\\
	Let $N(B_k)$ be the normalizer of $B_k$. Since $B_k$ is normal in $H_0$, $H_0 \subset N(B_k)$. Let $\fr n$ be the Lie algebra of $N(B_k)$. The orthogonal complement $\fr p$ of $\fr n$ consists of the direct sum of $d-k$ pairwise irreducible standard representations of $B_k$ and is contained in $\fr q$. So $\fr p = \bigoplus_{l=1}^{d-k} \bb K^{k}$, with $(\bb K,u) \in \lbrace (\bb R,1), (\bb C,2), (\bb H,4)\rbrace$. Since $k \ge 3$, if $u \in {1,2}$, respectively $k \ge 2$, if $u = 4$, the upper left $(d-k) \times (d-k)$-block in $G_d$ acts transitively on the irreducible representations of $B_k$. Therefore we can assume (by conjugating $P$), that the upper right $(d-k) \times k$-block, that represents $\fr p$ consists of eigenvectors of $P$. Since all $P_x$ are $\Ad_H$-equivariant, also $P_l$ is $\Ad_H$-equivariant for each $l \in \lbrace 1,\ldots r\rbrace$. Now denote by $(e_i)_j$ the $i$-th Standard unitvector in the $j$-th row. Since $P_l$ is $B_k$-invariant and $k$ is large enough, the concatenation of $P_l$ and the projection to a row is multiplication with an element of $\bb K$ by a more general version of Shur's Lemma (cf. Theorem 6.7, \cite{bt}). Since $B_k$ acts transitively on the unit-sphere in $\bb K^k$ this implies:
	\begin{align*}
	Q(P_l(e_{i_1})_{j_1},(e_{i_2})_{j_2}) &= 0, \text{ if }i_1 \neq i_2\\
	Q(P_l(e_{i})_{j}, (e_i)_j) &= Q(P_l(e_1)_j,(e_1)_j) = \lambda_j^l
	\end{align*}
	Now we choose for $j_1 \neq j_2$: $X = (e_1)_{j_1}$ and $Y=(e_2)_{j_2}$. Note, that $X$ and $Y$ are commuting eigenvectors of $P$. This reduces (\ref{gauss}) to
	\begin{align*}
	g(R(X,Y)Y,X) = -\frac{1}{4}\sum_{l=1}^{r}\lambda_{j_1}^l\lambda_{j_2}^l
	\end{align*}
	We set $\lambda_{j} \coloneqq (\lambda_j^1, \ldots, \lambda_j^r)$ for $j \in \lbrace 1, \ldots, d-k\rbrace$ and denote by $\langle .,. \rangle$ the standard scalarproduct on $\bb R^{r}$. Since all sectional curvatures at $p$ are positive, we get:
	\begin{align*}
	\langle \lambda_{j_1}, \lambda_{j_2} \rangle < 0, \text{ if } j_1 \neq j_2
	\end{align*}
	An euclidean vector space of dimension $r$ contains at most $r+1$ vectors, which have pairwise negative scalar product. Hence $d-k \le r+1$
\end{proof}

\begin{proof}[Proof of the Block Theorem] 
	Let $H\subset \lbrace K^-, K^+ \rbrace \subset G$ be the group diagram of the $G = L \times G_d$ action, such that $B_k$ is contained in $H$, where $k$ is maximal among all subgroups contained in $H$ conjugated to a lower $k \times k$ block. This implies, that $B_k$ is normal in $H_0$. The general strategy is to compute the possible group diagrams of essential actions and compare them to the ones contained in Tables \ref{appendix_table_cohom_one_even} and \ref{appendix_table_cohom_one_odd_sphere}.
	By the following we can assume the action to be essential:
	If a normal subgroup acts fixed point homogeneous then $M$ must be a rank one symmetric space with a linear action by Theorem \ref{intro_fp_hom}. Hence we can assume, that a normal subgroup $G^\prime$ of $G$ acts orbit-equivalent. Suppose $G^\prime$ does not contain $G_d$ as a normal factor. Then $L$ acts orbit equivalent. This implies that $H$ projects to $G_d$, which cannot happen, since $B_k$ is normal in $H_0$. From now on we assume the action of $G$ on $M$ to be essential.
	
	By the Block Lemma, we can assume that $d-k=2$, since for $d-k = 1$ the action is not essential. 
	We fix the following notation: For a subgroup $W \subseteq G$ define $W_L \coloneqq L \cap W$, $W_{G_d} \coloneqq G_d \cap W$. Furthermore there exists a connected and compact diagonal subgroup $W_{\Delta}$, such that $W_0 = (W_L \cdot W_\Delta \cdot W_{G_d})_0$. By $U_2$ we denote the upper $2 \times 2$ block.\\
	Since the action is spherical, $B_k$ is normal in $H$. The normalizer of $B_k$ is given as follows:
	\begin{align*}
		N(B_k)= S\begin{pmatrix} \text O(2) & 0\\ 0 & \text O(k)\end{pmatrix},\,  S\begin{pmatrix} \text U(2) & 0\\ 0 & \text U(k)\end{pmatrix} \text{ or } \begin{pmatrix} \text{Sp}(2) & 0\\ 0 & \text {Sp}(k)\end{pmatrix}
%	C(B_k) = \begin{cases}
%	S\begin{pmatrix} \text{O}(2) & 0\\ 0 & \Delta \bb Z_2\end{pmatrix} & u = 1 \\
%	S\begin{pmatrix} \text U(2) & 0\\ 0 & \Delta  S^1\end{pmatrix} & u = 2 \\
%	\begin{pmatrix} \text{Sp}(2) & 0\\ 0 & \Delta \bb Z_2\end{pmatrix} & u = 4 \\
%	\end{cases}, \quad
%	N_0(B_k)=\begin{cases}
%	S\begin{pmatrix} \text O(2) & 0\\ 0 & \text O(k)\end{pmatrix} & u = 1 \\
%	S\begin{pmatrix} \text U(2) & 0\\ 0 & \text U(k)\end{pmatrix} & u = 2 \\
%	\begin{pmatrix} \text{Sp}(2) & 0\\ 0 & \text {Sp}(k)\end{pmatrix} & u = 4 \\
%	\end{cases}
	\end{align*}
	There exists a closed subgroup $\bar H \subset G$ commuting with $B_k$, such that $H_0 = \bar H \cdot B_k$. Furthermore one of the singular isotropy groups, say $K^-$, contains a $(k+1) \times (k+1)$-block as a normal factor, because the standard $B_k$ representation must degenerate by the Isotropy Lemma. Hence $K^-_0 = \bar K^- \cdot B_{k+1}$ and $l_- = uk + u -1 > 1$. Since $B_{k+1}$ acts transitively on $\bb S^{l_-}$, the Weyl group element $w_-$ can be represented by an element in $G_d$. Therefore $w_-\proj_L(K^+)w_- = \proj_L(K^+)$. Since $H$ must project onto $\bar K^-$ by the classification of homogeneous spheres, we have $\proj_L(K^-) = \proj_L(H) \subseteq \proj_L(K^+)$. By linear primitivity, $K^+$ projects to $L$ and hence $K^+_L = \lbrace e \rbrace$, since $K^+_L$ is a normal subgroup of $L$. Furthermore we note, that $H$ cannot have corank $0$, since then both $\bb S^{l_\pm}$ are even dimensional, which implies $\bar K^+ = \bar K^- = \bar H = \lbrace e \rbrace$. But $B_k$ has at least corank $1$.\\
	First we exclude the case, where $B_k$ acts non trivially on both normal spheres:\\			
	Suppose, that $B_k$ does not act trivially on $\bb S^{l_+}$. Then all groups are connected. Furthermore $\proj_L(K^-) = \proj_L(H)=\proj_L(K^+)$ and hence  $L=\lbrace e \rbrace$, since the action is essential. Moreover, $l_- = l_+>1$, since the $B_k$-representations can only degenerate in a $(k+1)\times (k+1)$ block, which implies, that all groups are connected. \\
	If $u=1$, then $B_k$ has corank $1$. This immediately implies, that $M$ is even dimensional. Also $H= B_k$ and $K^\pm \cong B_{k+1}$. But $U_2$ acts transitively on the irreducible $B_k$ representations in $G_d$. Since $U_2 \subset N_0(H)$, this contradicts primitivity.\\
	If $M$ is odd-dimensional and $u = 2,4$, then $B_k$ has corank $2$ in $G_d$. Hence again $H=B_k$ and $K^\pm \cong B_{k+1}$ and we have a contradiction as before.\\
	Therefore we are left with the cases, where $\dim M$ is even and $u = 2,4$. \\
	$u=2$: Since both spheres are odd-dimensional, $K^\pm$ both have corank $0$ and hence $K^\pm \cong \U(k+1)$ and $H = S^1 \cdot B_k$, where $S^1$ is in the centralizer of $B_k$ and therefore $S^1 = \diag(z^n,z^m,z^l,\ldots,z^l)$. 
	$S^1$ acts with weights $n -m$, $n-l$ and $m-l$. If $n = m$, then again $U_2$ is contained in the normalizer of $H$, contradicting primitivity. Therefore we can assume $n \neq m$. But in this case $B_k$ acts trivially on the space with weight $n-m$ and hence it cannot degenerate in one of the normal spheres.\\
	$u=4$: In this case either $K^\pm \cong S^1 \times \Sp(k+1)$ and $H = S^1 \cdot \Sp(k)$ or $K^\pm \cong \Sp(1) \times \Sp(k+1)$ and $H = \Sp(1)\cdot\Sp(k)$. So first assume the case $H = S^1\cdot \Sp(k)$, with $S^1= \diag(z^n,z^m,1,\ldots,1)$, which acts with weights $2n$, $2m$, $n\pm m$, $m$ and $n$. If $n = \pm m$ the upper left $\SU(2)$ block is contained in the normalizer of $H$. But then $N(H)_0/H \cong \SO(3)$ acts with trivial principle isotropy group on a component of $M^H$, contradicting the Core-Weyl Lemma. If $n \neq \pm m$, then there can be at most two distinct non trivial weights in the upper $2 \times 2$ block, by the Isotropy Lemma. Hence $m =0$, and $B_\pm$ are both totally geodesic, since $K^\pm$ both contain the central element of $\Sp(k+2)$. This contradicts Frankel's Theorem, since $8k+7 = \dim G/H \le l_-+l_+ = 8k+6$.\\
	We are left with the case $H = \Sp(1) \cdot \Sp(k)$. By the Isotropy Lemma the $\Sp(1)$ factor in $H$ can only have $3$-dimensional representations in $U_2$ and hence must be given by $\diag(a,a,1, \ldots,1)$ and therefore acts non trivially on both normal spheres. Hence $K^\pm \cong \Sp(1) \cdot \Sp(k+1)$. Note that $\SO(2) \subset \Sp(2)$ acts transitively on the irreducible representations of $H$, since they must be of real type. Therefore the action is not primitive, since $N(H)$ contains $\SO(2)$.
	
	From now on we assume $B_k$ to act trivially on $\bb S^{l_+}$. Hence, we can assume, that $K^+ = \bar K^+ \cdot B_k$, $H_0 = \bar H \cdot B_k$ with $\bar H$ and $\bar K^+$ contained in the centralizer  of $B_k$ and $\bb S^{l_+} \cong \bar K^+/\bar H$. Also $\bar K^+ = K^+_\Delta \cdot \bar K^+_{G_d}$ and $\bar K^+$ projects onto $L$. Since $B_k$ has corank $1$ in $G_d$, if $u=1$, and $2$, if $u =2,4$, we have the following: In even dimensions $\rk \bar H - \rk L$ equals $0$, if $u = 1$, and $1$, if $u=2,4$. In odd dimensions $\rk \bar H -\rk L $ equals $-1$, if $u = 1$, and $ 0$, if $u =2,4$. Furthermore $\bar H$ has rank at most $1$, if $u = 1$, and at most $2$, if $u = 2,4$. We will distinguish the cases $0 \le \rk \bar H \le 2$.\\
	$\rk \bar H =0$: If $u=1$, then $\rk L =0$, if $\dim M$ is even, and $1$, if $\dim M$ is odd, by the above. In both cases $H_0 = B_k$ and hence $K_0^- = B_{k+1}$ and $K^+ = \SO(2) \cdot B_k$. Furthermore $H/H_0 \cong K^-/K^-_0 \cong \lbrace e \rbrace, \bb Z_2$. Suppose $H$ is connected. Then the Weyl group element $w_+$ can be represented by an element projecting to $  \diag(-1,-1,1, \dots,1)$. Therefore $w_+ K^- w_+ = K^-$ and hence by linear primitivity $2k+1 \le \dim G/H \le l_-+2l_+ = k+2$, a contradiction. Thus $H$ is not connected and we have the $\SO(k+2)$ action on $\bb C\bb P^{k+1}$, if $\dim M$ is even, and the $\SO(2) \SO(k+2)$ action on $\bb S^{2k+3}$, if $\dim M$ is odd.\\
	If $u=2,4$ then $\dim M$ is odd and $\rk L =0$. Hence $H_0 = B_k$, $K_0^- = B_{k+1}$ and $\bar K^+ = S^1,S^3$. If $l+=3$, then all groups are connected and $U_2 \subset N(H)$. If $u = 4$, this is not primitive, since $\bar K^+$ can be assumed to be on the diagonal of $U_2$. If $u=2$,  $K^+$ contains the center of $G_d$, which is not contained in $H$. Therefore $B_+$ is totally geodesic and $\dim G/H \le l_- + 2l_+$, by Lemma \ref{obstructions_lem_totgeod}. This is a contradiction, since $\dim G/H = 4k+4$. Let $l_+=1$ and assume $u =2$. In this case $K_0^-$ can only be extended by elements in $\U(k+1)$.  If $H$ commutes with the upper $2 \times 2$-block, then $N_0(H)/H = \U(2)$ acts with one dimensional singular isotropy groups, a contradiction to the Core-Weyl Lemma. Therefore we can assume $H$ does not commute with $U_{2}$. But then $N_0(H) = S(T^2 \cdot Z(\U(k))) \cdot B_k$, which contains $\bar K^+ = T^2$. But this is not primitive. Let $u = 4$. Note that both of $N(H)\cap K^\pm/H$ are at least one dimensional. Hence $\vert W \vert \le 4$ and we get from linear primitivity $8k +10 = \dim G/H \le 8k+8$, a contradiction. \\
	$\rk \bar H =2$: $u = 2,4$. If $\dim M$ is even, then $\rk L  = 1$. Since the action is essential $L = \SU(2)$ and $\bar K^+$ contains a diagonal $\SU(2)$ as a normal subgroup. Since $\rk \bar H = 2$, we have $l_+ = 2$. If $u=2$ this implies $K^+ = \Delta \SU(2) \cdot S^1 \cdot \SU(k-1)$, $H = T^2 \cdot \SU(k-1)$ and $K^- = S^1 \cdot \U(k+1)$, which is the diagram of $\bb C\bb P^{2k+1}$.\\
	Let $u = 4$. If $H$ contains an $S^3$ factor, it has to be contained in $\Sp(k+2)$ and diagonal in the upper $2 \times 2$ block by the Isotropy Lemma. But then it does not commute with another $S^3$, which is a contradiction. Therefore $H$ contains an $S^1$ factor, that acts trivially on $\bb S^{l_+}$ and hence is contained in $\Sp(k+2)$. Assume $S^1 = \diag(z^n,z^m,1, \ldots, 1)$, which acts with weights $2n$, $2m$ and $n\pm m$ on the upper $2 \times 2$ block and with weights $n$ and $m$ on the irreducible representations of $B_k$. By the Isotropy Lemma, we can assume $n = m = 1$. But the representation with weight $n + m$ is not equivalent to a representation in one of $K^\pm$ by conjugacy with an element of $N_0(H)$, contradicting the Isotropy Lemma.\\
	If $\dim M$ is odd, then $\rk L = 2$. Since $L = \SU(2) \times S^1$ is not essential $L = \Sp(2)$.    Because $\bar K^+$ and $\bar H$ have the same rank, $l_+$ must be even and since  $K^+$ projects onto $L$, $l_+ =4$. Therefore $K^+ = \Delta \Sp(2) \cdot \Sp(k)$, $H = \Delta \Sp(1)^2 \cdot \Sp(k)$ and $K^- = \Sp(1)^2 \cdot \Sp(k+1)$, which is the diagram of $\bb S^{8k +15}$.\\
	Now only the case $\rk H = 1$ is left. If $\dim M$ is even, then $u =2,4$, since otherwise $L =\SO(2)$, which is not possible for essential actions. Therefore we have $\rk L = 0$. Let $u = 2$. Since $K^- = \U(k)$ is a maximal subgroup of $\SU(k+2)$, $H$ is connected. By the Isotropy Lemma $\bar H = \diag(z^n,z^m,z^l, \ldots, z^l)$, with $m = \pm n$. If $m = n$, then  $N_0(H)/H = \SO(3)$ acts with trivial principal isotropy group, a contradiction to the Core-Weyl Lemma. Hence  $n = -m =1$. $l_+ = 3$ is not possible, since $\bar H$ is the maximal torus of the upper $2 \times 2$ block.  In the case $l_+ = 2$, we have $K^+ = \SU(2) \cdot \SU(k)$, $ H = S^1 \cdot \SU(k)$ and $K^- = \U(k+1)$, which is the diagram of $\bb H \bb P^{k+1}$.\\
	Let $u =4$. If $\bar H = \SU(2)$, then $K^- = \Sp(1)\cdot \Sp(k+1)$, which is maximal and hence all groups are connected. By the Isotropy Lemma $\bar H = \diag(a,a,1, \ldots,1)$. If $l _+ = 3$, the action is not primitive. If $l_+ =1$ the action is not linear primitve, since one of the three dimensional representations of in the upper $2 \times 2$ block is not contained in any of $K^\pm$. Hence we assume $\bar H = S^1 = \diag(z,z,1,\ldots,1)$ by the Isotropy Lemma. $l_+ = 1$ contradicts the Isotropy Lemma, since $\bar H$ has non trivial weights in the upper $2 \times 2$ block. Therefore $H$ is connected and $N_0(H)/H = \SO(3)$ acts with trivial principal isotropy group, a contradiction. \\
	Now let $\dim M$ be odd. Then $u=2,4$ and $\rk L =1$. Assume $\bar H=S^1$. If $u =4$ again, since $\bar H$ acts with non trivial weights on $U_{2}$, $l_+ =2,3$. Like before $N_0(H)/H =\U(2)$ acts with trivial principal isotropy group. Therefore $\bar K^+$ contains an $\SU(2)$ commuting with $\bar H$. This implies $l_+ = 3$ and $H \subset \Sp(k+2)$, contradicting the Isotropy Lemma.\\
	Now let $u =2$. If $L = \SO(2)$, $H_0,\,K_0^- \subset \SU(k+2)$, since the action is essential. Since $K_0^- = \U(k+1)$ is a maximal subgroup of $\SU(k+2)$, it can only be extended by elements in $L = \SO(2)$. Therefore either $N_0(H)/H = \SO(2) \times \SO(3)$ acts with trivial principal isotropy group, which is a contradiction to the Core-Weyl Lemma, or $N_0(H) = T^2\cdot B_k$, which is not primitive. If  $L = \SU(2)$ then $l_+ = 2,3$. In the first case $H = \Delta S^1 \cdot \SU(k)$, $K^- = S^1 \cdot \SU(k+1)$ and $K^+ = \Delta \SU(2)\cdot \SU(k)$, which is the diagram of $\bb S^{4k+7}$. Let $l_+ =3$. $H$ must project trivially to $L = \SU(2)$. Otherwise $N_0(H)/H = \U(2)$ acts with trivial principal isotropy group and $K^+$ must contain the $\SU(2)$ by the Core-Weyl Lemma, which is a contradiction. But then $K^+$ contains the central element $(-1,1)$ and by Lemma \ref{obstructions_lem_totgeod} we have $4k +6 = \dim G/H \le 2k+7$.\\
	Now assume $\bar H =\SU(2)$. Then $u = 4$ and $H_0 \subset \Sp(k+2)$. By the Isotropy Lemma, it must be diagonal in the upper $2 \times 2$ block. Since $\bar H$ does not commute with another $\SU(2)$ subgroup of $U_{2}$, $L = \SO(2)$ and $H$ acts trivially on $\bb S^{l_+}$, which contradicts the Isotropy Lemma.
\end{proof}

%% file: documents/even.tex
\section{The even dimensional case}\label{seceven}

In this section we will start with the classification of cohomogeneity one manifolds with quasipositive sectional curvature in even dimensions. The case of positive sectional curvature was first done by Verdiani in \cite{v1} and \cite{v2}. However, our proof is oriented to the proof of Grove, Wilking and Ziller in \cite{gwz} Section 14. Namely, we have the following result.

\begin{thm}\label{even_thm_class}
	Let $M$ be a compact even dimensional cohomogeneity one manifold. Suppose $M$ admits an invariant metric with quasipositive sectional curvature. Then $M$ is equivariantly diffeomorphic to a rank one symmetric space with a linear action.
\end{thm}

Note, that by the Rank Lemma, one of the principal isotropy groups has maximal rank in $G$. This will be very useful, since there is a classification of subgroups of maximal rank in Lie groups. Furthermore the Weyl group has at most four elements, if one of $l_\pm$ is greater than one. The proof is split up into the two cases where $G$ is simple or not. In the simple case the proof in Section 14 of \cite{gwz} just requires a few adjustments. For the non simple case a new proof is required. In both cases the statements of the results carry over almost verbatim. For the entire section let $M$ be an even dimensional cohomogeneity one $G$-manifold with an invariant metric of quasipositive curvature and group diagram $H \subset \lbrace K^-,K^+\rbrace \subset G$.
\subsection{$G$ is not simple}

\begin{lem}\label{even_prop_semisimple}
	If $G$ is not simple and the action of $G$ on $M$ is essential, then $M$ is $G$-equivariantly diffeomorphic to $\bb C \bb P^{2k-1}$ with the tensor product action of \emph{$\SU(2)\times \SU(k)$}.
\end{lem}

\begin{proof}
	We can assume $G=L_1 \times L_2$ and both $L_i$ have positive dimension. Since $\rk G = \rk K$ there exist subgroups $K_i \subseteq L_i$, such that $K_0^- = K_1 \times K_2$.  $K^-$ cannot contain a normal subgroup of $G$, because the action is essential, which implies, that $G$ cannot have a center of positive dimension and is therefore semisimple. Furthermore let $H_0 = (H_1 \cdot H_\Delta \cdot H_2)_0$ and $K_0^+=(K_1^+ \cdot K_\Delta^+\cdot K_2^+)_0$. By the classification of homogeneous Spheres, we can assume, that one of $K_i$, say $K_2$ acts transitively on the normal sphere $\bb S^{l_-}$ and that $H$ projects onto $K_1$. Hence we can assume, that the Weyl group element $w_-$ is represented by an element in $K_2$, which implies $w_-\proj_{L_1}(K^+)w_- = \proj_{L_1}(K^+)$. Since $\proj_{L_1}(K^-) = \proj_{L_1}(H) \subseteq \proj_{L_1}(K^+)$ the projection of $K^+$ to $L_1$ must be onto and therefore $K^+_1$ is a finite group, since the action is essential. Hence $K^+$ cannot have full rank and has corank $1$ by the Rank Lemma. Moreover $K^+_\Delta \cong L_1$ and $\rk(L_2) - \rk(K_2^+) = \rk(L_1) + \rk(L_2) - (\rk(K_\Delta^+) + \rk(K_2^+)) = 1$. Since $K_\Delta^+$ commutes with $K_2^+$ and the projection of $K^+_\Delta$ to $L_2$ has finite kernel, $L_1$ must have rank $1$ and is therefore isomorphic to $\SU(2)$.  $K_1$ is of maximal rank and the action is essential, therefore, $K_1 \cong S^1$ and  $H_\Delta \cong S^1$. Furthermore $l_+$ must be even since $K^+$ has corank $1$. $K_\Delta^+$ must act non trivially on $\bb S^{l_+}$ and hence by the classification of homogenous spheres $l_+ = 2$. Obviously $K^+_2$ acts trivially on $\bb S^{l_+}$, which implies $K^+_\Delta/H_\Delta \cong \bb S^{l_+}$ and $K^+_2 = H_2$. Therefore by primitivity the sphere $\bb S^{l_-} = K_2/H_2$ must be almost effective. Since $S^1 = K_1$ acts nontrivially on $\bb S^{l_-}$, the only possibilities are $\bb S^{l_-} \cong \U(n+1) / \U(n)$ or $\bb S^{l_+} \cong \U(1) \times \Sp(n+1)/\Delta \U(1) \cdot \Sp(n)$. Therefore up to covering $(K_2,H_2)$ is one of $(\SU(n+1), \SU(n))$, $(\U(n+1), \U(n))$, $(\Sp(n+1),\Sp(n))$, $(\U(1)\times \Sp(n+1), \Delta \U(1) \cdot \Sp(n))$.\\
	Suppose $L_2$ is not simple. Since $K_2$ is of maximal rank in $L_2$, $K_2$ cannot be simple and hence $L_2$ contains a rank one factor isomorphic to $\SU(2)$. This implies $L_2 = \SU(2) \times L^\prime_2$ and $K_2 = S^1 \times K^\prime_2$. Furthermore let $H_2^\prime = H\cap L_2^\prime$.  Obviously $K^\prime_2$ acts transitively on $\bb S^{l_-}$. By the same argument as above $K^+$ must project onto the $\SU(2)$-factor of $L_2$. Since $H_2$ commutes with $K^+_\Delta$, this means $H_2 = H_2^\prime$, forcing $H$ to have at least corank $2$, a contradiction. Therefore $L_2$ is simple.\\
	To determine the group diagram of the possible actions of $G$ on $M$, we thus have to find quadruples $(L_2,K_2,\SU(2)\cdot H_2, H_2)$, where $L_2$ is simple, $K_2$, $\SU(2)\cdot H_2$ have maximal rank in $L_2$, $H_2$ has corank one and acts spherical, and $(K_2,H_2)$ is one of the pairs above up to covering. Spherical subgroups of simple groups are classified in Table \ref{appendix_table_spherical}.\\
	First assume, that $H$ only consists of normal factors of rank $1$. This implies $\rk H_2 \le 2$. So we first consider the case $\rk H_2 = 0$. In this case $L_2 = \SU(2)$, $K^- = S^1 \times S^1$, $K_0^+ = \Delta \SU(2)$ and $H_0 = \Delta S^1$. If $H$ is connected, then $w_-$ can be represented by a central involution, normalizing $K^+$. $w_+$ is given by a diagonal Element normalizing $H$ and hence also normalizing $K^-$. Thus $\fr g = \fr k^+ + \fr k^-$ by linear primititvity and therefore $5 = \dim G/H \le l_+ + l_- = 3$, a contradiction.  Hence $H$ is not connected and we receive the group diagram of $\bb C \bb P^3$, with the tensor product action of $\SU(2) \times \SU(2)$. \\
	Now assume $\rk H_2 = 1$. Then $\rk L_2 = 2$ and hence $L_2$ is one of $\SU(3)$, $\Sp(2)$ or $G_2$.  \\
	$L_2 = \SU(3)$: $H_2$ cannot be three dimensional, since then $K_2 = \SU(3)$, and the action is not essential. Hence $H_2 = S^1$. Since $H_2$ commutes with the $\SU(2)$ factor of $K^+$, we can assume that $K^+$ projects to the maximal $\U(2)$ subgroup of $\SU(3)$, where the $\SU(2)$-factor is the lower block. $l_- =1$ contradicts primitivity, since $H$ projects to the maximal torus of $L_2$. Hence $l_- = 3$.  and $K_2 = \U(2)$. But this is the diagram of $\bb C \bb P^5$ with the tensor product action of $\SU(2) \times \SU(3)$.\\
	$L_2=\Sp(2)$:
	If $H_2$ is three dimensional, then $L_2$ must either contain $\SU(3)$ or $\Sp(2)$, but $\Sp(2)$ does not contain $\SU(3)$ as a subgroup and the latter case is not essential. Hence $H_2 = S^1$ and $l_-\le 3$. 
	%Hence $H_2 = S^1 = \lbrace \diag(z^n,z^m) \rbrace$ with $m$ and $n$ relatively prime. Since $H_2$ has to commute with some $\SU(2)$, either $m = 0$, $n =0$ or $m = n = 1$, where $m = 0$ and $n=0$ are the same case up to conjugating the groups. \\
	%	 	$n = 0$: Then $K^+_0$ projects to the upper $\Sp(1)$ block. And $K_2$ has to be the lower $\Sp(1)$ block. Hence the action is not primitive. \\
	%	 	$m=n=1$: Then the $\SU(2)$ factor of $K^+$ projects to $\SU(2)$ in $\Sp(2)$. 
	%		Since $H$ projects to the maximal torus of $\Sp(2)$, $l_-$ cannot be $1$. Hence $l_- = 3$ and $K_2 = \U(2)$ or $K_2 = \U(1) \cdot \Sp(1)$. 
	But then $11 = 13 -2 = \dim G/H \le 2(l_-+l_+) \le 10$ by linear primitivity, a contradiction. \\
	$L_2 = G_2$: If $H_2$ is one dimensional, then again $l_- \le 3$ and linear primitivity rules out this case. Hence $H_2 = \SU(2)$ and $K_2 = \SU(3)$, since no $\Sp(2)$ is contained in $G_2$. Also $K^+ = \Delta \SU(2) \cdot \SU(2)$ projects to $\SO(4) \subseteq G_2$. As in the proof of \cite{gwz} Proposition 14.2 the isotropy representation on the tangent space $T^+$ of the singular orbit $B_+$ decomposes into an eight dimensional and a three dimensional irreducible representation. The representation on $S^2 T^+$ splits into two trivial, two five dimensional and into representations on which $G_2 \cap K^+$ acts nontrivially. Hence $B_+ $ is totally geodesic. Lemma \ref{obstructions_lem_totgeod} implies $13=\dim G/H \le 2\cdot l_++l_- =9 $, a contradiction. \\
	Now assume, that $\rk H_2 = 2$. Then $\rk L_2 = 3$ and since $L_2$ is simple, it must be one of $\Spin(7)$, $\SU(4)$, $\Sp(3)$. 
	%If $H$ is simple, then it must be one of $\Sp(2)$ or $\SU(3)$. $\Spin(7)$ does not contain $\Sp(3)$ or $\SU(3) \cdot \SU(2)$ and the other two cases are not essential. 
	Because of primitivity $H_2$ cannot be a torus. Hence it contains at least one three dimensional normal subgroup. It cannot contain more than one three dimensional group, since it must be one of the spheres above. Hence $H_2 = S^1 \cdot \SU(2)$ and $\dim H = 5$ and $l_- \le 7$.\\
	$L_2 = \Spin(7), \Sp(3)$: By linear primitivity $19 = 24 -5 = \dim G/H \le 2(l_+ + l_-) \le 2 (7+2) = 18$. \\
	$L_2 = \SU(4)$: In this case $K_2 = \U(3)$ and $\SU(2)$ is a lower block in $L_2$. Only the upper $\SU(2)$ block commutes with this. Hence the $\SU(2)$ factor of $K^+$ projects to the upper block and $K^+ = \Delta \SU(2) \cdot \U(1) \cdot \SU(2)$, $K^- = S^1 \times \U(3)$, $H = \Delta S^1 \cdot \U(1) \cdot \SU(2)$, which is the diagram of the action of $\SU(2) \SU(4)$ on $\bb C \bb P^{5}$.  \\
	From now on we can assume, that a simple factor of $H_2$ has at least rank $2$.\\
	$(K_2, H_2) = (\SU(n+1),\SU(n))$: $n \ge 3$. Since the corank of $H_2$ in $L_2$ is one and the isotropy action $L_2/H_2$ has to be spherical, $L_2$ has to be one of $\SU(n+1)$ or $\Spin(k)$ for $k = 6,7,8,9$. The first case is not essential. If $L_2$ is one of $\Spin(6)$ or $\Spin(7)$, then $H_2=\SU(3)$ and $K_2 = \SU(4)$, but then $L_2 = \Spin(6)$ is not essential. If $L_2 = \Spin (7)$ then $\SU(2)\cdot \SU(3)$ is not subgroup of $L_2$. Similarly $\SU(5)$ is not contained in $\Spin(8)$ or $\Spin(9)$.\\
	$(K_2,H_2)=(\U(n+1),\U(n))$: $n \ge 3$. Since the simple factor of $H_2$ is spherical in $L_2$, the only possibilities for $L_2$ are $\SU(n+2)$ or $\Spin(k)$ for $k = 8,9,10,11$. The first case is handled by the Block Theorem. So let $L_2 = \Spin (k)$ for $k = 8,9,10,11$. If $k = 8,9$, then $n = 3$ and if $k = 10,11$, then $n = 4$. If $k = 8,9$, then by linear primitivity $21 \le \dim G/H \le 2 (l_-+l_+) = 2 \cdot 9 = 18$ and if $k = 10,11$, then $31 \le \dim G/H \le 22$. \\
	$(K_2,H_2) = (\Sp(n+1),\Sp(n))$: $n \ge 2$. Since $\Sp(n)$ must be spherical in $L_2$ with corank one, $L_2$ must be one of $\Sp (n+1)$ or $\SU(4)$. The first case is not essential and the second case does not contain $\Sp(3)$.\\
	$(K_2,H_2) = (\U(1) \times \Sp(n+1), \Delta \U(1) \cdot \Sp(n))$: $n \ge 2$: In this case $L_2$ must be one of $\Sp(n+2)$, $\SU(5)$, $F_4$. The first case is handled by the Block Theorem, $\SU(5)$ does  not contain $\U(1) \times \Sp(3)$. Hence  $L_2 = F_4$, but then by linear primitivity $43 = 55 -12 = \dim G/H \le 2(l_-+l_+) = 2(13+2) = 30$, a contradiction. 
	%		If the $\SU(2)$ factor in $H_2$, contains the central element, then the action is effectively given by an $\SU(2) \times \SO(7)$ action and $H$ contains a $3 \times 3$ block, which is not possible by the block theorem. The only other spherical $\SU(2)$ in $\Spin(7)$ is given by one of the normal factors contained in $\Spin(4)$.  Let $j$ be the central involution of this $\SU(2)$ factor. Then $C(j) = \SU(2) \times \Spin(3)\cdot \Spin(4)$ acts by cohomogeneity one on $M^j$. Then $\Delta\SU(2)\cdot H_2 = K^+ = K^+\cap C(j)$. Furthermore $K^- = S^1 \times S^1 \cdot \Sp(2)$ or $K^- = S^1 \times U(3) $. In both cases 
\end{proof}

\subsection{$G$ simple}

 For the simple case we distinguish between the two cases, where  $H$ has a simple normal factor of rank at least two or only factors of rank one, and start with the latter case. 
 \begin{prop}
 	Let $G$ be simple. Assume, that all simple normal factors of $H$ have rank one. If the action is essential, then $(M,G)$ is one of the following: \emph{$(\bb C \bb P^6, G_2)$, $(\bb C \bb P^9, \SU(5))$, $(\bb H \bb P^2,\SU(3))$, $(\bb H \bb P^3, \SU(4))$ or $(\text{Ca}\bb P^2, \Sp(3))$} with one of the actions given in Table \ref{appendix_table_cohom_one_even}.
 \end{prop}
 \begin{proof}
 	We will follow the proof of Proposition 14.3 in \cite{gwz}, which is the analogue of this result for positive curvature and only consider the cases, where the proof has to be adjusted. Since all normal factors of $H$ are one and three dimensional $l_\pm = 1,2,3,4,5,7$ and at least one is odd. The Weyl group has order at most $4$, and the the order is $2$ if $l_\pm$ are both one of $3,5,7$. \\
 	The case $\rk G \le 2$ works like in \cite{gwz} except if $G=\SU(3)$, $H=S^1$, $K^- = \U(2)$ and $K^+ = \SO(3)$. Let $j$ be the involution in $H$, then $\tilde G = C(j) = \U(2)$ acts by cohomogeneity one on $M^j$ with isotropy groups $\tilde K^+ = K^+\cap C(j) = \O(2)$, $\tilde K^- = K^- \cap C(j) = T^2$ and $\tilde H = H$.  $\tilde H$ has a non trivial $2$-dimensional representation in $\U(2)$ and acts trivially on $T^2/H = S^1$. Hence $\tilde B_- = \U(2) / T^2$ is totally geodesic. By Lemma \ref{obstructions_lem_totgeod}, we have $3 = \dim \tilde G /\tilde H \le 2 \cdot \tilde l_- +\tilde l_+ =2$.\\
 	%the connectedness lemma, the inclusion of this singular orbit is one-connected. Hence $M^j$ is simply connected, a contradiction, since $\O(2)/\SO(2) = \bb S^0$.\\
 	We now proceed as in \cite{gwz}. First remember some arguments therein: Let $\rk G \ge 3$. $\dim H \le 3 \rk H = 3(\rk G -1)$ and therefore $\dim G - 3\rk G \le \dim G/H -3$. The lower Weyl group bound implies $\dim G/H \le 2(7+4) = 22$ and hence $\dim G - 3 \rk G \le 19$.\\
 	Suppose there is an orbit of codimension $8$. Then $K^-$ contains $\Sp(2)$ as a normal factor and hence $G$ must be one of $\Spin(7)$ or $\Sp(3)$. The latter case works in the same way as in \cite{gwz}. Hence let $G = \Spin(7)$. Note that $H = S^1 \cdot \Sp(1)$, since $\Spin(7)$ does not contain $\Sp(1) \cdot \Sp(2)$. The central involution $j$ of $\Spin(7)$ is contained in $\Spin(5) = \Sp(2) \subseteq K^-$. Since $l_-=7$, it cannot be contained in $H$ and hence $B_-$ is totally geodesic. By Lemma \ref{obstructions_lem_totgeod}, we get $17 = \dim G/H \le 14 + l_+$. Hence $l_+ \ge 3$. If $l_+$ is odd, then $\dim G/H \le 2 \cdot 7 = 14$, a contradiction. Hence we assume $l_+ = 4$, but then $\bb S^{l_+} = \SO(5)/\SO(4)$, which is not possible since $ H = S^1 \cdot \Sp(1)$.	The other cases work in the same way as in \cite{gwz}.
 \end{proof}
 We end the evendimensional classification with the case, where $H$ contains a simple normal subgroup of rank at least $2$:
 \begin{prop}
 	Let $G$ be simple, such that $H$ a simple normal subgroup of rank at least two. If the action is essential, then $(M,G)$ is one of the following pairs: \emph{$(\bb C\bb P ^{n-1}, \SO(n))$, $(\bb H \bb P^n, \SU(n))$, $(\bb C \bb P ^{15},\Spin(10))$, $(\bb S^{14}, \Spin(7))$ or $(\bb C \bb P^7, \Spin(7))$} with one of the actions given in Table \ref{appendix_table_cohom_one_even}.
 \end{prop}
 \begin{proof}
 	We follow the proof of Proposition 14.4 of \cite{gwz}. Let $H^\prime$ denote the simple normal factor of $H$ with rank at least $2$. By Lemma \ref{obstructions_lem_simple} there is only one such subgroup. If $G$ is one of $\Sp(n)$ or $\SU(n)$, then the cases are handled as in \cite{gwz}. If $G = \Spin(k)$ then the only case, that does not work as in \cite{gwz} is $H^\prime = \SU(3)$. Suppose $k \ge 8$. Then $\rk H \ge 3$ and thus not all $6$ dimensional representations can degenerate in $G_2$. Hence we can assume that $K^-$ contains an $\SU(4)$ factor and $l_- = 7$. Since all other simple normal factors of $H$ are at most one dimensional, we have: $\dim H \le 3(\rk H -2) + 8 = 3(\rk G -3) +8 = 3 \rk G -1$. If $l_+$ is odd and at least $3$, the lower Weyl group bound implies $\dim G - 3\rk G \le 13$, which is not possible, since $k \ge 8$. If $l_+$ is even or $1$, then $\dim G - 3 \rk G \le 25$, which is only true for $k \le 9$. If $k = 9$, then $\rk H = 3$. Since $l_+$ is either even or $1$, $H = S^1 \cdot \SU(3)$. Hence the lower Weyl group bound implies $27 = \dim G/H \le 2(l_-+l_+)\le 2 (7+6) = 26$. Therefore $k = 8$. As above $H = S^1 \cdot \SU(3)$. The lower Weyl group bound gives: $19=28-9 = \dim G/H \le 14 + 2 l_+$. Therefore $l_+ \ge 4$, but $l_+ = 4$ is not possible, since then $\rk H \ge 4$. Therefore $l_+ = 6$ and $\bb S^{l_+} = G_2 / \SU(3)$. Then $K^+ = S^1 \cdot G_2$, but $\Spin(8)$ does not contain this group.  The case $k = 6$ is handled as in \cite{gwz} and so is $k = 7$, except if there is no isotropy group isomorphic to $\SU(4)$. Then $K_0^- = G_2$, $H_0 = \SU(3)$ and $K^+ = S^1 \cdot \SU(3)$. If the center of $\Spin(7)$ is not contained in $H$, then the singular orbit $B^+ = \Spin(7) / K^+ $ is totally geodesic in $M$. By Lemma \ref{obstructions_lem_totgeod}, we have $13 = \dim G/H \le 2 l_+ + l_- = 2+6 = 8$, a contradiction. Hence we have recovered the linear action of $\Spin(7)$ on $\bb C \bb P^7$. The exceptional cases can be ruled out as in \cite{gwz}. This finishes the proof.
 \end{proof} 

%% file: documents/odd.tex
\section{The odd dimensional case}\label{secodd}

The classification in odd dimensions splits up into several subcases. For the entire section, let $M$ be an odd dimensional quasipositively curved cohomogeneity one $G$-manifold with group diagram $H \subset \lbrace K^-,K^+\rbrace \subset G$, such that the action is essential.\\

 Suppose that $H$ has full rank in $G$ and $G = L_1 \times L_2$ is not simple. Since $H$ has full rank $l_\pm$ are both even, $H = H_1 \times H_2$ and $K^\pm = K^\pm_1 \times K^\pm_2$. Therefore $G$ is semisimple. By the classification of homogeneous spheres for both $K^\pm$ at least one $K_i^\pm$ acts trivially. But then by primitivity at least one factor of $G$ must be contained in $K^\pm$ and the action is not essential. Hence $G$ is simple. Now the equal rank cases can be handled as in \cite{gwz} Section 5. From now on we can assume, that $H$ has corank two.

\subsection{$G$ is not semisimple}

\begin{thm}
	Suppose $G$ is not semisimple. Then either $G=S^1\cdot L$ with \emph{$L \in \lbrace \SO(n),$ $\Spin(7),$ $G_2\rbrace$} and the action is the tensor product action on $\bb S^{2n-1}$, $\bb S^{15}$ or $\bb S^{13}$, respectively, or otherwise \emph{$G = \U(2)\cdot \SU(2)$} with the tensor product action on $\bb S^7$. 
\end{thm}
\begin{proof}
	We can assume $G=S^1 \times L$. $H \cap S^1$ is in the kernel of the action and can  be assumed to be trivial. $H$ does not project onto $S^1$, because the action is essential. If $S^1$ acts freely, then the same proof as in \cite{gwz} Proposition 6.1 applies. Hence we assume that $S^1$ does not act freely and therefore $K^- \cap S^1 \neq \lbrace e \rbrace$. This implies that $B_- = G/K^-$ is totally geodesic. If $G/K^-$ has finite fundamental group then $K^-$ must project to $S^1$. Since $H$ does not project to $S^1$, this implies $l_- = 1$ and hence $G/K^-$ is totally geodesic of codimension $2$ and therefore $M$ is a homology sphere by the Connectedness Lemma. If $l_+$ is at least $2$, then the inclusion of $B_-$ in $M$ is $2$-connected and hence $B_-$ is simply connected. Therefore we only have to consider the case: $l_+ =1$, $K^+ \cap S^1 = e$. Furthermore we can assume, that $K^-$ does not project onto $S^1$ and hence $H_0, K^-_0 \subset L$. $H_0$ is a normal factor of $K^+ = \Delta S^1 \cdot H_0$. By Lemma \ref{obstructions_lem_totgeod}, we have $\dim G/H \le 2\cdot l_- + 1$.
	Since $\bb S^{l_-}$ is almost effective, the strategy is to compute the possible group diagrams by going through the possible cases for $K_0^-/H_0$:\\
	First assume, that $l_-$ is even. Then $K_0/H_0$ = $\SO(2n+1)/\SO(2n)$, $G_2/\SU(3)$. Consider $\bb S^{l_-} = G_2/\SU(3)$. Since $\rk H \ge 2$ and $H$ has corank $1$ in $L$, there is a simple factor $L^\prime$ in $L$ containing $H_0 = \SU(3)$. Since $L^\prime/H_0$ is spherical $(L^\prime,H_0)= (\Spin(6), \SU(3))$, $ (\Spin(7),\SU(3))$, $ (G_2, \SU(3))$. But $\Spin(6)$ does not contain $G_2$ and $L^\prime = G_2$ will not be essential. Hence $L = L^\prime = \Spin (7)$. But then $14 = \dim G/H \le 2 \cdot 6 + 1 = 13$, a contradiction. \\
	Now assume $K^-_0 / H_0 = \SO(2n+1)/\SO(2n)$. First assume $n \ge 3$. Since $H_0$ is simple of rank at least $3$, $L$ contains a simple factor $L^\prime$ containing $H_0$. Hence $(L^\prime, H_0) = (\Spin(k), \Spin(2n))$, where the Block Theorem applies, or $(L^\prime, H_0)$ is one of  $(F_4, \Spin(6))$, $(F_4,\Spin(8))$. But then either $37 \le \dim G/H \le 14$ or $24 \le \dim G/H \le 18$, which are both contradictions. Hence let $n = 1$. Then $\rk L = 2$ and $\dim H =1$. Hence $\dim L \le 2 l_- +1 = 5$. But then $L = S^1 \times  S^3$ and this is not essential. Let $n = 2$. Then $l_- = 4$, $\rk L = 3$ and  $\dim H = 6$. Hence $\dim L \le 8 + 6 = 14$ and the only possible choice for $L$ is $X \times \Sp(2)$, where $X$ is a compact rank one Lie group. But this is not essential.\\
%	\begin{center}
%		\begin{tabular}{||c|c|c||}
%			\hline
%			Case & $K_0^-/H_0$ & $l^-$ \\
%			\hline \hline
%			$1$ & $\SO(2n+2)/\SO(2n+1)$ & $2n+1$\\
%			\hline
%			$2$ & $\SU(n+1)/\SU(n)$ & $2n+1$\\
%			\hline
%			$3$ & $\U(n+1)/\U(n)$ & $2n+1$\\
%			\hline
%			$4$ & $\Sp(n+1)/\Sp(n)$ & $4n+3$\\
%			\hline
%			$5$ & $\U(1)\times \Sp(n+1)/\Delta \U(1)\cdot \Sp(n)$ & $4n+3$\\
%			\hline
%			$6$ & $\Sp(1)\times \Sp(n+1)/\Delta \Sp(1)\cdot \Sp(n)$ & $4n+3$\\
%			\hline
%			$7$ & $\Spin(9)/\Spin(7)$ & $15$\\
%			\hline
%			$8$ & $\Spin(7)/G_2$ & $7$\\
%			\hline
%		\end{tabular}
%	\end{center}
	We can now assume that $l_-$ is odd and hence $K^-$ has full rank in $L$.  So if $L$ is not simple, then $K^-$ is one of the non simple groups acting transitively on a sphere and $L = X \times L^\prime$, where $X$ is a rank one group. Since the action is essential $X = \SU(2)$ and $K^-$ contains the maximal torus of $X$. We can assume, that $w_-$ is represented by an element of $L^\prime$. By linear primitivity $K^+$ projects to $X$. But then $K^+$ has corank 2 or contains $X$, a contradiction. From now on we can assume, that $L$ is simple. \\
	$K_0^-/H_0 = \SO(2n+2)/\SO(2n+1)$: Let $n \ge 2$. The options for $(L,H_0)$ are $(\SO(k), \SO(2n+1))$,  $(\SO(8), \Spin(7))$, $(\SO(9), \Spin(7))$, $(F_4, \Spin(7))$. The first case is handled by the Block Theorem. In the second case the action is not essential. The third and fourth case are ruled out by the dimensional bound on $L$: $\dim L \le 14 +21 = 35$.
	If $n =1$, then $\rk L = 2$ and $\dim H = 3$. Hence $\dim L \le 6 + 3 = 9$. Therefore $L = \SU(3)$, which contains no $\SO(4)$.\\
	$K^-_0/H_0 = \SU(n+1)/\SU(n)$: Assume $n \ge 3$. Then $(L,H_0)$ is one of $(\SU(n+1), \SU(n))$, $(\Spin(6), \SU(3))$, $(\Spin (7),\SU(3))$, $(\Spin(8),\SU(4))$, $(\Spin (9), \SU(4))$. The first two cases are not essential and the last two cases do not contain $\SU(5)$ as a subgroup. Hence the case $(L,H_0) = (\Spin(7), \SU(3))$ is left. Then $K^-_0 = \Spin(6)$ and $K^+ = \Delta \SO(2) \cdot \SU(3)$. But the only $S^1$ commuting with $\SU(3)$ is contained in $\Spin(6)$ and hence the action is not primitive.
	If $n = 2$, then $\rk L = 2$, $l_- = 5$, $\dim H = 3$. Therefore $\dim L \le 10+3 = 13$ and $L = \SU(3)$, $\Sp(2)$.  But the first case is not essential and the second case does not contain $\SU(3)$.
	$n = 1$ implies $\dim H = 0$ and $l_- = 3$ with $\rk L = 1$. Hence $L = \SU(2)$, which is not essential. \\
	$K_0^-/H_0 = \U(n+1)/\U(n)$: Assume $n \ge 3$. $H_0 = S^1 \cdot H^\prime = S^1 \cdot \SU(n)$. Then $(L,H^\prime) = (\SU(n+2), \SU(n))$, $(\Spin(k),\SU(l))$, with $(k,l) = (8,3),$ $(9,3),$ $(10,4),$ or $(11,4)$, $(F_4,\SU(3))$. In the first case the Block Theorem applies. If $H^\prime = \SU(3)$, then $l_- = 7$, $\dim H = 9$ and thus $\dim L \le 14 + 9 = 23$ and all groups have greater dimension. Similarly for $H^\prime = \SU(4)$ we have $l_- = 9$ and $\dim H = 16$. Then $\dim L \le 18 +16 = 34$ and all groups have greater dimension. Now let $n = 2$. Then $\dim H = 4$, $\rk L = 3$ and $l_- = 5$. Hence $\dim L \le 14$. But there is no simple Lie group with this dimensional bound of rank $3$. Assume $n = 1$. Then $\dim H = 1$, $l_- = 3$, $\rk L =2$ and $\dim L \le 7$. But there is no simple Lie group of rank $2$ with this dimensional bound. \\
	$K^-_0/H_0 = \Sp(n+1)/\Sp(n)$: Assume $n \ge 2$. Then $(L,H_0)$ is one of $(\Sp(n+1),\Sp(n))$, $(\SU(4), \Sp (2))$, $(\Spin (7), \Sp(2))$. In all cases the Block Theorem applies since the last two cases are effectively $(\SO(6),\SO(5))$ and $(\SO(7),\SO(5))$. Assume $n = 1$. Then $\dim H = 3$, $\rk L = 2$ and $l_- =7$. The only rank two group containing $\Sp(2)$ as a subgroup is $\Sp(2)$ and hence the action is not essential.\\
	$K^-_0/H_0 = X \times \Sp(n+1)/\Delta X \cdot \Sp(n)$ with $X \in \lbrace S^1, S^3 \rbrace$:  Assume $n \ge 2$. Then $(L,H^\prime)$ is one of $(\Sp(n+2), \Sp(n))$, $(\SU(5), \Sp(2))$, $(\Spin(8), \Sp(2))$ or $(\Spin(9), \Sp(2))$. The last three cases cannot contain $X \times \Sp(3)$ as subgroups. In the first case the Block Theorem applies. If $n = 1$, then $\rk L = 3$, the dimension of $H$ is $4$ or $6$ and $l_- = 7$. Hence $\dim L \le 14+6 = 20$. But there is no simple Lie group of rank $3$ with this property, except $\SU(4)$, which does not contain $\Sp(3)$.\\
	$K^-_0/H_0 = \Spin(9)/\Spin(7)$: $\rk L  = 4$, $l_- = 15$ and $\dim H = 21$. Hence $\dim L \le 51$. Therefore $(L,H_0) = (\Spin(9), \Spin(7))$, $(\Spin(8), \Spin(7))$. The first case is not essential and the second does not contain $\Spin(9)$. \\
	$K^-_0/H_0 = \Spin(7)/G_2$: $l_- = 7$, $\rk L = 3$ and $\dim H = 14$. Hence $\dim L \le 28$. $(L,H_0) = (\Spin (7), G_2)$, which is not essential.
\end{proof}

\subsection{$G$ semisimple of rank two}

	From now on let $G$ have rank two. We will only consider the non simple case $G = S^3 \times S^3$, since in the simple case the proof of \cite{gwz} Proposition 7.3  applies.  For an imaginary unit quanternion $x$ we denote the group $C^x_{p,q} = \lbrace (e^{x \cdot p \cdot \theta}, e^{x \cdot q \cdot \theta}) \vert \theta \in \bb R \rbrace\cong S^1$.
\begin{prop} \label{odd_prop_s3s3}
	Let	$G=S^3 \times S^3$. If the action is essential, then $M$ is one of the Eschenburg spaces $E_p$, $p \ge 0$, $P_k$, $k \ge 1$, the Berger Space $B^7$ or $Q_k$, $k \ge 1$, with one of the actions described in Table \ref{appendix_table_odd_cohom_one}. 
\end{prop}

%The Proposition easily follows from the following Lemma, which is the analogue of Lemma 7.2. in \cite{gwz} for quasipositive curvature. We denote the groups in the effective picture with a bar.
%\begin{lem} \label{odd_lem_s3s3}
%	Under the conditions of proposition \ref{odd_prop_s3s3} one of the following holds
%	\begin{enumerate}[label = (\alph*), topsep=0pt,itemsep=-1ex,partopsep=1ex,parsep=1ex]
%		\item $\bar H = 1$, $\bar K^- \cong S^3$, $\bar K^+ \cong S^1$. In $S^3 \times S^3$: $H = \bb Z_2$, $K^- = \Delta S^3 \cdot H$, $K^+ = C^i_{(p,p+1)}$, $p \ge 0$.
%		\item $\bar H = \bb Z_2$, \emph{$\bar K ^- \cong \SO(2)$}, \emph{$\bar K^+ \cong \O(2)$}. In $S^3 \times S^3$: $H = \bb Z_4 \oplus \bb Z_2$, $K^- = C^i_{(1,1)} \cdot H$, $K^+ = C^j_{(p,p+1)} \cdot H$, $p \ge 1$
%		\item $\bar H = \bb Z_2 \oplus \bb Z_2$,  \emph{$ \bar K^- \cong \bar K^+ \cong \O(2)$}. In $S^3 \times S^3$: $H = Q$, $K^- = C^i_{(1,1)} \cdot H$, $K^+ = C^j_{(p,p+2)} \cdot H$, $p \ge 1$ odd, or the same picture with slopes $\lbrace (3,1), (1,3)\rbrace$. 
%	\end{enumerate}
%\end{lem}
\begin{proof} Since the action is not necessarily effective, we will denote groups in the effective action by $\bar G$, $\bar K^\pm$ and $\bar H$. $H$ is finite and therefore $l_\pm \in \lbrace 1,3\rbrace$. 
As in \cite{gwz} Lemma 7.2, we can rule out the case $l_-=l_+=3$ by primitivity.\\
	Hence at least one of the singular orbits, say $B_-$, has codimension $2$. Therefore $K_0^- = C^i_{(p,q)}$, with $p$ and $q$ relatively prime. Note that, unlike in \cite{gwz}, $p = 0$ is allowed here. $T^- \coloneqq T_{p_-}B_- = W_0 \oplus W_1 \oplus W_2$, which are invariant subspaces of the $K_0^-$ action on $T^-$ with weight $0$ on $W_0$, $2p$ on $W_1$ and $2q$ on $W_2$. From this we can compute the weights of the $K_0^-$ action on 
	\begin{align*}
		S^2(T^-) = S^2(W_0) \oplus W_0 \otimes W_1 \oplus W_0 \otimes W_2 \oplus S^2(W_1) \oplus W_1 \otimes W_2 \oplus S^2(W_2)
	\end{align*}
	The weights are $0$; $2p$; $2q$; $4p$, $0$; $2p \pm 2q$; $4q$, $0$.\\
	Let $l_+ = 3$. Then $K^-$ is connected and $K_0^+ \cong \Delta S^3$. Note, that $B_+$ is totally geodesic, since the central element in $K_0^+$ is also central in $G$. If $H$ is connected, then the normal weight at $B_-$ is $1$ and $B_-$ is totally geodesic, since all weights on $S^2(T^-)$ are even. But this contradicts the Partial Frankel Lemma, since $\dim B_+ = 3$ and $\dim B_- = 5$. Hence $H = \bb Z_2$. Note, that $H$ is central and thus $N(H) = G$. Therefore we can assume, that $K_0^+ = \lbrace (a,a) \vert \, a \in S^3 \rbrace$ and $K^- = C^i_{(p,q)}$ 	with $0 \le p \le q$ relatively prime integers. The normal weight at $B_-$ is $2$. Since $B_+$ is three dimensional, $\II^-$ cannot vanish on $S^2(W_1 \oplus W_2)$, since $W_1 \oplus W_2$ is four dimensional and this would be a contradiction to the Partial Frankel Lemma. Therefore one of the weights of the $K_0^-$ action on $S^1(W_1\oplus W_2)$ must be $2$ and we get $q = p+1$. Thus we have the diagram of the $S^3 \times S^3$ action on $E^7_p$.\\
	Now we assume, that both $l_\pm = 1$. Since $\bar H$ only consists of elements of order $2$ by Lemma \ref{lem_weylgroup_codim2}, $H$ can only consist of elements of order $2$ or $4$. Hence the normal weights can only be $2$ or $4$. We assume $K^\pm_0 \cong C^i_{(p_\pm, q_\pm)}$. Due to primitivity none of $p_\pm$ and $q_\pm$ can be zero.
	First we claim that $H$ cannot contain an element $h$ of the form $(a,\pm 1)$ or $(\pm 1, a)$ with $a$, unless $a$ is central: Assume $H$ contains such an element. Then $N(h)_0 = S^1 \times S^3$ or $S^3 \times S^1$ and hence $M^h$ would be a totally geodesic submanifold of codimension $2$. Therefore $M^h$ has dimension $5$. Since none of $p_\pm$ and $q_\pm$ is zero, the action is still essential. But then up to covering it must the the tensor product action of $\SO(2) \times \SO(3)$ on $\bb S^5$, which implies that one of $p_\pm$, $q_\pm$ must be zero, a contradiction. By the same reasoning as in the proof of \cite{gwz} Lemma 7.2, we can now show	$1 = \min\lbrace \vert q_+ \vert, \vert q_-\vert\rbrace  = \min\lbrace \vert p_+ \vert, \vert p_-\vert \rbrace$. The proof therein still works, since it relies on Petrunin's generalization of Frankels Theorem \cite{petr} applied to $M/(S^3\times 1)$, which only needs that a shortest geodesic between the projections of $B_\pm$ intersects an open set of positive curvature.\\
	$\bar H = \bb Z_2$: By Lemma \ref{lem_weylgroup_codim2} we can assume, that $\bar K^- = \SO(2)$ and $\bar K^+ = \O(2)$. The nontrivial element $\bar h$ is in the second component of $\bar K ^+$. Let $h$ denote an element in $H$, whose image is $\bar h$ in the effective picture. $h = (h_1,h_2)$, with $h_1,h_2$ being imaginary unit quanternions. As in \cite{gwz} Lemma 7.2 
%	$\bar h$ acts trivially on $\bb S^{l_-}$ and by reflection on $\bb S^{l_+}$. Hence $h$ does the same and we can assume $K_0^- = C^i_{(p_-,q_-)}$ with $0 \le p_-\le q_-$ relatively prime and $h = (i, \pm i)$. Furthermore $p_-$ and $q_-$ are both odd. Since conjugation by $h$ must preserve $K_0^+$ and induces a reflection on it, 
	we can assume, that $K_0^- = C^i_{(p_-,q_-)}$ with $0 \le p_-\le q_-$ relatively prime and both odd, $h = (i, \pm i)$ and after possible conjugating with an element in $N(H)$, that $K_0^+ = C^j_{(p_+,q_+)}$ with $p_+$ and $q_+$ being positive relatively prime integers. \\
	There are now two possible cases for $H$. Either $H = \bb Z_4 = \lbrace  \pm (1,1), \pm h \rbrace $ or $H = \bb Z_4 \oplus \bb Z_2 = \lbrace (\pm 1, \pm 1), (\pm i, \pm i) \rbrace$. \\
	$H = \bb Z_4$: If one of $p_+$ and $q_+$ is even, then $K_0^+\cap \bb Z_4 = \lbrace e \rbrace$. Hence the normal weight is $1$. Since all weights in $S^2(T^+)$ are even, $B_+$ is totally geodesic of codimension $2$ and hence $M$ is a sphere by the Connectedness Lemma. Therefore we assume both $p_+$ and $q_+$ to be odd and hence the normal weight is $2$. But since $p_+$ and $q_+$ are both odd $\II^+\vert_{S^2(W_1 \oplus W_2)} = 0$. $\II^-\vert_{S^2(W_0 \oplus W_i)}$ for $i=1,2$ and $\II^-\vert_{S^2(W_1\oplus W_2)}$ cannot vanish by the Partial Frankel Lemma. If the normal weight at $B_-$ is $2$, then both $p_-=q_-=1$, since otherwise $\II^-\vert_{S^2(W_0 \oplus W_i)}=0$ for both $i$. But then $\II^-\vert_{S^2(W_1\oplus W_2)} =0$. If the normal weight at $B_-$ is $4$, then both $p_-$ and $q_-$ must be $2$, but they are relatively prime, which rules out this case.\\
	$H = \bb Z_2 \oplus \bb Z_4$: We have $H\cap K_0^- = \bb Z_4$ and $H\cap K_0^+ = \bb Z_2$. If $p_- \ge 2$, then $(p_+,q_+) = (1,1)$ and $q_--p_- = 2$, since otherwise $B_-$ is totally geodesic. But then $\II^+\vert_{S^2(W_1^+ \oplus W_2^+)} =0$ and $\II^-\vert_{ S^2(W_0^-\oplus W_2^-) }= 0$, a contradiction to the Partial Frankel Lemma. 
	Hence let $p_- = 1$ and therefore $\II^-\vert_{S^2(W_0^- \oplus W_1^-)} = 0$. If $\vert p_+ - q_+\vert \neq 1$, then $\II^+\vert_{S^2(W_1^+\oplus W_2^+)} = 0$, contradicting the Partial Frankel Lemma. Therefore $\vert p_+-q_+\vert = 1$. If $q_+ \ge 2$, then $q_- = 1$ and up to interchanging $(p_+, q_+) = (p,p+1)$ for $p \ge 1$, which is the diagram of $Q_p$. Hence we assume $q_+ = 1$. Then $p _+ =2$ and $\II^+\vert_{S^2(W_0^+ \oplus W_1^+)}=0$. Since $q_- = 1$ is the former case, we assume $q_-\ge 3$. If $q_- - p_- \neq2$, we get a contradiction to the Partial Frankel Lemma by the following: 
	Pick $v \in W_1^- \backslash \lbrace 0 \rbrace$ and define $W_v \coloneqq W_0^- \oplus W_2^- \oplus \bb R v$. Note that $W_v$ is a four dimensional subspace of $T^-$ and each $w \in W_v$ has the form $w = w_0 + w_2 + b v$ for some $b \in \bb R$ and $w_i \in W_i^-$. Now by parallel transporting $W_0^+\oplus W_1^+$ along $c$ to $T_{p_-}M$, we see that the intersection with each $W_v$ is at least one dimensional.  Let $0 \neq w_v $ be an element in the intersection. Then: 
	\begin{align*}
			\II^-(w_v,w_v) &= \II^-(w_0,w_0) + \II^-(w_2,w_2) + b^2\II^-(v,v) \\&+ 2\II^-(w_0,w_2) + 2b\II^-(w_2,v) + 2ab\II^-(w_0,v) \\&= b^2\II^-(v,v)
	\end{align*}
	If $\II^-(v,v) = 0$, we are done. Suppose $\II^-(v,v)\neq 0$. Since the action of $K^-$ is transitive on $\bb S^{l_-}$, we can choose $v$, such that $\langle \II^-(w_{gv}, w_{gv}), \dot c(-1)\rangle = b^2(g)\langle g \II(v,v), \dot c(-1)\rangle \ge 0$, contradicting the Partial Frankel Lemma. \\
	Hence $q_- =3$ and we get the manifold $R$ (cf. Table \ref{appendix_table_odd_cohom_one}). But in this case Ziller and Verdiani constructed non trivial parallel Jacobi fields along a horizontal geodesic on all nonnegatively curved cohomogeneity one metrics on $R$ (cf. \cite{vz}). Hence it cannot have quasipositive curvature.\\
	$\bar H = \bb Z_2 \oplus \bb Z_2$: In this case $\bar K^\pm \cong \O(2)$. 
	%There are two non central elements $h_\pm$ in $H$ mapping to the involutions $\bar h_\pm$ in $\bar H$, that represent the second component of $\bar K^\pm$. Since $h_-$, $h_+$ and $h_-\cdot h_+$ must be unit imaginary quanternions, they must anticommute in $G$. Note also that $\bar h_\pm$ acts by reflection on $\bb S^{l_\pm}$ and hence
	As in \cite{gwz} Lemma 7.2, we can arrange $H = \Delta Q$, where $Q = \lbrace \pm1, \pm i, \pm j, \pm k \rbrace \subset S^3$, $K_0^- = C^i_{(p_-,q_-)}$ and $K_0^+ = C^j_{(p_+,q_+)}$, with $p_\pm$ and $q_\pm$ odd and furthermore $0< p_-\le q_-$ and $0< p_+,q_+$. 
	%Hence $h_- = (\pm i, \pm i)$ and $h_+ = (\pm j, \pm j)$ implying, that all $p_\pm$ and $q_\pm$ are odd. Furthermore, we can assume $0< p_-\le q_-$ and $0< p_+,q_+$. 
%	There are two possibilities for $H$: $\Delta Q$ or $\Delta Q \oplus \langle (1,-1)\rangle$. The latter case generates another component for $K^\pm$ and is thus not simply connected. So we are left with $H=\Delta Q$. 
	The weights of both normal actions are $4$. If both $p_+$ and $q_+$ are at least $2$, then $p_- = q_- =1$. Since $p_+$ and $q_+$ are odd, $B_+$ is totally geodesic unless $\vert p_+ - q_+\vert = 2$. Hence up to ordering $(p_+, q_+) = (p, p+2)$ for $p\ge 1$, which is the diagram of $P_p$. Now suppose $p_+ = 1$. If $q_+ \neq 1$, then $p_- = q_- =1$ and $\II^-\vert_{S^2(W_0^- \oplus W_i^-)} = 0$ and if $q_+ \ge 5$, then $\II^+$ vanishes on $W_0 \otimes W_1 \oplus S^2(W_0 \oplus W_2) \oplus W_1 \otimes W_2$. In order for $B_+$ to be not totally geodesic $\II^+\vert_{S^2(W_1)} \neq 0$. As before we get a contradiction to the Partial Frankel Lemma.\\
%	 For $v \in W_1 \backslash \lbrace 0 \rbrace$ set $W_v = W_0 \oplus W_2 \oplus \bb R v$. Each $w \in W_v$ has the form $w = w_0 + w_2 + a \cdot v$. Hence we get
%	\begin{align*}
%		\II^+(w,w) &= a^2\II^+(v,v) + 2a \II^+(v,w_0) + 2a \II^+(v,w_2) + \II(w_0 + w_2,w_0+w_2)\\
%		& = a^2 \II^+(v,v)
%	\end{align*}
%	Since $K_0^+$ acts transitively on $\bb S^{l_+}$, we can find $v \in W_1$, such that $\langle \II^+(w,w),\dot c(1)\rangle \le 0$ for all $w \in W_v$, contradicting Partial Frankel.\\
	If $q_+ = 1$ and $p_+\ge 3$, then $p_- =1$ and $q_- \ge 3$. Hence $\II^-\vert_{S^2(W_0^- \oplus W_2^-)} = 0$ and $\II^+\vert_{S^2(W_0^+ \oplus W_1^+)} = 0$. If $\vert p_- - q_-\vert \neq 2$ or $\vert p_+ - q_+\vert \neq 2$, we get a contradiction to the  Partial Frankel Lemma, as before. Therefore $(p_-,q_-) = (1,3)$ and $(p_+,q_+) = (3,1)$, which is the diagram of $B^7$. These are all cases, since $p_\pm = 1 = q_\pm$ is not primitive.
\end{proof}

\subsection{$G$ semisimple of rank three}

	In the following let $G$ be a compact Lie group of rank $3$ acting almost effectively and isometrically with cohomogeneity one on a quasipositively curved Riemannian manifold $M$, such that the principal isotropy group $H$ has corank $2$.  $H_0$ has rank $1$ and is therefore one of $S^1$, $S^3$, $\SO(3)$. Note, that $H/H_0$ is cyclic and by the Isotropy Lemma $\max\lbrace l_-,l_+\rbrace \ge 2$ and by the Rank Lemma at least one of $l_\pm$ is odd.\\
	For $H_0=S^1$: $(l_-,l_+) = (1,2), (1,3),(2,3),(3,3)$. Hence for a singular isotropy group $K$: $K_0 = T^2$, if $l=1$; $K_0 = \SO(3),$ $S^3$, if $l = 2$; $K_0 = \U(2),$ $S^1 \times S^3$, if $l=3$; where $l+1$ is the codimension of the singular orbit with isotropy group $K$. \\
	For $H_0$ of dimension $3$: $l_\pm = 1,3,5,7$. For a singular isotropy group $K$: $K_0=\U(2)$, $S^3\times S^1$, $\SO(3) \times S^1$, if $l=1$; $K_0 = \SO(4),$ $S^3 \times S^3$, $l = 3$; $K_0 = \SU(3)$, if $l = 5$; $K_0 = \Sp(2)$, if $l = 7$. 
%	If $H_0 = \SU(2)$ (in the effective version), the lowest dimension of a representation is $4$, which must degenerate and hence $K_0^\pm = \SU(3)$ or $\Sp(2)$. 
	We will first look at the case, where $G$ contains a normal subgroup of rank $1$ and prove the following analogue of Proposition 8.1 in \cite{gwz}:	
\begin{prop}
	If $G$ has a normal subgroup of rank $1$ and the action is essential, then it is equivariantly diffeomorphic to the tensor product action of \emph{$\SU(2) \times \SU(3)$} on $\bb S^{11}$. 
\end{prop}
\begin{proof}
	We will closely follow the proof of Propostion 8.1 in \cite{gwz}. By allowing a finite kernel, we can assume $G = S^3 \times L$, where $L$ is a compact simply connected Lie group of rank $2$. Hence we have the cases $L = S^3 \times S^3$, $\SU(3)$, $\Sp(2)$ or $G_2$. If $H_0$ is $3$-dimensional, it must be contained in $L$, since otherwise the action is not essential. \\
	$L=S^3 \times S^3$: $H$ cannot be $3$-dimensional, since then the action can not be essential, because $H$ would project onto one of the factors. Hence $H_0 = S^1$ and one of $l_\pm$, say $l_-$ is equal to $2$ or $3$. \\
	$l_- =3$: The semisimple part of $K^-$ is $S^3$ projecting onto at least one of the factors. Hence, the involution in $S^3$ is central in $G$ and acts as the reflection on the normal sphere. Therefore $B_-$ is totally geodesic and  $8 = \dim G/H \le 2\cdot l_- +l_+ = 6 + l_+$, which implies $l_+ \ge 2$. If $l_+ =3$, then also $B_+$ is totally geodesic contradicting the Partial Frankel Lemma. If $l_+ = 2$, then  all groups are connected and $K^+ \cong S^3$, $K^- \cong S^3 \times S^1$. Hence $K^-$ cannot project onto all factors, but must project onto at least two, since the action is essential. Therefore $K^- = \lbrace (a,a,z) \vert \, a \in S^3, z \in S^1 \rbrace$ and $H= \lbrace (z^n,z^m,z^k) \vert \, z \in S^1 \rbrace$ with $n,m,k \neq 0$, because otherwise the action is not essential. Since $H$ must be the maximal torus of some $S^3$, we have $n = m = k = 1$ and $K^+ = \lbrace (a,a,a) \vert \, a \in S^3 \rbrace$, but then $K^+$ and $K^-$ only generate $\Delta S^3 \times S^3$ and the action is not primitive.
	Hence we are left with the case $l_- = 2$ and $l_+ = 1$ and $K_0^- = S^3$ and $K^+=T^2$. Since the action is primitive, $K_0^-$ must project to all factors and thus $K_0^- = \lbrace (a,a,a) \vert\, a \in S^3 \rbrace$ and $H_0 = \lbrace (z,z,z) \vert z \in S^1\rbrace$. Furthermore $K^-/K_0^- = H/H_0$ has at most $2$ elements, because $K^-_0$ can only be extended by the central elements of $G$. $w_-$ can be reperesented by $(j,j,j)$.
	%		 and $T^2 = \lbrace (z^nw^m,z^kw^l,z^pw^p)\vert z,w  \in S^1 \rbrace$. Obviously $(j,j,j)$ normalizes $T^2$. 
	Let $\iota \in T^2$ represent $w_+$. $\iota = (\iota_1,\iota_2,\iota_3)$ has order $2$, if $H$ is connected, and order $4$ otherwise, with $\iota^2 \in H$. Then $j \cdot \iota_i = \pm \iota_i \cdot j$, with negative sign, if $\iota_i$ is not central in $S^3$. But then $\iota \cdot (j,j,j) = (j,j,j) \cdot \iota \cdot \iota^2$. Since $\iota^2$ is contained in $H$ this implies $\vert W \vert \le 4$ and by linear primitivity $8 \le 2 \cdot (l_- + l_+) = 6$, a contradiction.\\
	$L=\SU(3)$: 
	%If $\dim H = 3$, then $H$ must be the $\SU(2)$-block in $\SU(3)$, since this is the only $3$-dimensional spherical subgroup (up to conjugacy). Since $H_0$ cannot act trivially on both normal spheres one of $K^\pm$ must contain $\SU(3)$ and the action is not essential. 
	With the same reasoning as in \cite{gwz} Proposition 8.1, we can assume $H_0 = S^1$. $H_0$ cannot be contained in the $S^3$-factor, since then the action is not essential.
	%$H_0$ acts with weight $2$ on the Lie algebra of the $S^3$ factor and trivial on the $\SU(3)$ factor. Therefore this action must degenerate in one of the isotropy groups say $K^-$, which contains the $S^3$ factor. But the action is not essential in this case.\\
	We have the following two cases:
	\begin{enumerate}[label = (\alph*), topsep=0pt,itemsep=-1ex,partopsep=1ex,parsep=1ex]
		\item The involution $\iota \in H_0$ is not central in $G$.
		\item The involution $\iota \in H_0$ is central in $G$
	\end{enumerate}
	In the second case, the proof of Proposition 8.1 in \cite{gwz} applies and hence we only give the proof of the first case.\\
	Suppose $l_- = 1$. Then $K^- = T^2$, $l_+ \ge 2$ and $H/H_0$ is cyclic. If $l_+ = 2$ then $K_0^+$ is isomorphic to $\SU(2)$ or $\SO(3)$. If $K^+_0 \cong \SO(3)$, then $K^+_0 \subseteq \SU(3)$ and hence the action is not primitive. If $K^+_0 \cong \SU(2)$, then $K_0^+$ is diagonal in $S^3 \times \SU(3)$ and cannot project to $\SO(3)$, since then the central element in $H_0$ will be central in $G$. Hence $K_0^+$ projects to the $\SU(2)$-block and $K^-\subseteq N(H_0)_0 = S^1 \times T^2$, which is not primitive. If $l_+ = 3$, then $N(H) \cap K^\pm /H$ are both at least one dimensional and hence $\vert W \vert \le 4$. This implies $10 = \dim G/H \le 2\cdot (l_-+l_+) = 8$, a contradiction. \\
	Hence we are left with the cases $(l_-,l_+) = (2,3)$, $(3,3)$ and all groups are connected. Suppose $(l_-,l_+) = (2,3)$. $K^- = \SO(3)$ or $\SU(2)$, $H = \lbrace (z^n, \diag(z^m,z^k,z^l) ) \vert\, z \in S^1 \rbrace$, $K^+ \cong \SU(2) \times S^1$ or $\U(2)$. $K^-$ cannot be contained in $\SU(3)$: Otherwise, the $\SU(2)$-factor of $K^+$ must be diagonal by primitivity. Since $H$ is conjugated to $\lbrace \diag(z,\bar z, 1) \vert \, z \in S^1 \rbrace$, it does not project to the $S^3$-factor and hence must commute with $\Delta \SU(2)$. But there is no $\SU(2)$ commuting with $H$. Therefore $K^- \cong \SU(2)$ is diagonal projecting to the upper $2 \times 2$-block in $\SU(3)$ and $H = \lbrace (z,  \diag(z,\bar z, 1))\vert\, z \in S^1\rbrace$. Since the action is primitive the $\SU(2)$-factor of $K^+$ cannot project to the upper block. The only other possibility up to conjugacy by an element in the normalizer of $H$, that commutes with some $S^1$ will be the lower $\SU(2)$ block in $\SU(3)$. But then $K^+ = \Delta S^1 \cdot \SU(2)$ and we get the tensor product action of $\SU(2) \times \SU(3)$ on $\bb S^{11}$. \\
	We are left with the case $(l_-,l_+) = (3,3)$. We can assume, that the $\SU(2)$-factor of $K^- = \SU(2) \cdot S^1$ is diagonal, since the action is essential and primitive. Note that in this case the $S^1$-factor of $K^-$ is contained in $\SU(3)$ and $K^- \cong \SU(2) \times S^1$. 
	%All involutions in $\SU(2)\times S^1$ are central. $N(\iota)_0 = \SU(2) \times \U(2)$ and $K^- \subset N(\iota)$, since $\iota$ is also central in $K^-$. Now $K^+$ has to be isomorphic to $\U(2)$, since otherwise by the same argument $K^+ \subseteq N(\iota)_0$, which contradicts primitivity. 
	We can assume, that the $\SU(2)$-factor of $K^-$ projects to the upper $2 \times 2$ block. This forces  the $S^1$-factor of $K^-$ to be  $\lbrace 1 \rbrace \times \lbrace \diag(z,z,\overline z^2) \vert \, z \in S^1 \rbrace$. Hence the element $(-1,1)$, which is central in $G$, is contained in $K^-$. Since the involution $\iota \in H$ is not central, $B_-$ is totally geodesic as the fixed point component of an involution. By Lemma \ref{obstructions_lem_totgeod}, we have $10 = \dim G/H \le 2\cdot l_-+l_+ = 9$, a contradiction.\\
	$L = \Sp(2)$: 
	%If $H$ is $3$-dimensional, then $H_0 = \Sp(1) \times 1$ or $H_0 = \Delta \Sp(1)$. In the latter case $L/H_0$ is  effectively $\SO(5)/\SO(3)$ and the Block Theorem applies. Hence $H_0 = \Sp(1) \times 1$. Since $H_0$ has only $4$-dimensional non trivial representations, one of $K^\pm$, say $K^-$ must contain $\Sp(2)$, by the Isotropy Lemma, since $\Sp(2)$ contains no $\SU(3)$, but this is not essential. 
	 Like in \cite{gwz}, we can assume $H_0 = S^1 = (z^n,\diag(z^m,z^l))$. Note that the involution in $H_0$ must be central, since otherwise $N(\iota)_0 = S^3 \times \Sp(1) \times \Sp(1)$ acts with cohomogeneity one and one dimensional principal isotropy group on a quasipositively curved manifold, contradicting the first case. Therefore $m$ and $l$ are either both even and non zero or both odd. By the Isotropy Lemma $m = l$. Since the action is essential one of $K^\pm$, say $K^-$ projects to $S^3$. If $n=0$, then $H_0$ acts trivially on $S^3$ and we can assume $K^-_0 = \SU(2)\cdot H_0$, by primitivity. Hence $K^-$ projects to $\U(2) \subset \Sp(2)$. If $n \neq 0$, then $N(H)/H\cong \U(2)$ acts with trivial principal isotropy group on a component of $M^H$. By the Core-Weyl Lemma, this action is fixed point homogeneous and hence contains the $\SU(2)$-factor in one of the singular isotropy groups. Therefore we can assume, that $K^- = \SU(2)\cdot H_0$ projects to $\U(2)$. If $K^-$ projects to $\U(2) \subset \Sp(2)$, then $l_- = 3$ and the adjoint action of the $\SU(2)$-factor in $K^-$ on $\fr g$ factorizes through $\SO(3)$. Since it acts transitively on the normal sphere, $B_-$ is totally geodesic by the equivariance of the second fundamental form and therefore by Lemma \ref{obstructions_lem_totgeod}, we have $12= \dim G/H \le 2 \cdot l_- + l_+ = 6 +l_+$, a contradiction, since $l_+ \le 3$.\\
	$L = G_2$: Suppose $H_0$ is three dimensional, then $H_0 = \SU(2) \subseteq \SU(3)$, since this is the only spherical subgroup of $G_2$ of dimension three. One of $K_0^\pm$, say $K_0^-$, is isomorphic to $\SU(3)$ and $l_- =5$. By primitivity $K^+_0 $ projects to the $S^3$-factor and since the action is essential $K_0^+$ contains $\Delta \SU(2)$ as a normal factor. $l_+$ can neither be $1$ or $5$, since $K_0^+$ contains $\Delta \SU(2)$ as a diagonal subgroup and $H$ is $3$-dimensional. Hence $l_+ = 3$ and all groups are connected. $K^+ = \Delta \SU(2) \cdot \SU(2)$, projecting to the maximal rank subgroup $\SO(4)$ in $G_2$. Since the $\SU(2)$-factors of $\SO(4)$ intersect in their centers $K^+$ contains the central element of $G$. But $H$ does not contain this element and hence $B_+$ is totally geodesic. Therefore $14 = \dim G/H \le 2\cdot l_+ +l_- = 11$, a contradiction. \\
%	We now assume $H_0 = S^1 \subset S^1 \times \SU(3)$. Note that the involution $\iota$ in $H_0$ must be central, since otherwise $N(\iota) \cong S^3 \times \SO(4)$ acts with one dimensional principal isotropy group on a component of $M^\iota$.  We can assume, that $H_0 = (z^n,\diag(z^m,z^l,z^k))$, with $m+l+k =0$. This acts with weights $2n$, $m-l$, $m-k$, $l-k$, $m$, $l$, $k$. By the Isotropy Lemma there are at most $2$ distinct weights. Hence $m = \pm l$. If $m = l$, then we have non trivial weights $2n$, $3m$, $m$, $2m$, which are more that two distinct ones. Hence $m = -l$, but then $k = 0$ and $2n = m$ or $m = n$. In both cases 
	Now we can assume $H_0 = S^1$ and $K^-$ has rank $2$. Hence $(l_-,l_+) = (1,2)$, $(1,3)$, $(3,2)$, $(3,3)$. By the lower Weyl group bound $16 = \dim G/H \le \vert W\vert /2 (l_- + l_+)$. If $(l_-,l_+) = (3,3)$, then $\vert W \vert \le 4$ and $16 \le 2(3+3) = 12$, a contradiction. If $(l_-,l_+) = (1,2)$, then $H/H_0$ is cyclic and $\vert W \vert \le 8$ and therefore $16 \le 4(2+1) = 12$. Suppose $(l_-, l_+) = (1,3)$. Then $K^- = T^2$, since it is connected and $K_0^+ = \SU(2)\cdot S^1$. But then $N(H) \cap K^\pm/H \cong S^1$, which implies $\vert W \vert \le 4$ and hence $16 \le 2(1+3) = 8$. Hence we are left with the case $(l_-,l_+) = (3,2)$ and all groups are connected. $K^+ \cong \SU(2)$ must be diagonal, since otherwise the action is either not essential or $H$ contains an involution $\iota$ such that $N(\iota)_0=S^3 \times \SO(4)$, which acts with 1-dimensional principal isotropy group by cohomogeneity one on a quasipositively curved manifold, contradicting the first case. By the same argument $\SU(2)$ must project to $\SO(3)\subset G_2$. First suppose, that $S^3$ acts effectively free. Then $G_2$ acts essentially on an evendimensional quasipositively curved manifold with one dimensional principal isotropy group, which cannot exist by the evendimensional classification. Hence the action of $S^3$ is not effectively free. But then $K^-$ must contain a noncentral element $a$ of the $S^3$-factor. Hence $N(a) = S^1 \times G_2$ and $B_-$ contains an $11$-dimensional totally geodesic subspace. By the same argument as for Lemma \ref{obstructions_lem_totgeod}, the horizontal geodesic has index at least $17-2\cdot6 +1 =6$. Since this must be bounded from above by $l_+=2$, we have a contradiction.
	%Now pick an Element $b \in K^+$ which projects to an involution in $G_2$. Then $N(b) = S^1 \times \SO(4)$ and $N(b) \cap K^+$ is one dimensional. Therefore $B_+$ contains an at least $6$ dimensional totally geodesic submanifold, a contradiction to the Partial Frankel Lemma.
\end{proof} 

\begin{prop}
	If $G$ is simple with $\rk G = 3$, then it is either the linear irreducible representation of \emph{$\SU(4)$} on $\bb S^{13}$ or the cohomogeneity one action of \emph{$\SU(4)$} on one of the Baizaikin spaces $B_p^{13}$ for $p \ge 0$. 
\end{prop}
\begin{proof}
	We will closely follow the proof of Proposition 8.2 in \cite{gwz}. We have to consider the cases $G = \SU(4)$, $\Sp(3)$ or $\Spin(7)$. The case $G = \Sp(3)$ can be ruled out with the same arguments as in \cite{gwz}, so the proof is left out here.\\
	$G = \SU(4)$: First assume $H_0 = S^1 =  \diag(z^{p_1}, z^{p_2},z^{p_3},z^{p_4})$. Then $H_0$ acts with weights $p_i-p_j$. By the Isotropy Lemma there can be at most two distinct weights, which only leaves the possibilities $(p_1,p_2,p_3,p_4) = (1,1,-1,-1)$, $(1,-1,0,0)$, $(1,1,1,-3)$ and $(3,3,-1,-5)$. The last three cases can be ruled out like in \cite{gwz}. Hence we only consider the first case. Note that the only possibilities for $l_\pm$ are $(l_-,l_+) = (1,2)$, $(1,3)$, $(2,3)$, $(3,3)$. If $(l_-,l_+) = (1,3)$ or $(3,3)$ then $\vert W\vert \le 4$ and by linear primitivity $14 = \dim G/H \le 2\cdot (3+3) = 12$. If $(l_-,l_+) = (1,2)$, then $H/H_0$ is cyclic and hence $\vert W \vert \le 8$. Linear primitivity again implies $14 \le 4\cdot (1+2) = 12 $, a contradiction. The only case left is $(l_-,l_+) = (2,3)$ and all groups are connected. In this case $K^+ = \SU(2)\cdot S^1$ and $K^- \cong \SU(2)$, since the involution of $H$ is central in $G$. $H$ has weights $0$ and $\pm 2$. Since the action of $H$ on the Lie algebra of $K^-$ must be nontrivial, it must be contained in the invariant subspace of $\fr{su}(4)$ on which $H$ acts with weight $\pm2$. If $H$ also acts nontrivially on the $\SU(2)$-factor of $K^+$ its Lie algebra is also contained in this subspace. Since $N(H) = S(\U(2)\U(2))$ is fixing this subspace, this is a contradiction to linear primitivity. Hence $H$ commutes with this $\SU(2)$-factor. Therefore $N(H)/H \cong \SO(4)$ acts by cohomogeneity one with cyclical principal isotropy group on a component of $M^H$, which has quasipositive curvature. Since $N(H)\cap K^-/H = \bb Z_2$, it must be covered by a cohomogeneity one manifold on which $\SU(2) \times \SU(2)$ acts with two $3$-dimensional singular isotropy groups. But this can only be a space form, which has $\vert W \vert \le 4$, a contradiction since $14 = \dim G/H \le 2(3+2) = 10$.\\
	Now we can assume that $H_0$ is three dimensional. Like in \cite{gwz} the only three dimensional spherical subgroups of $\SU(4)$ are $\SU(2)\subseteq \SU(3) \subseteq \SU(4)$ and $\Delta \SU(2) \subseteq \SU(2)\SU(2) \subseteq \SU(4)$. In the latter case the Block Theorem applies since it is effectively $\SO(6)/\SO(3)$. Therefore we assume $H_0$ to be the lower $2 \times 2$-block. By the Isotropy Lemma one of $K^\pm$ is equal to $\SU(3)$ or $\Sp(2)$. First assume $K^- = \SU(3)$ the cases $l_+ =1,7$ are handled as in \cite{gwz}. For $l_+ =3$ we have $K^+ = \SU(2)\SU(2)$. $-\Id \in K^+$ is central and not contained in $H$ and thus represents a Weyl group element. But then we get by linear primitivity $12 = \dim G/H \le 2 \cdot l_+ + l_- = 6+5 = 11$, a contradiction. Now consider $K^- = \Sp(2)$. Note that $-\Id \in \Sp(2)$ is not contained in $H$ and hence $B_-$ is totally geodesic. If $K^+ = \Sp(2)$, the action is not primitive, since $N(H)$ acts transitively on all embeddings of $\Sp(2)$ containing $H$. If $K^+ = \SU(2) \SU(2)$, then both singular isotropy groups contain the central element and hence $\vert W \vert = 2$ and therefore $12 \le 7+3 = 10$. $K^+ = \SU(3)$ was already handled and we are left with $l_+ = 1$. $K^+ = K_0^+ \subseteq S(\U(2)\U(2))$ and we can assume that up to conjugacy $K^+ = H_0 \cdot \diag(z^k,z^l,\overline z^{(k+l)/2},\overline z^{(k+l)/2})$. If $-\Id$ was in $K^+$ then both $B_\pm$ are totally geodesic of total dimension $5+11 = 16 \ge 13$ contradicting the Partial Frankel Lemma. This implies $(k,l) = (2p,2q)$ with $p$ and $q$ relatively prime and not both odd. By choosing $z = i$ and multiplying with $\diag(1,1,\pm(i,-i))$, either $\iota = \diag(1,-1,-1,1)$ or $\diag(-1,1,-1,1)$ is contained in $K^+$. If $H$ is connected, then $\iota$ is a Weyl group element. By primitivity we can assume, that $K^-$ is the standard $\Sp(2)$, since $N(H) = S(\U(2)\U(2))$. But $\iota$ normalizes $\Sp(2)$ in this case. Hence by linear primitivity $12 \le l_-+l_+ = 1+7 = 8$, a contradiction. Therefore $H$ is not connected. Since $\Sp(2)$ can only be extended by $\bb Z_2$, we have that $H/H_0 \cong \bb Z_2$ with $\iota$ representing the second component. Now we can argue as in \cite{gwz} to show, that this can only be obtained by the Bazaikin spaces. \\
	$G= \Spin(7)$: The case $H_0 = S^1$ can be ruled out as in \cite{gwz}. If $H$ is $3$ dimensional, we can assume, that $H_0$ is embedded as the normal subgroup of a $4 \times 4$-block, since the Block Theorem applies, if $H_0$ is embedded as a $3 \times 3$-block. Therefore $K_0^-$ is one of $\SU(3)$ or $\Sp(2)$. The case $K^-_0 = \SU(3)$ can be ruled out as in \cite{gwz}. If $K_0^- = \Sp(2)$, then it contains the central involution of $\Spin(7)$, which represents a Weyl group element, since $l_- = 7$. Hence $\vert W \vert \le 4$ and thus by the lower Weyl group bound, $18 \le 2\cdot l_- + l_+ = 14 + l_+$. Therefore $l_+ \ge 4$ and hence $l_+ = 7$, since the case $l_+ = 5$ implies $K^+ = \SU(3)$, which was ruled out before. But then $K^+ = \Sp(2)$ also contains the central element of $\Spin(7)$ and therefore $\vert W \vert = 2$ and hence $18 \le l_-+l_+ = 7+7 = 14$, a contradiction. This finishes the proof.
\end{proof}

\subsection{$G$ semisimple and not simple with rank one normal subgroups}

\begin{prop}
	Let $G$ be a semisimple compact Lie group of rank at least four, which contains a normal factor of rank one. Then \emph{$G=\SU(2)\SU(n)$}, $M=\bb S^{4n-1}$ and the action is the tensor product action.
\end{prop}
\begin{proof}
	By allowing a finite kernel $G = S^3 \times L$ with $L$ being a compact Lie group of rank at least three. First we show, that at least one of $K^\pm$ projects surjectively onto the $S^3$-factor: Suppose both singular isotropy groups do not project onto $S^3$. By primitivity both must project onto some $S^1_\pm \subseteq S^3$. Hence $K^\pm = S^1_\pm \bar K^\pm$, where $\bar K^\pm \subseteq L$. In both cases $S^1_\pm$ cannot act trivially on the normal sphere $\bb S^{l_\pm}$, since otherwise one of $S^1_\pm$ would be contained in the other singular orbit, contradicting, that the action is primitive and essential. If both $l_\pm$ are equal to $1$, then $H_- =H = H_+$, which must be finite by primitivity. But this contradicts $\rk G \ge 4$. Therefore we can assume $l_- \ge 3$, which implies, that $H$ projects onto the $S^1$-factor of $K^-$, by the classification of homogenous spheres. By primitivity, this implies that $K^+$ projects onto $S^3$, contradicting the assumption.\\
	We can now assume, that $K^-$ projects onto $S^3$. Since the action is essential $K_0^- = \SU(2) \cdot K_L^-$, with $K_L^\pm = K_0^\pm\cap L$. We also write $H_0 = (H_{S^3} \cdot H_\Delta \cdot H_L)_0$ and $K^+_0 = (K_{S^3}^+ \cdot K^+_\Delta \cdot K^+_L)_0$. Note that $H$ cannot project onto the $S^3$-factor since the action is essential. By the classification of homogeneous spheres this implies $l_- = 2,3$ and therefore $K^+$ is connected. If $l_+ = 1$, then $H_L$ will be in the kernel of the isotropy action of $\bb S^{l_+}$. By Lemma \ref{lem_weylgroup_ineffkern} $H_L$ must be finite and thus $H$ has rank $0$ or $1$, which contradicts, that $L$ has rank at least $3$. Therefore both of $l_\pm \ge 2$ and all groups are connected. Furthermore note, that $\bb S^{l_+} = K_L^+/H_L$ almost effectively and we denote by $K^\prime$ the simple normal subgroup of $K^+_L$, that acts transitively on $\bb S^{l_+}$, and $H^\prime = H\cap K^\prime$. By the Isotropy Lemma, all representations of $H^\prime$ are contained in $K^\prime$. Note that $K^\prime$ has corank  at most $2$ in $L$ if $l_+$ is even and corank at most $3$ if $l_+$ is odd.\\
	We claim, that $K^\prime$ is contained in a simple factor of $L$, if $\rk K^\prime \ge 2$. First we assume $l_-=3$ and $\bb S^{l_-} = \SU(2)$. In this case $H$ is contained in $L$ with corank $1$. $K^\prime \subset L$ has at most corank $1$. If $K^\prime$ is not contained in a simple factor, then $L = L_1 \times L_2$ and $K^\prime$ projects non trivially to both factors. But then $K^\prime$ has corank at least $2$, since $\rk K^\prime \ge 2$. Now assume $l_- = 2,3$ and $\bb S^{l_-} = \U(2)/S^1$, if $l_- =3$, then $H = \Delta S^1 \cdot H_L$, $K^+ = S^1 \cdot K_L$ and $K^\prime$ has corank at most $2$ in $L$. If $\rk K^\prime \ge 3$, then we get a contradiction as before. Therefore assume $\rk K^\prime = 2$. The corank of $K^\prime$ in $L$ is at most $2$. If $K^\prime$ is not contained in a simple factor, then $L = L_1\times L_2$ with $L_i$ of rank $2$ and $K^\prime$ projects non trivially to both of them and therefore the corank of $K^\prime$ in $L$ is exactly $2$. Thus $K^\prime$ has to commute with an $S^1$ in $L$, but this contradicts $K^\prime$ being simple.\\
	Now we claim $H_{S^3} = H\cap S^3$ is central. If not, then let $a \in H_{S^3}$. $N(a) = S^1 \times L$ acts with cohomogeneity one on a component $M^\prime$ of $M^a$, which has codimension $2$.
	%If $H$ has finite projection to $S^3$, then $N(a)$ acts essentially on $M^\prime$. Therefore either $L/H = \SO(n+2)/\SO(n)$ for $n \ge 4$ and the Block Theorem applies or $L = \Spin(7)$ and $H= \SU(3)$, which must commute with $\Delta \SU(2)\subset K^-$, which is not possible. 
	Since $H$ is connected, it projects to $S^1 \subset S^3$. Since $L$ is semisimple, it acts orbit equivalent and essentially on $M^\prime$. If $l_- = 2$, then $M^\prime$ is covered by an action where both singular isotropy groups are isomorphic, which can only be the exceptional action of $\Spin(8)$ on $\bb S^{15}$ that cannot be extended by an $S^1$ action. 
	%But this contains $G_2$ in $H$, which does not commute with $\SU(2)$. 
	If $l_- =3$, then one of the singular orbits has codimension $2$, which cannot happen by induction, since $L$ is semisimple. Note that we are able to use Table \ref{appendix_table_cohom_one_odd_sphere} here. The proofs in the simple cases refer to this proposition, but only for subgroups of the form $S^3 \times L^\prime$, where $L^\prime$ has strictly lower rank, than the original group. Hence the induction argument does not break.\\
	Therefore we can assume, that $H_{S^3}$ is central. If $S^3$ acts effectively free, then we can apply the proof of Proposition 9.1 in \cite{gwz}. Therefore we assume, that $S^3$ acts not effectively free. If one of $K^\pm$, contains a central element, that $H$ does not contain, then the corresponding singular orbit is totally geodesic and we can apply Lemma \ref{obstructions_lem_totgeod}. If $K^-$ contains an element  $a \in S^3$, then $N(a) /N(a) \cap K^- = B_-$ is totally geodesic. If $K^+$ contains $a \in S^3$, then $N(a) /N(a)\cap K^+ = S^1 \times L/K^+$  has codimension $2$ in $B_+$. By the same argument as in the proof of Lemma \ref{obstructions_lem_totgeod}, we have $\dim G/H \le 2\cdot l_+ + l_- +4 \le 2\cdot l_++7$.\\
	We first assume, that $\rk K^\prime \ge 2$. By the previous discussion, this implies, that $K^\prime$ is contained in a simple factor $S$ of $L$. We are now looking for pairs of simple groups $(S,H^\prime)$, with $\rk S - \rk H \le 3$, such that $H^\prime$ is spherical and the isotropy representation $S/H$ has at most one non trivial representation, except if $H^\prime = \Spin(7)$, where two non trivial representations are allowed. Furthermore $S$ contains $K^\prime$ such that $K^\prime/H^\prime = \bb S^{l_+}$. We will not look at the cases, where $H^\prime$ is a $k \times k$ block in a classical Lie group, since then the Block Theorem applies.\\
	If $l_+$ is even, we have the following pairs by Table \ref{appendix_table_spherical}: $(\Spin(n), \SU(3))$ for $n = 7, 8, 9$, $(\Spin(8), \SU(4))$, $(G_2, \SU(3))$ and $(F_4,\SU(3))$. In the last case $\vert W \vert \le 8$ and by linear primitivity $\dim L \le 6 + 4 \cdot (3+6) = 42$, a contradiction. $(G_2,\SU(3))$ is not essential, since the $\SU(3)$ representations have to degenerate in $G_2/\SU(3)$. In the second case the Block Theorem applies, since there is an outer automorphism of $\Spin(8)$, such that $\SU(4)$ is mapped to the $6 \times 6$ block. Hence we are left with the first case. By the previous discussion $\dim G/H \le 2\cdot l_+ + l_- +4 \le 2\cdot l_++7$ and hence $\dim L \le 2 \cdot 6 +7 +6 = 25$, which only leaves $S = \Spin(7)$. By primitivity $\Delta \SU(2)\subset K^-$ projects to $\Spin(7)$ and hence $\Spin(7)$ must contain $\SU(2) \cdot \SU(3)$ as a subgroup, which is not the case.\\
	Now assume $l_+$ is odd. Then $H^\prime$ has corank at most $3$ in $S$ and we have the following pairs: For $\rk H^\prime = 1$, we have: $(\SU(n), \SU(2))$ as a $2 \times 2$ block with $n = 3,4,5$, $(\Spin(n), \Sp(1))$ as a factor of a $4 \times 4$ block for $n=5,6,7,8,9$, $(\Sp(n), \Sp(1))$ as a $1 \times 1$ block, with $n = 2,3,4$, $(G_2,\SU(2))$, where $ \SU(2) \subset \SU(3) \subset G_2$ as a $2 \times 2$ block. All other exceptional cases are ruled out by linear primitivity: $\dim L \le 4 \cdot (l_-+l_+) + \dim H -3 \le 44$. Furthermore we have $\dim H \le 7$ and $l_+ \le 7$. By the dimensional bound on $G/H$, this means $\dim L \le \dim H -3 + 2 \cdot 7 + 7\le 25$. In the first case we still have $n = 3,4,5$. $n = 3$ is not essential. Hence $n = 4$. If $l _+ = 7$, then $K^\prime$ contains $\Sp(2)$. Since $H^\prime$ is the lower $2 \times 2$ block, $K^-$ can only project to $\SU(2) \times \SU(2) \subset \Sp(2)$. Hence $w_\pm K^\mp w_\pm = K^\mp$ and therefore $\dim L \le 7+3+2 = 12$, a contradiction. If $n = 5$ then $l_+ = 7$ by the dimensional bound and $\Sp(2) \subset \SU(4) \subset \SU(5) = L$ and $K^+$ projects to $S^1 \cdot \Sp(1) \cdot \Sp(2)$, but this is not contained in $\SU(5)$. We are left with $n=4$. If $l_- = 7$, then by primitivity $K^+$ projects to $S^1 \cdot \Sp(2) \subset \SU(4)$, but $\SU(4)$ does not contain this group. Thus $l_+ = 5$ and $K^\prime = \SU(3) \subset \SU(4)$ as a $3 \times 3$ block. If $l_-=3$, then it is easy to see, that $w_-K^+w_- = K^+$ and therefore $\dim L \le 2+ l_++2\cdot l_- = 13$, a contradiction. Therefore let $l_- =2$. If there is an $S^3$-factor in $L$, then $K^-$ also projects to this and $H$ does not because the action is essential. Therefore $\dim H \le 4$. But then the dimensional bound gives $\dim L \le 1 + 2 \cdot 5 + 2 + 4 = 17$ a contradiction. This leaves $G = \SU(2) \cdot \SU(4)$ with the tensor product action on $\bb S^{15}$, which contradicts the assumption, that $\SU(2)$ acts not effectively free. In the second case $n = 5,6,7$. Since $\Spin(6) = \SU(4)$ was already considered, we are left with $n = 5,7$. If $l_+=7$, then $n =5$ is not essential and therefore $n =7$. But then $K^\prime$ contains $\Sp(2)$, which contains the central element of $\Spin(7)$, which is not contained in $H^\prime $ and thus $B_+$ is totally geodesic. By Lemma \ref{obstructions_lem_totgeod} $\dim L \le 2 + 2 \cdot 7 +3 = 19$, a contradiction. Therefore we have $l _+ \le 5$ and thus $\dim L \le 19$, which leaves $n = 5$. But $\Spin(5)$ contains no $\SU(3)$. In the third case $n = 2,3$ and since $\Spin(5) = \Sp(2)$, was already ruled out, we are left with $n = 3$. If $l_+=5$, then $\dim L\le 19$, hence $l_+ =7$ and $K^\prime = \Sp(2)$ as a $2 \times 2$ block. If $H$ contains two $3$ dimensional factors, then it has to project surjectively to each $\Sp(1)$-factor of the diagonal in $\Sp(3)$ and cannot commute with another $\SU(2)$. Therefore $\dim H \le 5$ and since $\dim L \le 23$ it cannot contain an $S^3$-factor. But then $N(\Sp(1))/\Sp(1) = S^3 \cdot \Sp(2)$ acts with one dimensional principal isotropy group, which was already ruled out. If $(S,H^\prime) = (G_2, \SU(2))$, then $l_+ =5$, since there is no $\Sp(2)$ in $G_2$. Then $L = S^3 \times G_2$ by the dimensional bound and thus $N(\SU(2))/\SU(2) = S^3 \times S^3 \times \SO(3)$ acts with one dimensional principal isotropy group, but this does not exist.\\
	For $\rk H^\prime \ge 2$, we have: $(\Spin(n), \SU(3))$ with $n = 6,7,8,9,10,11$, $(\SU(4), \Sp(2))$, $(\Spin(8), \Sp(2))$, $(\Spin(n), G_2)$ for $n = 7,8,9,10$, $(\Spin(n), \Spin(7))$ for $n = 9,10,11,12,$ $13$, $(S,H^\prime )$ for $S = F_4,E_6$ and $H^\prime = \SU(3), G_2, \Spin(7)$. In the last case, if $H^ \prime = \SU(3), G_2$, $\dim H \le 15$ and $l_+ \le 7$. Hence $\dim L \le \dim H -3 + 2 \cdot l_+ + 7 = 33$. If $H^\prime = \Spin(7)$, then $\dim L \le 56$ and hence $S = F_4$, but $F_4$ does not contain $\SU(2) \cdot \Spin(7)$ as a subgroup. In the first case $\dim H \le 10$ and $l_+=7$. Hence $\dim L \le 7+2 \cdot 7 + 7 = 28$ and therefore $n = 6,7,8$, but none of these groups contains $\SU(2) \cdot \SU(3)$, which is necessary by primitivity. The second and third case are ruled out by the Block Theorem since $\SU(4) = \Spin(6)$ and $\Sp(2) = \Spin(5)$. In the fourth case $l_+ = 7$ and $\dim H \le 15$. Hence $\dim L \le 33$, which leaves $n =7,8$, but the first one is not essential and the second one does not contain $\SU(2)\cdot G_2$ as a subgroup. In the fifth case $\dim H \le 22$ and $l_+ =15$, which implies $\dim L \le 56$ and hence $n = 9,10,11$. $n = 9$ is not essential and $\Spin(10)$ contains no $\SU(2) \cdot \Spin(7)$, which leaves $L = \Spin(11)$. But then $\Delta S^1 \subset H$ projects to the upper $3 \times 3$ block and hence acts nontrivially on the $8$-dimensional representations of $\Spin(7)$. But this action cannot degenerate in $\bb S^{l_+} = \Spin(9)/\Spin(7)$.\\
	We are now left with the cases, where $K^+$ consists of rank one groups only. In this case $l_+ = 2,3,4$ and $H$ has dimension at most $4$. This implies $\dim L \le \dim H -3 + 2 \cdot l_+ + 7 \le 16$. Furthermore $\rk L = 3,4$. The possible rank $3$ groups with this dimensional bound are $(S^3)^3$, $S^3 \times \SU(3)$, $S^3 \times \Sp(2)$, $\SU(4)$. The ones with rank $4$ are $(S^3)^4$, $(S^3)^2\times \SU(3)$, $(S^3)^2\times \Sp(2)$, $\SU(3)\times \SU(3)$. If $L$ consists of $S^3$-factors only, then $H$ has no $3$ dimensional factors and one of $l_\pm$, say $l_-$, equals $3$. Hence the central element of the $S^3$-factor of $K^-$ is central in $G$ and acts as the antipodal map on $\bb S^{l_-}$. Therefore $B_-$ is totally geodesic and by Lemma \ref{obstructions_lem_totgeod}, we have $\dim G \le 2+2\cdot l_-+l_+ = 11$, a contradiction. Let $L$ contain $S^3 \times \SU(3)$ as a normal subgroup. If $l_+ = 4$, then $\SU(2)\cdot\SU(2) \subset K^+$ is contained in $\SU(3)$, since the action is essential. But $\SU(3)$ does not contain such a subgroup. Hence $l_+ \le 3$ and the dimensional bound implies $\dim L \le -1+2 \cdot 3+ 7 =12$, which implies $L=S^3 \times \SU(3)$ and $H$ has dimension $2$. Now we assume $l_- = 2$. Then $l_+=3$ and $H=\Delta S^1\cdot S^1$ and $K^+= \Delta S^1 \cdot S^1 \cdot \SU(2)$. $K^+$ projects to $\U(2)$ in $\SU(3)$, since otherwise $K^+$ contains a central element, which is not contained in $H$, forcing $B_+$ to be totally geodesic, a contradiction since $11 = \dim L \le -1+2\cdot 3 +2 = 7$.  By primitivity $K^-$ projects to an $\U(2)$ as well. Let $\iota\in \Delta S^1$, be the involution contained in $\Delta S^1 \subset H$. Then $N(\iota) = S^3 \times S^3 \times \U(2)$. Hence $S^3 \times S^3 \times S^3$ acts orbit equivalent with one dimensional principal isotropy group, a contradiction. If both are $l_\pm = 3$, then the $S^3$-factors of $K^\pm$ both project nontrivially to $\SU(3)$ by primitivity and thus cannot project surjectively to both $S^3$-factors of $G$. The $S^1$-factors of $H$ can be chosen to commute with an $S^3$-factor of $K^-$ or $K^+$. Let $S^1 = H_L$ be the $S^1$-factor contained in $L$. Then $S^1 = (1,z^n,\diag(z^m,z^l,z^k))$. By the Isotropy Lemma, $H_L$ can only act with one non trivial weight. Hence $n =0$ and $m = k =1$ or $m=k = 2$ and $n =3$. In the first case $N(S^1)/S^1 = S^3 \times S^3 \times S^3$ acts with $1$-dimensional principal isotropy group, a contradiction. In the second case the involution of $S^3 \subset K^+$ is not contained in $H$, as well as for $K^-$. Furthermore $K^\pm$ both project to a subgroup isomorphic to $\U(2)$, since the involutions in the $S^3$-factors correspond to the Weyl group elements $w_\pm$, it is easy to see that $w_\pm \proj_{\SU(3)}(K^\mp)w\pm = \proj_{\SU(3)}(K^\mp)$. Hence by linear primitivity $8=\dim \SU(3) \le \dim(\U(2)) + \dim(\U(2))-2 = 6$, a contradiction. Let $L$ contain $S^3 \times \Sp(2)$ as a normal factor. If $l_+ = 4$, then $H_L \subset \Sp(2)$ and we can apply the Block Theorem, since we have $\SO(5)/\SO(3)$ effectively. Therefore $l_+\le 3$ and $H$ has no $3$-dimensional factor. Thus $\dim L \le -1 + 2 \cdot 3 +7 = 12$, a contradiction. If $L= \SU(4)$, then by the same arguments as before the Block Theorem applies for $l_+=4$ and thus $\dim L \le 12$, a contradiction. We are left with the case $L = \SU(3) \times \SU(3)$. Let $l_+ = 4$. If $S^3= H_L$ projects to some $2 \times 2$-block, then it has a $4$-dimensional representation, which contradicts the Isotropy Lemma. Hence $H_L = \SO(3)$. But this must have a $5$ dimensional representation in one of the factors and is therefore not spherical. For $l_+ \le 3$, we again have $\dim L \le 12$, which is a contradiction. This finishes the proof.
\end{proof}

\subsection{$G$ semisimple and not simple without rank one normal subgroups}

	Let $G$ be a semisimple and not simple subgroup without normal factors of rank one, which acts by cohomogeneity one on a quasipositively curved manifold $M$. 
\begin{prop}
	If the action by $G$ on $M$ is essential and the principal isotropy group has corank $2$, then \emph{$G = \Sp(2) \times \Sp(n)$} and the action is the tensor product action on $\bb S^{8n-1}$.
\end{prop}
\begin{proof}
	Let $G = L_1 \times L_2$ such that $\rk L_i \ge 2$ for $i = 1,2$. We fix the notation $K^\pm_0 = (K_1^\pm \cdot K_\Delta^\pm \cdot K_2 ^\pm)_0$ and $H_0 = (H_1 \cdot H_\Delta \cdot H_2)_0$. If we can assume that one of $K^\pm_i$ acts transitively on the normal sphere, we can apply the proof of \cite{gwz} Proposition 10.1. If one of the $H_i$ is non finite after the action is made effective, then one of $K_i^\pm$ must act transitively on $\bb S^{l_\pm}$. Otherwise they both either act effectively free or trivially, which implies $H_i$ would be a subset of $H_-\cap H_+$, which is finite. So we assume that both $H_i$ are finite. Then $H_i \subseteq H_- \cap H_+$ and thus $H_\Delta$ has corank $2$. The projections of $H$ to the $L_i$ have at most finite kernel. Hence $\rk L_i = 2$. Since $G$ has no normal factors of rank one, both $L_i$ must be simple. Therefore $L_i \in \lbrace G_2, \SU(3), \Sp(2)\rbrace$.  One of $l_\pm$, say $l_-$ is odd and therefore $K^-$ has corank one. If one of $K_i^-$ has full rank, then $K^-_0 = K_1^- \times K_2^-$ and one of $K^-_i$ acts transitively on $\bb S^{l_-}$. Therefore assume, that all of $K^\pm_i$ and $K^\pm_\Delta$ have rank one. Then also all simple factors of $H$ have rank one and $\bb S^{l_-}$ is one of $S^1$, $\SO(4)/ \SO(3)$, $\U(2) / \U(1)$, $\SU(2)$. If one of $K^-_i$ is three dimensional or if $K^-$ is abelian, then one of $K^-_i$ acts transitively. Hence $K^-_\Delta \cong \SU(2)$, $K_i^- \cong S^1$ and $H_0 = T^2$. If $l_+ = 1$, then $K^+$ is abelian and one of the factors $K_i^+$ acts transitively. If $l_+= 3$, then $\vert W \vert \le 4$ and hence $\dim G - 2 \le 2 \cdot (l_-+l_+) \le 2 \cdot(3+3) = 12$. But this is a contradiction since $\dim L_i \ge 8$. Therefore we are left with the case $l_+ = 2$ and all groups are connected. $K^+_\Delta$ has corank $2$ and contains an $\SU(2)$-factor, such that one $S^1$ subgroup of $H$ commutes with $\SU(2)$. Therefore $K^+$ projects to $\U(2)$ in both factors. By the lower Weyl group bound we have $\dim G \le 4 \cdot (3 + 2) +2 = 22$. This rules out the cases $G = G_2 \times G_2$, $G_2 \times \Sp(2)$. We start with the case $G = \SU(3) \times \SU(3)$. Let $K^\prime$ denote the $S^1$-factor of $K^+$, which is contained in $H$. $K^\prime = ( \diag(z^p,z^q,z^r),\diag(z^n,z^m,z^l))$, such that $n+m+l=0 = p+q+r$. $S^1$ acts with weights $n-m$, $m -l$, $n-l$ and $p-q$, $q-r$, $p-r$. By the Isotropy Lemma and since $K^+$ projects to $\U(2)$, we can assume that $K^\prime = (\diag(z,z,\bar z^2), \diag(z,z,\bar z^2))$ and that $\Delta\SU(2) = \lbrace (\diag(A,1), \diag(A,1))\vert \, A\in \SU(2) \rbrace$. Since both $K_1^-, K_2^-$ commute with $K^-_\Delta = \SU(2)$ both are one of $\diag(z, \bar z^2,z)$ or $\diag(\bar z^2,z,z)$. Therefore $K^-$ contains elements of the center of $G$, which does not intersect $H$. Hence $B_-$ is totally geodesic and therefore by Lemma \ref{obstructions_lem_totgeod} $14 = \dim G/H \le 2\cdot 3+2 = 8$, a contradiction. The case $\SU(3) \times G_2$ is ruled out similarly.\\
	Let $G = \Sp(2) \times \SU(3)$. Again $K^+$ projects to $\U(2)$ in $\SU(3)$ and to $\Sp(1) \times S^1$ or $\U(2)$ in $\Sp(2)$. Since $H$ contains the involution $\iota$ in $\Delta \SU(2)$, we have that $N(\iota)_0 = \Sp(1) \times \Sp(1) \times \U(2)$ or $\Sp(2) \times \U(2)$ acts with two dimensional principal isotropy group on a component of $M^\iota$. Since the action is orbit equivalent to the action of $\Sp(1) \times \Sp(1) \times \SU(2)$ or $\Sp(2) \times \SU(2)$ with one dimensional principal isotropy group this is a contradiction.\\
	Therefore we can assume that $K_1^-$ acts transitively on $\bb S^{l_-}$. The remaining proof is the same as in \cite{gwz} Proposition 10.1.
\end{proof}

\subsection{$G$ simple of rank at least $4$}

	We will finish the classification with the case, where $G$ is simple and $\rk G \ge 4$.
\begin{prop}
	If $G$ is a simple Lie group of rank at least $4$, that acts essentially with cohomogeneity one on a quasipositively curved manifold $M$, then either \emph{$G=\Spin(k)$} for $k = 8$ or $10$ and the action is given by one of the exceptional linear actions of \emph{$\Spin (8)$} on $\bb S^{11}$ and \emph{$\Spin(10)$} on $\bb S^{31}$, or \emph{$G=\SU(5)$} and the action is given by the linear action of \emph{$\SU(5)$} on $\bb S^{19}$.
\end{prop}
\begin{proof}
	If $G = \Sp(k)$ for $k \ge 4$ or is an exceptional Lie group, we can apply the same proofs as in \cite{gwz} Proposition 11.1 and Proposition 11.4.
	\begin{center}
		$G=\Spin(k)$ for $k \ge 8$
	\end{center}
	If $k=8, 10$ we can apply the proof of the analogue statement for positive curvature of Proposition 11.3 in \cite{gwz}. Therefore we first assume $k = 9$:
	In this case let $T^2$ be a maximal torus of $H_0$ contained in $S^1\cdot \SU(4) \subseteq \Spin(8)$ and let $\iota \in T^2\cap \SU(4)$ be an involution. Then $N(\iota)_0 = \Spin(8)$ or $\Spin(5)\cdot \Spin(4)$. If $N(\iota)_0 = \Spin(8)$, then $M^\iota_c$ has even codimension at most $8$ and $\dim M \ge 22$, since $\dim H \le 14$. Therefore $M^\iota_c$ is covered by the cohomogeneity one manifold with group diagram $G_2 \subset \lbrace \Spin(7), \Spin(7)\rbrace \subset \Spin(8)$. Note that the two $\Spin(7)$ groups are mapped to each other by an outer automorphism of $\Spin(8)$. Therefore $H_0$ contains $G_2$ as a subgroup of maximal rank and hence $H_0 = G_2$. Furthermore one of $K^\pm_0$, say $K^-_0$, must contain $\Spin(7) \subset \Spin(8)$. Since the outer automorphism of $\Spin(8)$ that maps one of the $\Spin (7)$ to the other has order $3$, it cannot cover a manifold with an exceptional orbit. Hence both $K^\pm$ contain the $\Spin(7)$. But then $K^- \cdot K^+ = \Spin(8)$ and the action is not primitive. The remaining case $N(\iota) = \Spin(5) \cdot \Spin(4)$ can be ruled out as in \cite{gwz}.\\
	Let $k \ge 11$. Then $\rk H \ge 3$. The simple spherical subgroups of $\Spin(k)$ are given by $\Sp(1)$, $\SU(3)$, $\Sp(2)$, $G_2$, $\SU(4)$, $\Spin(7)$, $\Spin(l)$ with the embeddings contained in Table \ref{appendix_table_spherical}. Suppose $H$ contains a simple factor $H^\prime$. If it is an $l \times l$-block it its ruled out by the Block Theorem. Now suppose $H^\prime = G_2$. It has a seven dimensional representation in $\Spin(7)$, which must degenerate. Furthermore $N(G_2)_0 = \Spin(k-8) \times G_2 \subset \Spin(k-8)\times \Spin(8)$. Since $H$ has at least rank $3$, there is an at least one dimensional subgroup $H^{\prime\prime}$ of $\Spin(k-8)$, such that $H^{\prime\prime} \cdot G_2$ acts on a subspace of $\Spin(k)$ orthogonal to $\Spin(8)$, but this cannot degenerate. Now suppose $H^\prime = \Sp(2)$. Then $\Sp(2)$ has a five dimensional representation in $\Spin (8)$ and an eight dimensional orthogonal to $\Spin(8)$. Both must degenerate and therefore $(l_-,l_+) = (5,11)$ and $K^-$ contains $\Spin(6)$ and $K^+$ contains $\Sp(3)$ as a normal factor. Furthermore $\rk H \le 3$ by primitivity, which means $k = 11$ and $\dim H \le 13$. By the lower and upper Weyl group bound $42 \le \dim G/H \le 2 \cdot (l_-+l_+) = 32$, a contradiction. $H^\prime = \SU(4)$ has a six dimensional representation contained in $\Spin (8) $ and an eight dimensional orthogonal to $\Spin(8)$ both must degenerate and hence $K^-$ contains $\Spin(7)\subseteq \Spin(8)$  and $K^+$ contains $\SU(5)$ as a normal subgroup. Let $\iota \in \SU(4)$ be the central involution. Then $N(\iota)_0 = \Spin(k-8)\times \Spin(8)$ acts with cohomogeneity one on $M^\iota_c$. But this action is obviously not primitive, since $k \ge 11$, a contradiction. $H^\prime = \Spin(7)$. This has a seven and an eight dimensional as well as a trivial representation in $\Spin(k)$. The seven and eight dimensional must degenerate in $\Spin(9)$ and the trivial one in $S^1 \cdot \Spin(7)$. $\Spin(9)$ contains the central element of $\Spin(k)$ and therefore $\vert W \vert  \le 4$ and by linear primitivity $\dim G \le 2 \cdot l_- + l_+ + \dim H = 52$, a contradiction. Now suppose $H^\prime= \Sp(1)$. If it is a block, then the Block Theorem applies. Therefore it is empedded in $\Spin(4) \times \ldots \times \Spin(4)$. If it projects to more than one of these $\Spin(4)$, then the isotropy representation contains three and four dimensional as well as trivial representations. Since $H$ cannot contain $\Sp(n)$, this factor must degenerate in $\Sp(1) \times \Sp(1)$ and $\Sp(2)$ or $\SU(3)$. But then $\rk H \le 2$ and $k = 8,9$. Hence $\Sp(1) \subset \Spin(4)$ as one of the factors. This must degenerate in $\SU(3)$ or $\Sp(2)$ and $l_- = 5$ or $7$. If $H$ contains no other simple factor, then $\dim G \le 4 \cdot(l_-+l_+) + \dim H \le 4 \cdot 8 + 6 = 38$, a contradiction. If $H$ contains another simple factor it can only be $\SU(3)$ and hence $l_+ = 11$ and $\dim H \le 2+3+8 = 13$ and therefore linear primitivity implies $\dim G \le 2 \cdot 18 + 13 = 49$, a contradiction. If $H^\prime = \SU(3)$, then its representation must degenerate in $\SU(4)$ and therefore $l_- = 7$. $H$ cannot contain another simple factor of rank at least two by Lemma \ref{obstructions_lem_simple} and also no simple rank one factor. Hence $l_+ = 7,1$. $\dim H \le 11$ and by linear primitvity $\dim G \le 2 \cdot (7+7) + \dim H\le 28+11 = 39$ or $\dim G \le 4 \cdot (1+7) +11 = 43$, a contradiction.
	%$K^-$ contains $\Spin (6)$ and $K^+$ contains $\SU(4)$. But then $K^-$ contains the central element of $\Spin(k)$ and hence $B_-$ is totally geodesic contradicting Lemma \ref{obstructions_lem_totgeod}, since $k \ge 11$. 
	This only leaves the case where $H_0$ is abelian. If both $l_\pm=1$ the action is not primitive since $\rk H \ge 3$. Hence $l_- =3$, but then $\dim G \le 4 \cdot 6 + 4 = 28$.  This finishes the proof.
\begin{center}
	$G=\SU(k)$, for $k \ge 5$
\end{center}
	Again we closely follow the proof in positive curvature of Proposition 11.2 in \cite{gwz}. In particular we just have to show that $H$ contains an $\SU(2)$ block since then the proof of \cite{gwz} applies. Let $\iota \in H_0$ be a non central involution. Then $N(\iota) = S(\U(k-2l)\U(2l))$. The proof of \cite{gwz} only fails in the case $(k,l) =(5,1)$. Then $\rk H = 2$. If $H = \SU(3)$, or $\Sp(2)$, then $H$ obviously contains a $2 \times 2$ block. Therefore assume that $H$ contains $\SU(2)$ which is diagonally in $S(\U(2)\U(2))\subset \SU(4)$. But then $\SU(2)$ has a four and a three dimensional representation. Since this $\SU(2)$ only commutes with a $2$ torus, this must degenerate in $\SU(2) \times \SU(2)$ and $\SU(3)$ or $\Sp(2)$. But $\SU(2)$ does not extend to $\SU(3)$ and the only extension to $\Sp(2)$ is effectively given as $\SO(5)/\SO(3)$. Thus we can assume $H = T^2$ and we  proceed as in \cite{gwz}, to rule out this case. Now we can follow the proof of \cite{gwz} proposition 11.2.
\end{proof}

%% file: documents/appendix.tex
\section{Appendix}

This Appendix is  a collection of information used to proof the results in this paper and it mostly coincides with Appendix \RN{2} of \cite{gwz}.  We use the notation $\rho_n$, $\mu_n$ and $\nu_n$ for the standard representations of $\SO(n)$, $\SU(n)$ and $\Sp(n)$. The spin representations of $\SO(n)$ are denoted by $\Delta_n$ and the half spin representations by $\Delta_n^\pm$. $\phi$ denotes a two dimensional irreducible representation of $S^1$ and  all other $N$-dimensional irreducible representations are denoted by $\psi_N$.

Table \ref{appendix_table_spherical} coincides with Table B in \cite{gwz}. It is a list of simple Lie groups together with their simple spherical subgroups and their embeddings obtained in \cite{w2} Propositions 7.2 - 7.4. All the embeddings are the classical inclusions, except for $\Spin(7) \subset \SO(8)$, which is the embedding by the spin representation. 
		\begin{table}[H]
		\caption{Simple spherical subgroups of simple groups}
			\begin{center}
		\begin{tabular}{||c|c|c||}
			\hline
			$G$ & $H$ & Inclusions\\
			\hline \hline
			$\SU(n)$ & $\SU(2)$ & $\SU(2) \subset \SU(n)$ given by $p \, \mu_2 \oplus q \, id$ \\
			\hline
			$\SU(n)$ & $\Sp(2)$ & $\Sp(2) \subset \SU(4) \subset \SU(n)$\\
			\hline 
			$\SU(n)$ & $\SU(k)$ & $k \times k$ block\\
			\hline \hline
			$\SO(n)$ & $\Sp(1)$ & $\Sp(1) \subset \SO(n)$ given by $p \, \nu_1\oplus q \, id$\\
			\hline
			$\SO(n)$ & $\SU(3)$ & $\SU(3) \subset \SO(6) \subset \SO(n)$\\
			\hline
			$\SO(n)$ & $\Sp(2)$ & $\Sp(2) \subset \SO(8) \subset \SO(n)$\\
			\hline
			$\SO(n)$ & $G_2$ 	& $G_2 \subset \SO(7) \subset \SO(8) \subset \SO(n)$\\
			\hline
			$\SO(n)$ & $\SU(4)$ & $\SU(4) \subset \SO(8) \subset \SO(n)$\\
			\hline
			$\SO(n)$ & $\Spin(7)$ & $\Spin(7) \subset  \SO(8) \subset \SO(n)$\\
			\hline
			$\SO(n)$ & $\SO(k)$ & $k \times k$ block\\
			\hline \hline
			$\Sp(n)$ & $\Sp(1)$ & $\Sp(1) = \lbrace \diag(q,q,\ldots,q,1,\ldots,1) \vert \, q \in S^3 \rbrace$\\
			\hline
			$\Sp(n)$ & $\Sp(k)$ & $k \times k$ block\\
			\hline\hline
			$G_2$ 	 & $\SU(3)$ & maximal subgroup \\
			\hline
			$F_4, E_6,$ & $\Spin(k)$ & $\Spin(k) \subset \Spin(9) \subset F_4 \subset E_6 \subset E_7 \subset E_8$\\
			$E_7, E_8$  & & $k = 5, \ldots,9$ standard embedding \\
			\hline
			$F_4,E_6, E_7, E_8$ 	 & $\SU(3)$ &	$\SU(3) \subset \SU(4) \subset \Spin(8) \subset F_4 \subset E_6 \subset E_7 \subset E_8$\\
			\hline
			$F_4,E_6, E_7, E_8$ 	 & $G_2$ &	$G_2 \subset  \Spin(7)  \subset  \Spin(8) \subset F_4 \subset E_6 \subset E_7 \subset E_8$\\
			\hline
		\end{tabular}
		\end{center}
		\label{appendix_table_spherical}
	\end{table}

	Table \ref{appendix_table_homsphere} contains all transitive actions on spheres together with their isotropy representations. This is especially important when applying  the Isotropy Lemma. We note, that all at least two dimensional irreducible subrepresentations of the isotropy representation $K/H$ act transitively on the unit sphere.

	\begin{table}[H]
	\caption{Transitive actions on Spheres}
	\begin{center}
		\begin{tabular}{||c|c|c|c||}
			\hline
			Dimension & K & H & Isotropy representation\\
			\hline \hline
			$n$ 	& $\SO(n+1)$ & $\SO(n)$ & $\rho_n$\\
			\hline
			$2n+1$	& $\SU(n+1)$ & $\SU(n)$ & $\mu_n \oplus id$\\
			\hline
			$2n+1$	& $\U(n+1)$ & $\U(n)$ & $\mu_n \oplus id$\\
			\hline
			$4n+3$	& $\Sp(n+1)$ & $\Sp(n)$ & $\nu_n \oplus 3 id$\\
			\hline
			$4n+3$	& $\Sp(1) \times \Sp(n+1)$ & $ \Delta \Sp(1)\Sp(n)$ & $\nu_1 \otimes \nu_n \oplus \rho_3 \otimes id$\\
			\hline
			$4n+3$	& $\U(1) \times \Sp(n+1)$ & $\U(1)\Sp(n)$ & $\phi \otimes \nu_n \oplus \phi \otimes id \oplus id$\\
			\hline
			$15$	& $\Spin(9)$ & $\Spin(7)$ & $\rho_7 \oplus \Delta_8$\\
			\hline
			$7$ 	& $\Spin(7)$ & $G_2$ & $\phi_7$\\
			\hline
			$6$	& $G_2$ & $\SU(3)$ & $\mu_3$\\
			\hline
		\end{tabular}
	\end{center}
	\label{appendix_table_homsphere}
\end{table}
	
		Table \ref{appendix_table_cohom_one_even}, which coincides with \cite{gwz} Table F, contains all essential cohomogneity one actions on compact rank one symmetric spaces in even dimensions, together with their extensions.
	
	\begin{table}[H]
		\caption{Essential actions on rank one symmetric spaces in even dimensions}
		\begin{center}
			\setlength\tabcolsep{1.5pt}
			\begin{tabular}[h]{||c|c|c|c|c|c|c||}
				\hline
				$M$ & $G$ &  $K^-$ & $K^+$ & $H$ & $(l_-,l_+)$& $W$\\
				\hline \hline
				$\bb S^4$ & $\SO(3)$ &  $S(\O(2)\O(1))$ & $S(\O(1)\O(2))$ & $\bb Z_2 \oplus \bb Z_2$ & $(1,1)$& $D_3$\\
				\hline
				$\bb S^{14}$ & $\Spin(7)$ &  $\Spin(6)$ & $G_2$ & $\SU(3)$ & $(1,6)$& $D_2$\\
				\hline \hline
				$\bb C \bb P^{k+1}$ & $\SO(k+2)$ &  $\SO(2)\SO(k)$ & $\O(k+1)$ & $\bb Z_2 \cdot \SO(k)$ & $(1,k)$& $D_2$\\
				\hline
				$\bb C \bb P^{2k+1}$ & $\SU(2)\SU(k+1)$ &  $\Delta \SU(2)S^1\SU(k-1)$ & $S^1\U(k)$ & $T^2\SU(k-1)$ & $(2,2k-1)$& $D_2$\\
				\hline
				$\bb C \bb P^{6}$ & $G_2$ &  $\U(2)$ & $\bb Z_2 \SU(3)$ & $\bb Z_2 \SU(2)$ & $(1,5)$& $D_2$\\
				\hline
				$\bb C \bb P^{7}$ & $\Spin(7)$ &  $S^1\SU(3)$ & $\bb Z_2 \Spin(6)$ & $\bb Z_2 \SU(3)$ & $(1,7)$& $D_2$\\
				\hline
				$\bb C \bb P^{9}$ & $\SU(5)$ &  $S^1\Sp(2)$ & $S(\U(2)\U(3))$ & $S^1 \SU(2)^2$ & $(4,5)$& $D_2$\\
				\hline
				$\bb C \bb P^{15}$ & $\Spin(10)$ &  $S^1\SU(5)$ & $S^1\Spin(7)$ & $S^1\SU(4)$ & $(9,6)$& $D_2$\\
				\hline\hline
				$\bb H \bb P^{k+1}$ & $\SU(k+2)$ &  $\SU(2)\SU(k)$ & $\U(k+1)$ & $S^1\SU(k)$ & $(2,2k+1)$& $D_2$\\
				& $S^1\SU(k+2)$ &  $\Delta S^1\SU(2)\SU(k)$ & $S^1\U(k+1)$ & $T^2\SU(k)$ & &\\
				\hline\hline
				$\text{Ca} \bb P^{2}$ & $\Sp(3)$ &  $\Sp(2)$ & $\Sp(1)\Sp(2)$ & $\Sp(1)^2$ & $(11,8)$& $D_2$\\
				& $S^1\Sp(3)$ &  $\Delta S^1\Sp(2)$ & $S^1\Sp(1)\Sp(2)$ & $S^1\Sp(1)^2$ & $(11,8)$&\\
				& $\Sp(1)\Sp(3)$ &  $\Delta \Sp(1)\Sp(2)$ & $\Sp(1)^2\Sp(2)$ & $\Sp(1)^3$ & $(11,8)$&\\
				\hline
			\end{tabular}
		\end{center}
		\label{appendix_table_cohom_one_even}
	\end{table}

	In Table \ref{appendix_table_cohom_one_odd_sphere} all essential cohomogeneity one actions on odd dimensional spheres are given together with their extensions. The information about this is contained in \cite{st} and \cite{gwz} Table E.
	
	\begin{table}[H]
	\caption{Essential actions and extensions on odd dimensional Spheres}
	\begin{center}
		\setlength\tabcolsep{1.5pt}
		\begin{tabular}{||c|c|c|c|c|c|c|c||}
			\hline
			$n$ & $G$ & $\chi$ & $K^-$ & $K^+$ & $H$ & $(l_-,l_+)$& $W$\\
			\hline \hline
			$8k+7$ & $\Sp(2)\Sp(k+1)$ & $\nu_2 \hat\otimes \nu_{k+1}$ & $\Delta \Sp(2) \Sp(k-1)$ & $\Sp(1)^2\Sp(k)$ & $\Sp(1)^2\Sp(k-1)$ & $(4,4k+1)$ & $D_4$\\
			\hline 
			$4k+7$ & $\SU(2)\SU(k+2)$ & $\mu_2 \hat\otimes \mu_{k+2}$ & $\Delta \SU(2) \SU(k)$ & $S^1\SU(k+1)$ & $S^1\SU(k)$ & $(2,2k+1)$ & $D_4$\\
			
			 & $\U(2)\SU(k+2)$ & $\mu_2 \hat\otimes \mu_{k+2}$ & $\Delta \U(2) \SU(k)$ & $T^2\SU(k+1)$ & $T^2\SU(k)$ &  & $D_4$\\
		
			$7$ & $\SU(2)\SU(2)$ & $\mu_2 \hat\otimes \mu_{2}$ & $\Delta \SU(2) $ & $T^2$ & $S^1$ & $(2,1)$ & $D_4$\\
			\hline 
			$2k+3$ & $\SO(2)\SO(k+2)$ & $\rho_2\hat \otimes \rho_{k+2}$ & $\Delta \SO(2) \SO(k)$ & $\bb Z_2\SO(k+1)$ & $\bb Z_2\SO(k)$ & $(1,k)$ & $D_4$\\
			\hline 
			$15$ & $\SO(2)\Spin(7)$ & $\rho_2 \hat\otimes \Delta_7$ & $\Delta \SO(2) \SU(3)$ & $\bb Z_2\Spin(6)$ & $\bb Z_2\SU(3)$ & $(1,7)$ & $D_4$\\
			\hline 
			$13$ & $\SO(2)G_2$ & $\rho_2 \hat\otimes \phi_7$ & $\Delta \SO(2) \SU(2)$ & $\bb Z_2\Spin(3)$ & $\bb Z_2\SU(2)$ & $(1,7)$ & $D_4$\\
			\hline 
			$7$ & $\SO(4)$ & $\nu_1 \hat\otimes \nu_3$ & $S(\O(2)\O(1))$ & $S(\O(1)\O(2))$ & $\bb Z_2 \oplus \bb Z_2$ & $(1,1)$ & $D_6$\\
			\hline 
			$15$ & $\Spin(8)$ & $\rho_8 \oplus \Delta_8^\pm$ & $\Spin(7)$ & $\Spin(7)$ & $G_2$ & $(7,7)$ & $D_2$\\
			\hline 
			$13$ & $\SU(4)$ & $\mu_4 \oplus \rho_6$ & $\SU(3)$ & $\Sp(2)$ & $\SU(2)$ & $(5,7)$ & $D_2$\\
		
			 & $\U(4)$ & $\mu_4 \oplus \rho_6$ & $S^1\SU(3)$ & $S^1\Sp(2)$ & $S^1\SU(2)$ & $(5,7)$ & \\
			 \hline 
			 $19$ & $\SU(5)$ & $\Lambda^2\mu_5$ & $\Sp(2)$ & $\SU(2)\SU(3)$ & $\SU(2)^2$ & $(4,5)$ & $D_4$\\
			 
			 & $\U(5)$ & $\Lambda^2\mu_5$ & $S^1\Sp(2)$ & $S^1\SU(2)\SU(3)$ & $S^1\SU(2)^2$ & $(4,5)$ &  \\
			 \hline 
			 $31$ & $\Spin(10)$ & $\Delta_{10}^\pm$ & $\SU(5)$ & $\Spin(7)$ & $\SU(4)$ & $(9,6)$ & $D_4$\\
			  & $S^1\Spin(10)$ & $\Delta_{10}^\pm$ & $\Delta S^1\SU(5)$ & $S^1\Spin(7)$ & $S^1\SU(4)$ & $(9,6)$ &\\
			  \hline \hline
			  $7$ & $\SU(3)$ & $\ad$ & $S(\U(2)\U(1))$ & $S(\U(1)\U(2))$ & $T^2$ & $(2,2)$ & $D_3$\\
			  \hline 
			  $9$ & $\SO(5)$ & $\ad$ & $\U(2)$ & $\Spin(7)$ & $\SO(3)\SO(2)$ & $(2,2)$ & $D_4$\\
			  \hline 
			  $13$ & $G_2$ & $\ad$ & $\U(2)$ & $\U(2)$ & $T^2$ & $(2,2)$ & $D_6$\\
			  \hline 
			  $13$ & $\Sp(3)$ & $\psi_{14}$ & $\Sp(2)\Sp(1)$ & $\Sp(1)\Sp(2)$ & $\Sp(1)^3$ & $(4,4)$ & $D_3$\\
			  \hline 
			  $25$ & $F_4$ & $\psi_{26}$ & $\Spin(9)$ & $\Spin(9)$ & $\Spin(8)$ & $(8,8)$ & $D_3$\\
			  \hline
			
		\end{tabular}
	\end{center}
	\label{appendix_table_cohom_one_odd_sphere}
\end{table}

 Table \ref{appendix_table_odd_cohom_one} contains the group diagrams of examples and candidates for (quasi-)positively curved cohomogeneity one manifolds in odd dimensions. Note that most examples come from actions by $S^3 \times S^3$. For an imaginary unit quanternion $x \in S^3$, the induced $S^1$ subgroup of $S^3\times S^3$ with slopes $(p,q)$ is denoted by $C^x_{(p,q)}= \lbrace (e^{px\theta},e^{qx\theta})\vert \,\theta \in \bb R \rbrace$. More detailed information about these manifolds is contained in Section 4 of \cite{gwz}. Table \ref{appendix_table_odd_cohom_one} immediately gives some isomorphisms between the manifolds contained there: $P_1 = \bb S^7$, $Q_1 = W^7_{(2)}$, $E_1 = W^7_{(1)}$ and $B_1^{13} = B^{13}$. Note that by  \cite{gwz} Lemma 4.2 the actions on $B^7$, $P_k$, $Q_k$ and $R$ do not admit any extensions. The manifolds $E_p^7$ and $B_p^{13}$ admit an extension by $S^1$. Note furthermore, that the new examples occurring compared to the positively curved case are the Eschenburg Space $E_0^7$ and the Bazaikin Space $B_0^{13}$. Both admit metrics with positive sectional curvature almost everywhere (cf. \cite{k}). Namely $E_0\backslash B_+$ and $B_0 \backslash B_+$ are positively curved with this metric. The manifold $R$ does not admit a metric with quasipositive curvature by \cite{vz}.

	\begin{table}[H]
	\caption{Cohomogeneity one actions in odd dimensions}
	\begin{center}
		\setlength\tabcolsep{1.5pt}
		\begin{tabular}{||c|c|c|c|c|c|c|c||}
			\hline
			$M^n$ & $G$ &  $K^-$ & $K^+$ & $H$ & $\bar H$ & $(l_-,l_+)$& $W$\\
			\hline \hline
			$\bb S^7$ & $S^3 \times S^3$ & $C^i_{(1,1)} H$ & $C^j_{(1,3)} H$ & $Q$ & $\bb Z_2 \oplus \bb Z_2$ & $(1,1)$ & $D_6$\\
			\hline 
			$B^7$ & $S^3 \times S^3$ & $C^i_{(3,1)} H$ & $C^j_{(1,3)} H$ & $Q$ & $\bb Z_2 \oplus \bb Z_2$ & $(1,1)$ & $D_3$\\
			\hline 
			$W_{(1)}^7$ & $S^3 \times S^3$ & $\Delta S^3 \cdot  H$ & $C^i_{(1,2)} $ & $\bb Z_2$ & $1$ & $(3,1)$ & $D_2$\\
			\hline 
			$W_{(2)}^7$ & $S^3 \times S^3$ & $C^i_{(1,1)} H$ & $C^j_{(1,2)} H$ & $\bb Z_4 \oplus \bb Z_2$ & $\bb Z_2$ & $(1,1)$ & $D_4$\\
			\hline 
			$B^{13}$ & $\SU(4)$ & $\Sp(2)\cdot \bb Z_2$ & $\SU(2)\cdot S^1_{1,2}$ & $\SU(2)\cdot\bb Z_2$ & $\SU(2)\cdot\bb Z_2$ & $(7,1)$ & $D_2$\\
			\hline \hline
			$E_p^7, \, p\ge 0$ & $S^3 \times S^3$ & $\Delta S^3 \cdot  H$ & $C^i_{(p,p+1)} $ & $\bb Z_2$ & $1$ & $(3,1)$ & $D_2$\\
			\hline 
			$B_p^{13},\, p \ge 0$ & $\SU(4)$ & $\Sp(2)\cdot \bb Z_2$ & $\SU(2)\cdot S^1_{p,p+1}$& $\SU(2)\cdot\bb Z_2$ & $\SU(2)\cdot\bb Z_2$ & $(7,1)$ & $D_2$\\
			\hline \hline
			$\bb P_k, \, k\ge 1$ & $S^3 \times S^3$ & $C^i_{(1,1)} H$ & $C^j_{(k,k+2)} H$ & $Q$ & $\bb Z_2 \oplus \bb Z_2$ & $(1,1)$ & $D_3$ or $D_6$\\
			\hline
			$Q_k, \, k\ge 1$ & $S^3 \times S^3$ & $C^i_{(1,1)} H$ & $C^j_{(k,k+1)} H$ & $\bb Z_4 \oplus \bb Z_2$ & $\bb Z_2$ & $(1,1)$ & $D_4$\\
			\hline 
			$R$ & $S^3 \times S^3$ & $C^i_{(3,1)} H$ & $C^j_{(1,2)} H$ & $\bb Z_4 \oplus \bb Z_2$ & $\bb Z_2$ & $(1,1)$ & $D_4$\\
			\hline
		\end{tabular}			
	\end{center}
	\label{appendix_table_odd_cohom_one}
\end{table}